\documentclass{siamart220329}
\UseRawInputEncoding
\usepackage{algorithm}
\usepackage{algpseudocode}

\makeatletter
\def\algbackskip{\hskip-\ALG@thistlm}
\makeatother

\usepackage{comment}

\usepackage{amsmath}
\usepackage{mathtools}
\usepackage{amssymb}
\usepackage{mathrsfs}
\usepackage{bm}
\usepackage{algpseudocode}
\usepackage{upgreek}
\usepackage{subcaption}

\usepackage{graphicx}
\usepackage{epstopdf}

\usepackage{stmaryrd}
\usepackage[shortlabels]{enumitem}

\newcommand{\grad}{{\mathbf{grad}}}

\newcommand{\cc}{{\normalfont{\text{c}}}}
\newcommand{\loc}{{\normalfont{\text{loc}}}}

\newcommand{\br}{{\normalfont{\boldsymbol{r}}}}

\newcommand{\dd}{{\normalfont{\text{d}}}}

\usepackage{xcolor}
\definecolor{ao}{rgb}{0.0, 0.5, 0.0}

\usepackage{bbm}

\makeatletter
\newcommand*{\encircled}[1]{\relax\ifmmode\mathpalette\@encircled@math{#1}\else\@encircled{#1}\fi}
\newcommand*{\@encircled@math}[2]{\@encircled{$\m@th#1#2$}}
\newcommand*{\@encircled}[1]{%
  \tikz[baseline,anchor=base]{\node[draw,circle,outer sep=0pt,inner sep=.2ex] {#1};}}
\makeatother

\DeclareSymbolFont{extraup}{U}{zavm}{m}{n}
\DeclareMathSymbol{\varheart}{\mathalpha}{extraup}{86}
\DeclareMathSymbol{\vardiamond}{\mathalpha}{extraup}{87}

\usepackage{tikz}

\usepackage{pifont}

\usepackage{fancyhdr}
\usepackage[us,12hr]{datetime} 
\fancypagestyle{plain}{
\fancyhf{}
\lhead{\thepage}
}
\usepackage{color}
\newcommand{\half}{\frac{1}{2}}

\newcommand{\norm}[1]{\left \lVert #1 \right \rVert}
\newcommand{\snorm}[1]{\left \lvert #1 \right \rvert}

\newcommand{\IR}{\mathbb{R}}
\newcommand{\IC}{\mathbb{C}}

\newcommand{\IN}{\mathbb{N}}

\newcommand{{\diff}}[1]{{\normalfont{\text{d}} #1}}
\newcommand{{\Rf}}[1]{{\Re\left\{#1\right \}}}

\newcommand{\OV}{\mathsf{V}}
\newcommand{\OK}{\mathsf{K}}

\newcommand{{\D}}{\normalfont{\text{D}}}
\newcommand{{\G}}{\normalfont{\text{G}}}
\newcommand{{\U}}{\normalfont{\text{U}}}
\newcommand{{\I}}{\normalfont{\text{I}}}

\newcommand{{\A}}{\normalfont{\text{A}}}

\newcommand{\isdef}{\mathrel{\mathrel{\mathop:}=}}

\newcommand{{\bx}}{{\bf x}}
\newcommand{{\by}}{{\bf y}}
\newcommand{{\bz}}{{\bf z}}
\newcommand{\bxref}{\hat{\bx}}

\newcommand{{\bvarphi}}{{\boldsymbol{\varphi}}}

\newcommand{\OA}{\mathsf{A}}
\newcommand{\OT}{\mathsf{T}}
\newcommand{{\N}}{\normalfont{\text{N}}}
\newcommand{{\y}}{{\boldsymbol{{y}}}}
\newcommand{{\z}}{{\boldsymbol{z}}}

\newcommand{{\bc}}{{\bf c}}
\newcommand{{\bd}}{{\bf d}}

\usepackage{bm}

\newcommand{\dual}[2]{\left \langle #1,#2 \right \rangle}
\newcommand{\dotp}[2]{\left ( #1,#2 \right )}

\newtheorem{assumption}[theorem]{Assumption}

\newtheorem{example}[theorem]{Example}
\newtheorem{problem}[theorem]{Problem}

\theoremstyle{plain}
\newtheorem{remark}[theorem]{Remark}

\newcommand{\JD}[1]{{\color{blue} #1}}

\numberwithin{equation}{section}

\begin{document}

\title{Parametric Shape Holomorphy of Boundary Integral Operators with Applications}



\author{
J\"{u}rgen D\"{o}lz\thanks{Institute for Numerical Simulation, University of Bonn, Friedrich-Hirzebruch-Allee 7, 53115 Bonn, Germany (\email{doelz@ins.uni-bonn.de}).}
\and 
Fernando Henr\'iquez\thanks{Chair of Computational Mathematics and Simulation Science, \'Ecole Polytechnique F\'ed\'erale de Lausanne, 1015 Lausanne, Switzerland (\email{fernando.henriquez@epfl.ch})
}}


\maketitle

\begin{abstract}
We consider a family of boundary integral operators supported on 
a collection of parametrically defined bounded Lipschitz boundaries.
Consequently, the boundary integral operators themselves also depend on the parametric
variables, thus leading to a parameter-to-operator map.
The main result of this article is to
establish the analytic or holomorphic dependence
of said boundary integral operators upon the parametric variables, i.e.,
of the parameter-to-operator map. As a direct consequence we also
establish holomorphic dependence of solutions to boundary integral
equations, i.e.,~holomorphy of the parameter-to-solution map.
To this end, we construct a holomorphic extension
to complex-valued boundary deformations and investigate the
\emph{complex} Fr\'echet differentiability of
boundary integral operators
with respect to each parametric variable.
The established parametric holomorphy results
have been identified as a key property to overcome the 
so-called curse of dimensionality in the approximation of
parametric maps with distributed, high-dimensional inputs.

To demonstrate the applicability of the derived results, we consider
as a concrete example the sound-soft Helmholtz
acoustic scattering problem and its frequency-robust boundary integral
formulations. For this particular application, we explore the consequences
of our results in reduced order modelling, Bayesian shape inversion,
and the construction of efficient surrogates using artificial neural networks.

\end{abstract}

\section{Introduction}
In this article, we focus on shape-parametric boundary
integral operators (BIOs) of the form
\begin{equation}\label{eq:A_bio}
	\OA_\y\colon 
	L^2(\Gamma_\y)
	\to 
	L^2(\Gamma_\y),
\end{equation}
where $\Gamma_\y\subset\IR^3$ is assumed to be a
bounded Lipschitz boundary depending on the parametric input
$\y\in \mathbb{U} \coloneqq[-1,1]^{\IN}$. 
The BIO in \cref{eq:A_bio}
is assumed to be of the form
\begin{equation}\label{eq:parametric_BIO}
	(\OA_\y\varphi)(\bx)
	\isdef
	\int\limits_{\Gamma_\y}{\sf a}(\by,\bx-\by)\varphi(\by)\mathrm{ds}_\by,
	\quad
	\bx
	\in
	\Gamma_\y,
\end{equation}
for any $\varphi\in L^2(\Gamma_\y)$, where the kernel function
${\sf a}: \Gamma_\y \times \IR^3 \backslash \{{\bf 0}\} \rightarrow \IC$ is bounded
in the first argument, and possibly \emph{singular} when $\bx=\by$
in the second one. A typical example application is the $L^2(\Gamma_\y)$-based 
frequency-stable formulation for the acoustic scattering by random surfaces, 
wherein the boundary's randomness is encoded through a
Karhunen-Lo\`eve-type expansion of a given random field.
This is, for example, the setting considered in \cite{DHJM2022,BG14,GP18,BG12}.

In the following, we are interested in the regularity of the \emph{parameter-to-operator}
map
\begin{equation}\label{eq:param_to_operator}
	\mathbb{U}
	\ni
	\y
	\mapsto\OA_\y
	\in 
	\mathscr{L}
	\left(
		L^2(\Gamma_\y)
        ,
        L^2(\Gamma_\y)
	\right)
\end{equation}
and, assuming  $\OA_\y $ to be boundedly invertible
for each $\y \in \mathbb{U}$, of the \emph{parameter-to-solution}
map
\begin{equation}\label{eq:param_to_solution}
	\mathbb{U}
	\ni
	\y
	\mapsto 
	u_\y
	\isdef
	\OA_\y^{-1} f_\y
	\in L^2(\Gamma_\y),
\end{equation}
where $f_\y\in L^2(\Gamma_\y)$ for each $\y \in \mathbb{U}$. For these maps,
constructing efficient surrogates is a central task in
computational uncertainty quantification and reduced order modelling of parametric partial
differential and boundary integral equations.
However, it is well-known that the so-called curse of dimensionality
in the parameter space renders most na{\"i}ve approximation methods
computationally unfeasible. To overcome this challenge and 
effectively design dimension-robust algorithms for the approximation of the maps in 
\cref{eq:param_to_operator} and \cref{eq:param_to_solution} one requires a better understanding of the \emph{parametric regularity} of the BIOs upon the shape-parametric representations.

In \cite{CCS15}, it is proposed to construct a holomorphic
extension over a tensor product of ellipses in the complex domain
as a key step to finally break the curse of dimensionality in the
sparse polynomial approximation of high-dimensional parametric problems.
Indeed, by varying the size of these ellipses on each parameter, 
i.e. by exploiting the \emph{anisotropic} parameter dependence, this
\emph{$(\boldsymbol{b},p,\varepsilon)$-holomorphy} allows to state provable
convergence rates that do not depend on the input's parametric dimension.
From a computational point of view, $(\boldsymbol{b},p,\varepsilon)$-holomorphy 
constitutes a stepping stone in the analysis and implementation
of a plethora of approximation
techniques such as sparse grid interpolation and quadrature \cite{ZS17,SS13,NTW2008,HHPS2018},
higher-order Quasi-Monte Carlo \cite{DKL14,DLC16,DGLS19,DGLS17}
integration (HoQMC), construction of artificial neural network (ANN)
surrogates \cite{SZ19,HSZ20,OSZ22,HOS22,ABD22}, 
and to establish sparsity of parametric posterior densities 
in Bayesian inverse problems \cite{CS16_2,SS13,SS16}.

The main objective of this work is to derive shape-parametric regularity
results for BIOs of the form \cref{eq:A_bio}.
In addition, we aim to demonstrate the usefulness of this property by 
studying its implications in reduced order modelling, expression
rates of ANN surrogates, and Bayesian shape
inversion.

\subsection{Previous Work}
The analytic dependence of the parameter-to-operator and
parameter-to-solution map is the primary object of study in 
many recent works tackling parametric shape deformations.
In the context of volume-based
variational formulations, this property has been 
established for subsurface flows \cite{CNT2016,HPS2016,HS2022},
time-harmonic
electromagnetic wave scattering by perfectly conducting
and dielectric obstacles \cite{JSZ16,AJS19}, stationary Stokes
and Navier-Stokes equation \cite{CSZ18}, volume formulations 
for acoustic wave scattering by a penetrable obstacle
in a low-frequency regime \cite{HSSS15}, and high-frequency
Helmholtz formulations \cite{SW23,GKS21}. 

As pointed out previously, to establish the sought parametric 
holomorphy property one verifies \emph{complex} Fr\'echet differentiability
of the parameter-to-operator or parameter-to-solution map. 
This approach is strongly connected to the notion of shape differentiability 
of the domain/boundary-to-solution map. 
In \cite{Pot94,Pot96} and \cite{Charam95}, the shape differentiability in the Fr\'echet sense 
of BIOs for acoustic and elastic problems is studied.
This result is extended in \cite{Pot96_3,LC121,LC122} to the BIOs that arise in electromagnetic wave scattering.
The same approach used to obtain these results is also employed in \cite{LC121,LC122,YL16} to establish the infinite differentiability of a collection of BIOs characterized by a class of pseudo-homogeneous kernels.
In \cite{kress1999far,LeLouer12} shape differentiability the elastic obstacle scattering is analysed. The technique presented in \cite{kress1999far} is subsequently extended 
in \cite{haddar2004frechet} to 
acoustic and electromagnetic scattering problems
equipped impedance boundary conditions.
However, regardless of the specific setting considered on the previous works, none of them imply an analytic or 
holomorphic dependance upon shape perturbations. Another drawback of the previously discussed results
is that none of them are capable of systematically account for polygons/polyhedra.

Recently, in \cite{HS21,henriquez2021shape} the holomorphic dependence
of the Calder\'on projector upon boundaries of class $\mathscr{C}^2$ has been established.
Therein, the two-dimensional analysis for both the Laplace and Helmholtz BIOs
is thoroughly analyzed. In \cite{PHJ23} this work has been extended to address
the case of multiple open arcs in two-dimensions. 
In \cite{dalla2022multi,DLM22,dalla2013family}, based on \cite{DR04,LR08}, an
analytic shape differentiability result is established,
and this approach is extended in \cite{DL23} 
to the layer potentials for the heat equation.  
Yet useful in the three dimensional case, these works
still rule out the possibility of polygonal/polyhedral boundaries in the analysis. 

In summary, current results in the literature are based on various smoothness assumptions on the considered boundaries. These smoothness assumptions are a used to deal with the singularities in the kernel function of the BIOs.
Lipschitz boundaries, let alone polygonal/polyhedral ones, are not covered by these results.

\subsection{Contributions}
We establish parametric holomorphy by verifying 
\emph{complex} Fr\'echet differentiability of suitable complex extensions
of the parameter-to-operator map \cref{eq:param_to_operator} and
parameter-to-solution map \cref{eq:param_to_solution} for BIOs on boundaries of Lipschitz domains in three dimensions. This is a substantial relaxation compared to previous results. The considered maps are viewed
as elements of the complex Banach space of bounded linear operators or functions,
with respect to a collection of affine-parametric bounded Lipschitz boundaries.

The contribution of this article is two-fold:
\begin{enumerate}
	\item 
	By considering a collection of affine shape-parametric
	boundaries and suitable assumptions on the integral kernel $\mathsf{a}(\cdot,\cdot)$,
	we prove that the parameter-to-operator map \cref{eq:param_to_operator}
	is $(\boldsymbol{b},p,\varepsilon)$-holomorphic in the sense of
	\cite[Definition 2.1]{CCS15} in a $L^2$-based setting.
	This is achieved by taking the limit of a family of BIOs where the singularity is suitably regularized.
	We conclude that the parameter-to-solution 
	map \cref{eq:param_to_solution} is also $(\boldsymbol{b},p,\varepsilon)$-holomorphic.
	\item 
	We demonstrate the practical applicability of our
	$(\boldsymbol{b},p,\varepsilon)$-holomorphy results
	by exemplarily applying our result to frequency-stable
	integral equations used in the boundary reduction
	of acoustic wave propagation in unbounded domains.
	This, in turn, serves a theoretical foundation for the analysis
	of the computational technique introduced in \cite{DHJM2022}.
\end{enumerate}

We conclude this work by demonstrating the theoretical usability of our
$(\boldsymbol{b},p,\varepsilon)$-holomorphy results
to obtain \emph{best $n$-term} approximation rates for \cref{eq:param_to_operator} and \cref{eq:param_to_solution}.
In addition, we explore the consequence of our result in reduced 
order modelling, Bayesian shape inversion, 
and in the construction of efficient surrogates using ANNs.

\subsection{Outline}
The remainder of the article is structured as follows. First, in \cref{sec:preliminaries}, we recall some preliminary facts on holomorphy in Banach spaces, parametrized surfaces, and boundary integral operators. $(\boldsymbol{b},p,\varepsilon)$-holomorphy of the boundary integral operators is established in \cref{sec:BIO}, whereas the illustration to boundary integral equations from acoustic scattering is done in \cref{sec:application}. \Cref{sec:HDA} demonstrates how the established $(\boldsymbol{b},p,\varepsilon)$-holomorphy leads to various \emph{best-$n$-term} approximation rates and is followed by concluding remarks in \cref{sec:concl}.

\section{Preliminaries}
\label{sec:preliminaries}
Let $\D \subset \IR^{3}$ be a bounded Lispchitz domain \cite[Definition 3.3.1]{HW08}. 
We denote by $L^p(\D)$, $p\in [1,\infty)$, the set of $p$-integrable functions over $\D$
and by $H^s(\D)$, $s\geq0$, standard Sobolev spaces in $\D$.
Let $\mathscr{C}^{0,1}(\partial \D)$ be the space of Lipschitz continuous functions on $\partial \D$ and, assuming sufficient regularity of $\partial \D$, $\mathscr{C}^{1,1}(\partial \D)$ be the space of continuously differentiable functions with Lipschitz continuous derivative on $\partial\D$.
In addition, let $H^s(\partial \D)$, for $s\in[0,1]$, {be the Sobolev}
space of traces on $\partial \D$ (\cite[Section 2.4]{SS10}, \cite[Section 4.2]{HW08}).
As is customary, we identify $H^0(\partial \D)$ with $L^2({\partial \D})$ and, 
for $s\in [0,1]$, $H^{-s}(\partial \D)$ with the dual space of $H^s(\partial \D)$
in the $L^2(\partial \D)$ duality pairing.

For (complex) Banach spaces $X$ and $Y$,
we denote by $\mathscr{L}(X,Y)$ the space of bounded linear operators 
from $X$ into $Y$ and by $\mathscr{L}_{\text{iso}}(X,Y)$ the (open)
subset of isomorphisms, i.e.~bounded linear operators with a bounded inverse. 
Recall that $\mathscr{L}(X,Y)$
is a (complex) Banach space equipped with the standard operator norm \cite[Theorem III.2]{RS80_Vol1}.

Finally, we equip $\mathbb{U}=[-1,1]^{\IN}$ with the product topology.
According to Tychonoff's theorem, this renders $\mathbb{U}$ compact with this topology.

\subsection{Holomorphy in Banach Spaces}
\label{sec:hol_banach}
For the convenience of the reader, we recall the basic facts of
holomorphy in Banach spaces.
To this end, let $X$ and $Y$ be \emph{complex} Banach spaces
equipped with the norms $\norm{\cdot}_{X}$ and $\norm{\cdot}_{Y}$,
respectively. 
\begin{definition}[{\cite[Definition 13.1]{Muj86}}]\label{def:complex_diff}
Let $U$ be an open, nonempty subset of $X$. 
A map $\mathcal{M}:U \rightarrow Y$ is said to be
\emph{complex Fr\'echet differentiable}
if for each $r\in U$ there exists a map
$(\frac{d}{d r} \mathcal{M})(r,\cdot)\in \mathscr{L}(X,Y)$ such that
\begin{align}
	\norm{
	{\mathcal{M}(r+\xi)}
	-
	\mathcal{M}(r)
	- 
	\left(\frac{d}{dr} \mathcal{M}\right)(r,\xi)}_{Y}
	=
	o\left(\norm{\xi}_{X}\right).
\end{align}
We say that $\left(\frac{d}{dr} \mathcal{M}\right)(r,\xi)$
is the \emph{Fr\'echet derivative} of the map
$\mathcal{M}: U \rightarrow Y$ at $r\in U$ in the
direction $\xi\in X$.
\end{definition}

The following results are extensions of some classical results from
complex analysis to Banach spaces.
\begin{theorem}[{\cite[Theorem 14.7]{Muj86}}]\label{thm:holomorphic}
Let $U$ be an open, nonempty subset of $E$. 
For the map $\mathcal{M}:U \rightarrow F$ the 
following conditions are equivalent:
\begin{enumerate}
\item $\mathcal{M}$ is complex Fr\'echet differentiable.
\item $\mathcal{M}$ is infinitely complex Fr\'echet differentiable.
\item $\mathcal{M}$ is holomorphic.
\end{enumerate}
\end{theorem}

\begin{remark}
In view of \cref{thm:holomorphic} it sufficies to show 
\emph{complex} Fr\'echet differentiability to establish
a holomorphic dependence. 
\end{remark}


\begin{lemma}[{\cite[Theorem 3.1.5, item c)]{herve2011analyticity}}]\label{lmm:herve}
Let $U$ be an open, nonempty subset of $X$, and 
let $\left(\mathcal{M}_n\right)_{n\in\mathbb{N}}$ 
be a sequence of holomorphic maps from $U \subset X$ to $Y$.
Assume that $\mathcal{M}_n$ converges uniformly to $\mathcal{M}$
in $U$.
Then $\mathcal{M}: U \subset X \rightarrow Y$ is holomorphic.
\end{lemma}

\begin{proposition}[{\cite[Proposition 4.2]{AJS19},\cite[Proposition 4.20]{HS21}}]\label{prop:hol_maps_banach_spaces}
Let $X$, $Y$ be complex Banach spaces and denote by $\mathscr{L}_{\normalfont\text{iso}}(X,Y)$
the (open) subset of operators with bounded inverse in $\mathscr{L}(X,Y)$.
\begin{enumerate}
	\item\label{prop:hol_maps_banach_spaces1}  Let $\mathsf{M}\in \mathscr{L}_{\normalfont\text{iso}}(X,Y)$. Then
	with
	\begin{align}\label{eq:set_CM}
		\mathcal{C}_\mathsf{M}
		\coloneqq
		\left \{
			\mathsf{T} \in \mathscr{L}(X,Y): \; 
			\norm{\mathsf{M}-\mathsf{T}}_{\mathscr{L}(X,Y)}
			<
			\norm{\mathsf{M}^{-1}}^{-1}_{\mathscr{L}(Y,X)}
		\right \}
	\end{align}
	$\mathcal{C}_\mathsf{M} \subset \mathscr{L}_{\normalfont\text{iso}}(X,Y)$
	and for all $\mathsf{T} \in \mathcal{C}_\mathsf{M}$ it holds
	\begin{align}
		\norm{\mathsf{T}^{-1}}_{\mathscr{L}(Y,X)}
		\leq
		\frac{
			\norm{\mathsf{M}^{-1}}_{\mathscr{L}(Y,X)}
		}{
		 	1-\norm{\mathsf{M}-\mathsf{T}}_{\mathscr{L}(X,Y)}
			\norm{\mathsf{M}^{-1}}_{\mathscr{L}(Y,X)}
		}.
	\end{align}
	\item\label{prop:hol_maps_banach_spaces2}
	The {\normalfont inversion} map
	\begin{align}\label{eq:def_inversion_map}
		\text{\normalfont{inv}}: \;
		\mathscr{L}_{\text{\normalfont{iso}}}(X,Y)
		\mapsto
		\mathscr{L}_{\text{\normalfont{iso}}}(Y,X):
		\mathsf{M}
		\mapsto 
		\mathsf{M}^{-1}
	\end{align}
	is holomorphic.
	\item\label{prop:hol_maps_banach_spaces3}
	The {\normalfont application} map
	\begin{align}\label{eq:def_application_map}
		\text{\normalfont{app}}: \;
		\left(
			\mathscr{L}(X,Y)
			,
			X
		\right)
		\rightarrow 
		Y: \; 
		(\mathsf{M},\mathsf{g})
		\mapsto 
		\mathsf{M}\, \mathsf{g}
	\end{align}
is holomorphic.
\end{enumerate}
\end{proposition}

\subsection{Parametric Holomorphy}
We proceed to define precisely the notion of parametric holomorphy 
to be used throughout this work. To this end, for $s>1$ we define 
the so-called Bernstein ellipse as
\begin{align}
	\mathcal{E}_s
	\coloneqq 
	\left\{
		\frac{z+z^{-1}}{2}: \; 1\leq \snorm{z}\leq s
	\right \} 
	\subset \IC.
\end{align}
In other words, this domain corresponds to the part of the complex domain which is
enclosed by the ellipse with foci at $z=\pm 1$ and semi-axes of length 
$a\coloneqq  (s+s^{-1})/2$ and $b \coloneqq  (s-s^{-1})/2$.
In addition, we define the tensorized poly-ellipse
\begin{align}
	\mathcal{E}_{\boldsymbol{\rho}} 
	\coloneqq  
	\bigotimes_{j\geq1} 
	\mathcal{E}_{\rho_j} \subset \IC^{\mathbb{N}},
\end{align}
where $\boldsymbol\rho \coloneqq  (\rho_j)_{j\geq1}$
is such that $\rho_j>1$, for $j\in \mathbb{N}$.
We adopt the convention $\mathcal{E}_1\coloneqq [-1,1]$
to include the limit case $\rho_j=1$.

The next definition, originally introduced in \cite{CCS15}, 
constitutes the precise notion of parametric holomorphy to be used throughout
this work.

\begin{definition}[{\cite[Definition 2.1]{CCS15}}]\label{def:bpe_holomorphy}
Let $X$ be a complex Banach space equipped with the norm $\norm{\cdot}_{X}$. 
For $\varepsilon>0$ and $p\in(0,1)$, we say that the map 
\begin{equation}
	\mathbb{U} 
	\ni 
	\y
	\mapsto 
	u_\y 
	\in 
	X
\end{equation}
is $(\boldsymbol{b},p,\varepsilon)$-holomorphic if and only if
\begin{enumerate}
	\item\label{def:bpe_holomorphy1}
	The map $\mathbb{U} \ni {\y} \mapsto u_\y \in X$ is uniformly bounded, i.e.~
	\begin{align}
		\sup_{\y\in \mathbb{U}}
		\norm{u_\y}_{X}
		\leq
		C_0,
	\end{align}
	for some finite constant $C_0>0$.
	\item\label{def:bpe_holomorphy2}
	There exists a positive sequence $\boldsymbol{b}\coloneqq (b_j)_{j\geq 1} \in \ell^p(\mathbb{N})$ 
	and a constant $C_\varepsilon>0$ such that for any sequence 
	$\boldsymbol\rho\coloneqq (\rho_j)_{j\geq1}$ 
	of numbers strictly larger than one that is $(\boldsymbol{b},\varepsilon)$-admissible,
	i.e.~satisfying
	\begin{align}
	\label{eq:admissible_polyradius}	
		\sum_{j\geq 1}(\rho_j-1) b_j 
		\leq 
		\varepsilon,
	\end{align}
	the map $\y \mapsto u_{\y}$ admits a complex
	extension $\z \mapsto u_{\z}$ 
	that is holomorphic with respect to each
	variable $z_j$ on a set of the form 
	\begin{align}
		\mathcal{O}_{\boldsymbol\rho} 
		\coloneqq  
		\displaystyle{\bigotimes_{j\geq 1}} \, \mathcal{O}_{\rho_j},
	\end{align}
	where
    \begin{equation}
    \mathcal{O}_{\rho_j}
	=
	\{z\in\IC\colon\operatorname{dist}(z,[-1,1])<\rho_j-1\}.
    \end{equation}
	\item\label{def:bpe_holomorphy3}
	This extension is bounded on $\mathcal{E}_{\boldsymbol\rho}$ according to
	\begin{align}
	\label{eq:bpe_hol_bound_epsilon}
		\sup_{\z\in \mathcal{E}_{\boldsymbol{\rho}}} \norm{u_\z}_{X}  
		\leq 
		C_\varepsilon.
	\end{align}
\end{enumerate}
\end{definition}

According to \cite[Lemma 4.4]{CCS15}, for any $s>1$,
it holds $\mathcal{E}_s \subset \mathcal{O}_s$.
This fact will be used extensively throughout this work.

\begin{remark}
Let $\boldsymbol\rho\coloneqq (\rho_j)_{j\geq1}$ 
be any sequence of numbers strictly larger than
one that is $(\boldsymbol{b},\varepsilon)$-admissible.
In view of \cref{def:complex_diff} and \cref{thm:holomorphic},
\cref{def:bpe_holomorphy2}
in \cref{def:bpe_holomorphy} is equivalent to stating that for
each $j\in \IN$ and $\z \in \mathcal{O}_{\boldsymbol\rho}$
there exists $\partial_{z_j} u_{\z} \in X$, referred to as the
derivative in the direction $z_j$, such that
\begin{equation}
	\norm{
		u_{\z + h \boldsymbol{e}_j}
		-
		u_{\z}
		-
		\partial_{z_j}
		u_{\z}
		h
	}_{
		X
	}
	=
	o
	\left(
		\snorm{h}
	\right),
	\quad
	\text{as }
	\snorm{h}
	\rightarrow 
	0^+,
	\;
	h\in \IC,
\end{equation}
where $\boldsymbol{e}_j$ signifies the
$j$-th unit vector.
\end{remark}

%
%
%
\subsection{Affine-Parametric Boundary Transformations}
\label{sec:affine_parametric}
Let $\hat{\Gamma}\subset\IR^3$, in the following referred to 
as the reference boundary, be a bounded
Lipschitz boundary. To define parametrized families of bounded Lipschitz boundaries
in $\IR^3$, we consider affine-parametric vector fields of
Karhunen-Lo\`eve type. 
To this end, we set $\br_\y\colon\hat{\Gamma}\rightarrow\IR^3$,
depending on the parametric input $\y\in\mathbb{U}$, as
\begin{equation}\label{eq:affine_parametric_representation}
	\br_\y(\bxref)
	=
	\bvarphi_0(\bxref)
	+
	\sum_{j\geq 1}
	y_j
	\bvarphi_j(\bxref),
	\quad
	\bxref \in \hat\Gamma,
	\quad
	\y
	=
	(y_j)_{j\geq1}
	\in 
	\mathbb{U},
\end{equation}
with $\bvarphi_j\colon\hat{\Gamma}\to\IR^3$ for $j\in \IN$.
This gives rise to a collection of parametric boundaries
$(\Gamma_\y)_{\y \in \mathbb{U}}$ of the form 
\begin{equation}
	\Gamma_{\y} 
	\coloneqq 
	\{
		\bx\in \IR^3:\; 
		\bx
		= 
		\br_\y(\bxref),
		\quad
		\bxref \in \hat\Gamma
	\}.
\end{equation}
This family of affinely parametrized boundaries have
recently been investigated in the context of forward and inverse uncertainty quantification.
We refer to \cite{DHJM2022} for a detailed discussion on their construction
using non-uniform rational B-splines (NURBS).

We make the following assumptions on $\br_\y$:

\begin{assumption}\label{assump:parametric_boundary}
Let $\hat\Gamma$ be the reference Lipschitz boundary.
\begin{enumerate}
	\item\label{assump:parametric_boundary2}
	The functions $(\bvarphi_i)_{i\in \IN} \subset \mathscr{C}^{0,1}(\hat\Gamma; \IR^3)$
	are such that for each $\y\in\mathbb{U}$ one has that
	$\br_\y\colon\hat{\Gamma}\to\Gamma_\y$ is bijective and bi-Lipschitz,
	and $\Gamma_\y$ is the boundary of a bounded Lipschitz domain.
	\item\label{assump:parametric_boundary3}
	There exists $p\in(0,1)$ such that 
	\begin{equation}
		\boldsymbol{b}
		\coloneqq
		\left(
			\norm{\bvarphi_j}_{\mathscr{C}^{0,1}(\hat{\Gamma},\IR^3)} 
		\right)_{j\in \IN}
		\in 
		\ell^p(\IN).
	\end{equation}
\end{enumerate}
\end{assumption}

\begin{remark}\label{rem:parametric_boundary_implications}
\Cref{assump:parametric_boundary3} of \cref{assump:parametric_boundary}
yields absolute uniform convergence of the series
\cref{eq:affine_parametric_representation}
	as an element of $\mathscr{C}^{0,1}(\hat{\Gamma},\IR^3)$.
\end{remark}
\begin{example}[Example of affine-parametric transformations]
A common approach in the uncertainty quantification of partial differential equations on random domains is to choose $\boldsymbol{\varphi}_j=\sqrt{\lambda_j}\boldsymbol{\psi}_j$, where $(\lambda_j,\boldsymbol{\psi}_j)$, $j\geq 1$, are the eigenfunctions of a covariance integral operator with covariance kernel $C\colon\widehat{\Gamma}\times\widehat{\Gamma}\to\mathbb{R}^{3\times 3}$. Such a kernel can, for example be obtained by taking the traces of matrix-valued covariance kernels in the whole of $\mathbb{R}^3$, see, e.g., \cite{DHJM2022} for details. An illustration of possible domain transformations where the covariance function was chosen to be a diagonal Gaussian kernel is shown in \cref{fig:domain deformations}.
\begin{figure}
    \centering
    \begin{minipage}{0.3\textwidth}
    \centering
    \includegraphics[width=0.9\textwidth,clip=true,trim=400 150 300 250]{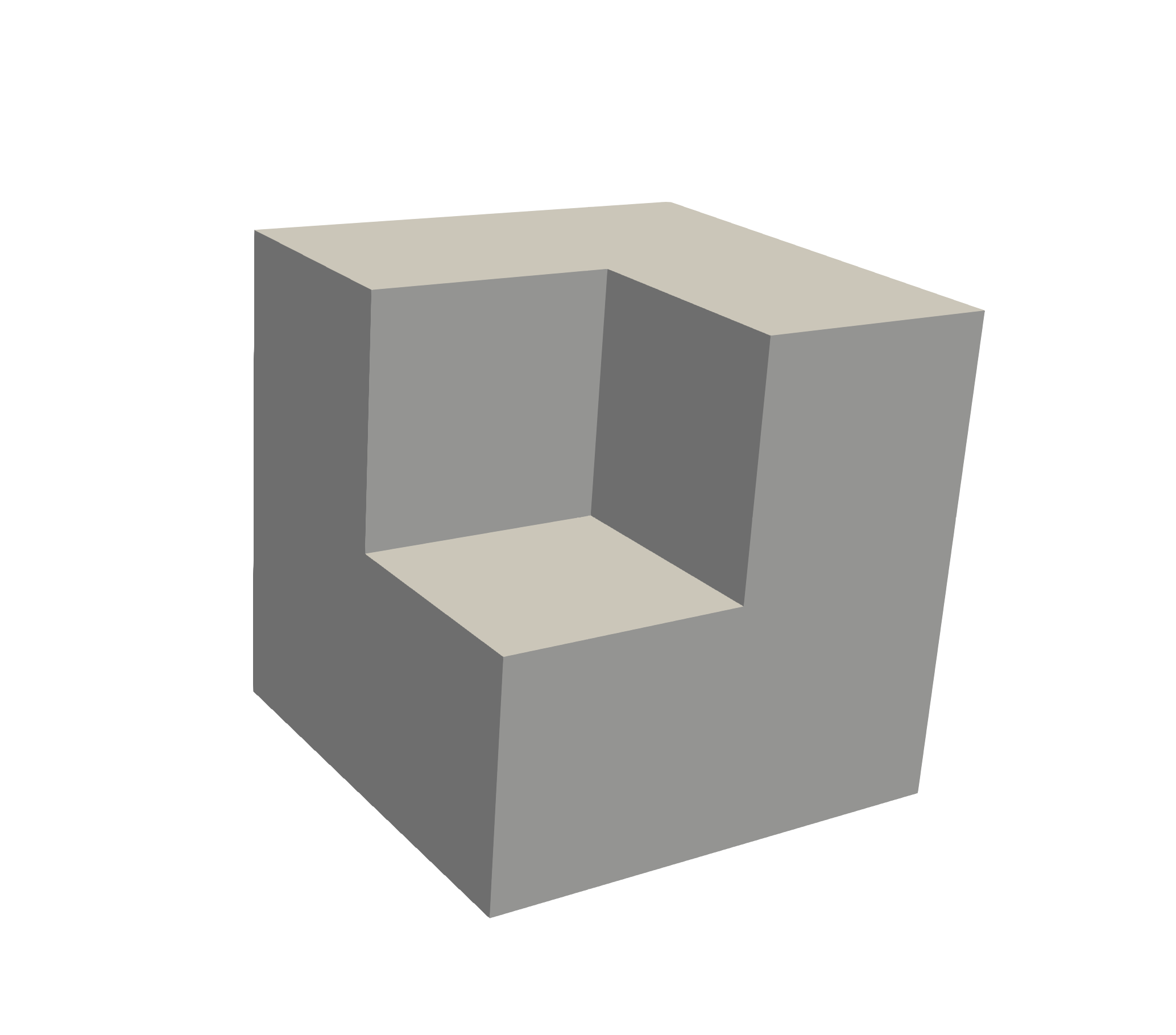}\\
    Reference boundary $\hat{\Gamma}.$
    \end{minipage}
    \hfill
    \begin{minipage}{0.6\textwidth}
    \centering
    \includegraphics[width=0.45\textwidth,clip=true,trim=400 150 300 250]{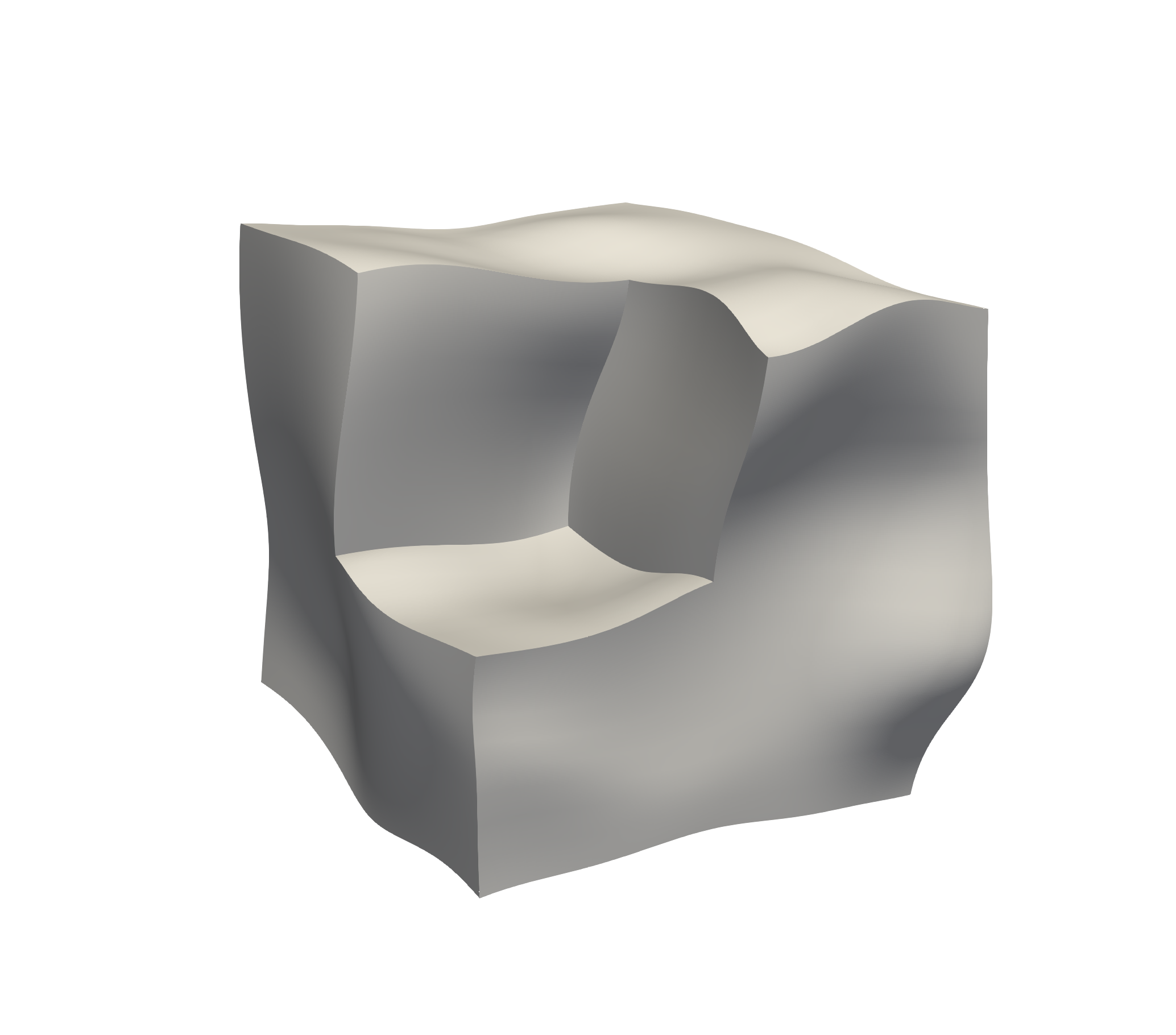}\quad
    \includegraphics[width=0.45\textwidth,clip=true,trim=400 150 300 250]{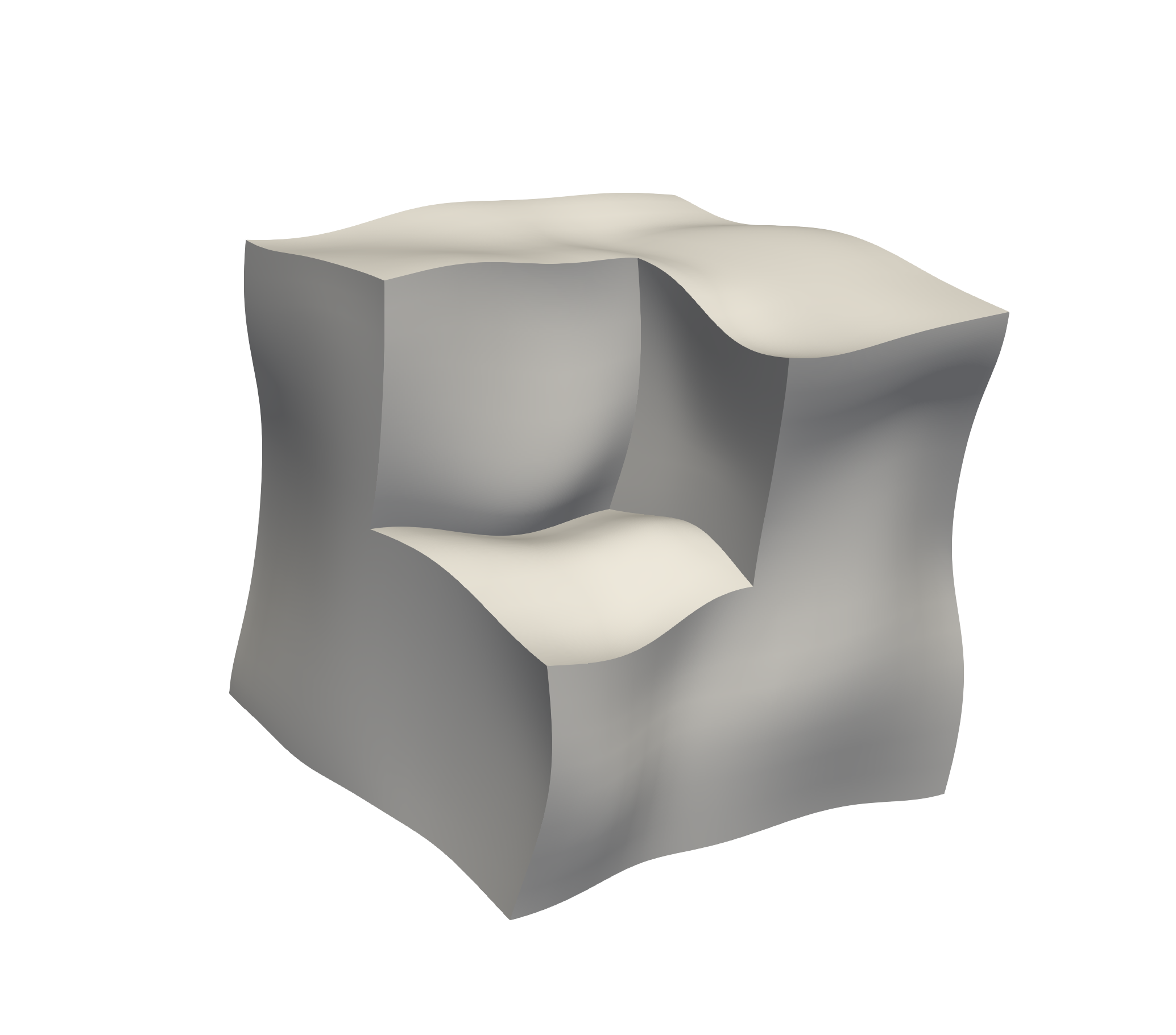}\\
    \includegraphics[width=0.45\textwidth,clip=true,trim=400 150 300 250]{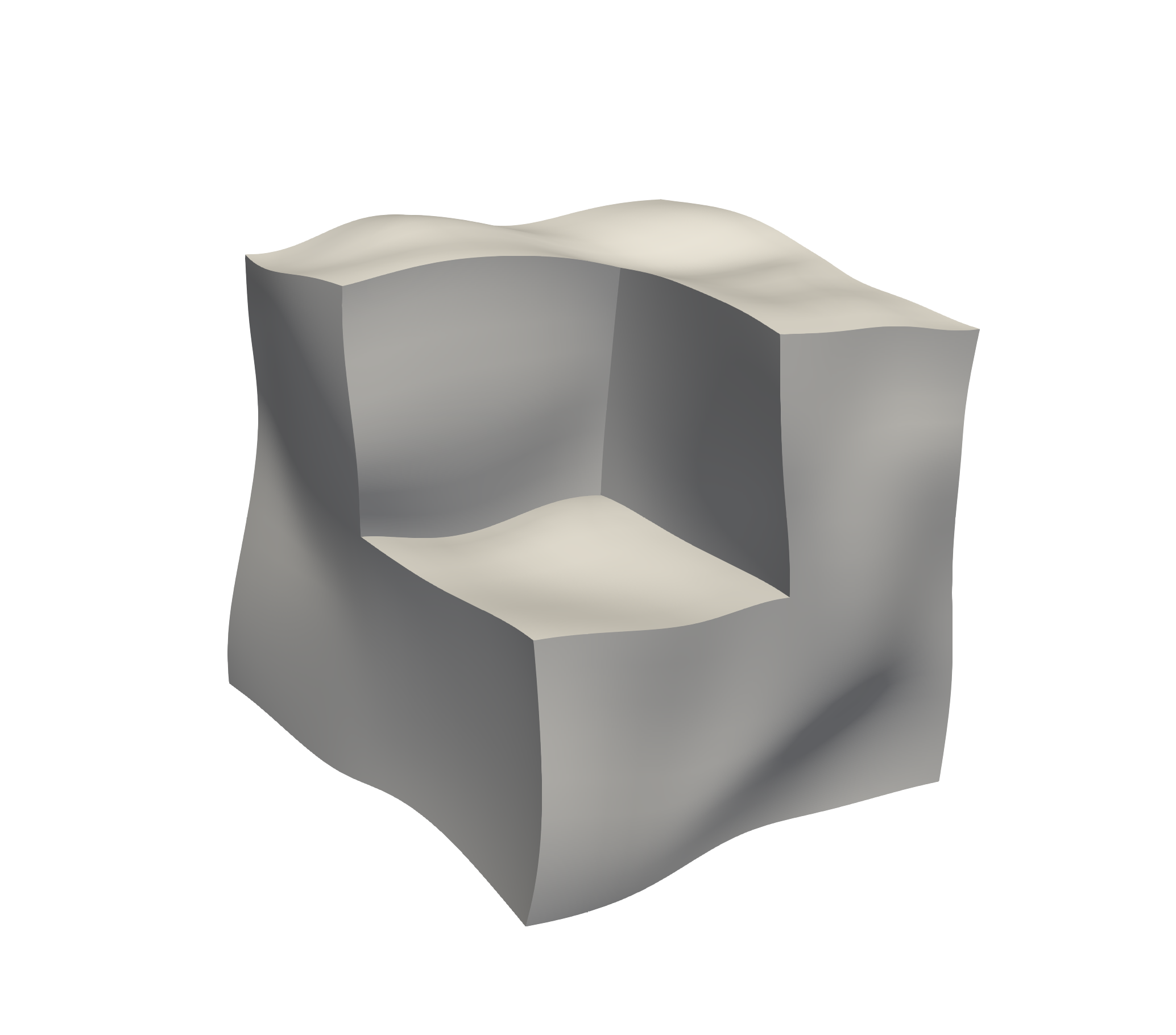}\quad
    \includegraphics[width=0.45\textwidth,clip=true,trim=400 150 300 250]{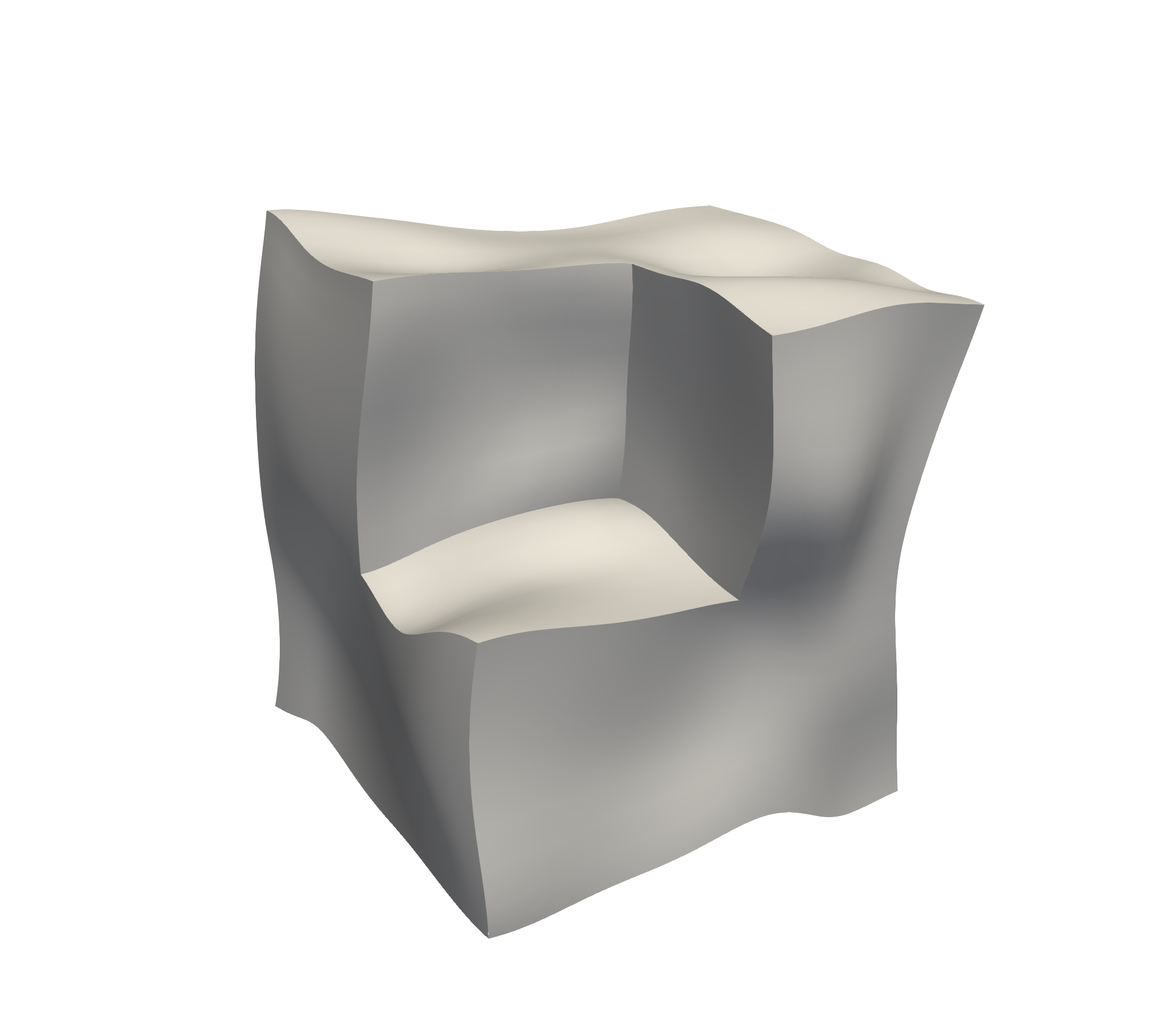}\\
    Examples of affine parametrically deformed boundary $\Gamma_\y.$
    \end{minipage}
    \caption{Illustration of a reference boundary $\widehat\Gamma$ and parametric boundaries $\Gamma_\y$.}
    \label{fig:domain deformations}
\end{figure}
\end{example}

We are interested in extending the set of
affine-parametric boundary parametrizations 
to a complex set of the form $\mathcal{O}_{\boldsymbol\rho}$,
and establishing its holomorphic dependence upon
the parametric variables, in the sense of 
\cref{def:bpe_holomorphy}. 

Let us set
\begin{align}
	\left(
		\hat{\Gamma}
		\times
		\hat{\Gamma}
	\right)^\star
	\coloneqq
	\left\{
		\hat{\bx},
		\hat{\by}
		\in 
		\hat{\Gamma}
		:
		\hat{\bx}
		\neq
		\hat{\by}
	\right\}.
\end{align}

\begin{lemma}\label{lmm:lipschitz_preservation}
Let \cref{assump:parametric_boundary} be satisfied
with $\boldsymbol{b} \in \ell^{p}(\IN)$ and $p\in (0,1)$.
There exists $\eta>0$ and $\varepsilon>0$
such that for any $(\boldsymbol{b},\varepsilon)$--admisible sequence
$\boldsymbol\rho =  (\rho_j)_{j\geq1}$ of numbers strictly
larger than one it holds
\begin{equation}
	\inf_{\z \in \mathcal{O}_{{\boldsymbol{\rho}}}}
	\inf_{(\hat\bx \times \hat\by) \in (\hat\Gamma \times \hat\Gamma)^\star}
	\Re
	\left\{
	\frac{
		\left(
			\boldsymbol{r}_{\z}(\hat\bx)
			-
			\boldsymbol{r}_{\z}(\hat\by)
		\right)
		\cdot
		\left(
			\boldsymbol{r}_{\z}(\hat\bx)
			-
			\boldsymbol{r}_{\z}(\hat\by)
		\right)
	}{
	\norm{
		\hat\bx
		-
		\hat\by
	}^2	
	}
	\right\}
	\geq
	\eta
	>0,
\end{equation}
\end{lemma}

\begin{proof}
The proof of this result follows the exact same steps as in
the proof of \cite[Proposition 3.2]{henriquez2021shape}.
For the sake of brevity, we skip it.
\end{proof}

\begin{lemma}\label{lmm:holomorphic_parametric_rep}
Let \cref{assump:parametric_boundary} be satisfied. 
Then, there exists $\varepsilon>0$ such that the map 
$\mathbb{U}\ni\y \mapsto \boldsymbol{r}_\y\in \mathscr{C}^{0,1}(\hat\Gamma,\IR^3)$
defined in \cref{eq:affine_parametric_representation}
is $(\boldsymbol{b},p,\varepsilon)$-holomorphic and continuous.
\end{lemma}

\begin{proof}
The proof of this result follows the exact same steps as in
the proof of \cite[Lemma 5.8]{CSZ18}.
For the sake of brevity, we skip it.
\end{proof}

\subsection{Parametric Holomorphy of Boundary Functions}
We study the parametric holomorphy property for a variety 
of functions appearing in the parametric holomorphy analysis
of BIOs to be carried out in \cref{sec:BIO}.

\subsubsection{Parametric Holomorphy of the Gram determinant}
Let $\y\in\mathbb{U}$, $f\colon\Gamma_\y \rightarrow \mathbb{C}$
a measurable function. Since we assume $\hat\Gamma$ to be Lipschitz,
there exists a decomposition $\mathcal{G}$ of $\hat{\Gamma}$ into finitely many
open and disjoint patches $\tau$, i.e.,
\begin{align}\label{eq:boundarydecomposition}
\hat{\Gamma}=\bigcup _{\tau\in\mathcal{G}}\overline{\tau},
\end{align}
such that each $\tau\in\mathcal{G}$ is the image of a Lipschitz
diffeomorphism $\chi_\tau\colon\widetilde{\tau}\to\tau$ with
$\widetilde{\tau}\subset\IR^2$. We remark that
$\chi_\tau$ is defined for each $\widetilde{\tau}$.

For each $\tau\in\mathcal{G}$,
we set $\chi_{\tau,\y}=\br_\y\circ\chi_\tau$ and $\widetilde{f}_{\tau,\y}=f\circ\br_\y\circ\chi_\tau$.
Then, the surface integral of $f$ over $\Gamma_\y$ can be written as
\begin{align}
	\int\limits_{\Gamma_\y} f(\mathbf{x}) \mathrm{d} s_{\mathbf{x}}
	=
	\sum_{\tau \in \mathcal{G}} 
	\int\limits_{\widetilde{\tau}} \widetilde{f}_{\tau,\y}(\widetilde{\mathbf{u}}) 
	\sqrt{g_{\tau,\y}(\widetilde{\mathbf{u}})} 
	\mathrm{d} \widetilde{\mathbf{u}},
\end{align}
with the Gram determinant
\begin{align}
g_{\tau,\y}(\widetilde{\mathbf{u}})=\operatorname{det}G_{\tau,\y}(\widetilde{\mathbf{u}}),\quad\widetilde{\mathbf{u}}\in\widetilde{\tau},
\end{align}
of the Gram matrix
\begin{align}
G_{\tau,\y}(\widetilde{\mathbf{u}})=\big(D_{\widetilde{\mathbf{u}}}\chi_{\tau,\y}(\widetilde{\mathbf{u}})\big)^{\top}\big(D_{\widetilde{\mathbf{u}}}\chi_{\tau,\y}(\widetilde{\mathbf{u}})\big),\quad\widetilde{\mathbf{u}}\in\widetilde{\tau},
\end{align}
where $D_{\widetilde{\mathbf{u}}}\chi_{\tau,\y}(\widetilde{\mathbf{u}}) \in \IR^{3 \times 2}$
signifies the Jacobian of $\chi_{\tau,\y}$.
\begin{lemma}\label{lemma:gramian_holomorphic}
Let \cref{assump:parametric_boundary} be satisfied.
Then there exists $\varepsilon>0$ such that for each
$\tau \in \mathcal{G}$ the maps
\begin{align}\label{eq:gramiandet}
	\mathbb{U}
	\ni
	\y
	\mapsto
	g_{\tau,\y}
	\in
	L^\infty(\widetilde{\tau})
\end{align}
and
\begin{align}\label{eq:gramiandetroot}
	\mathbb{U}
	\ni
	\y
	\mapsto
	\sqrt{g_{\tau,\y}}
	\in
	L^\infty(\widetilde{\tau})
\end{align}	
are $(\boldsymbol{b},p,\varepsilon)$-holomorphic and continuous. 
Moreover, for any $(\boldsymbol{b},\varepsilon)$--admisible sequence
$\boldsymbol\rho =  (\rho_j)_{j\geq1}$ of numbers strictly
larger than one, for the extension
\begin{align}
	\mathcal{O}_{\boldsymbol{\rho}}
	\ni
	\z
	\mapsto
	g_{\tau,\z}
	\in
	L^\infty(\widetilde{\tau})
\end{align}
there exist $0<\underline{\zeta}\leq\overline{\zeta}<\infty$ such that
\begin{align}\label{eq:gramianbound}
	0
    <
    \underline{\zeta}
    \leq
    \inf_{\z \in \mathcal{O}_{\boldsymbol{\rho}}}
	\inf_{\widetilde{\mathbf{u}} \in \widetilde{\tau}}
    \Re
	(g_{\tau,\z}(\widetilde{\mathbf{u}} ))
	\leq
    \sup_{\z \in \mathcal{O}_{\boldsymbol{\rho}}}
	\sup_{\widetilde{\mathbf{u}} \in \widetilde{\tau}}
    \Re
	(g_{\tau,\z}(\widetilde{\mathbf{u}} ))
    \leq
    \overline{\zeta}
    <
    \infty.
\end{align}
In addition, for each $\tau \in \mathcal{G}$ it holds 
\begin{equation}\label{eq:norm_sqrt_g}
	\sup_{\z \in \mathcal{O}_{\boldsymbol{\rho}}}
	\norm{\sqrt{g_{\tau,\z}}}_{L^\infty(\widetilde{\tau})}
	<\infty.
\end{equation}
\end{lemma}
\begin{proof}
It follows from \cref{assump:parametric_boundary3} in
\cref{assump:parametric_boundary} that for each $\tau \in \mathcal{G}$ we have that
\begin{equation}
\mathbb{U}
\ni
\y
\mapsto
D_{\widetilde{\bx}}\chi_{\tau,\y}\in[L^\infty(\widetilde{\tau})]^{3\times 2}
\end{equation}
and, thus, also
\begin{equation}
	\mathbb{U}
	\ni
	\y
	\mapsto
	G_{\tau,\by}
	\in 
	\left[
		L^{\infty}(\widetilde{\tau})
	\right]^{2\times2}
\end{equation}
is $(\boldsymbol{b},p,\varepsilon')$-holomorphic and continuous with $(\boldsymbol{b},p,\varepsilon')$ as in \cref{lmm:holomorphic_parametric_rep}. Since the determinant is a polynomial, this yields $(\boldsymbol{b},p,\varepsilon')$-holomorphy of $g_{\tau,\z}$. Moreover, for each $\y\in\mathbb{U}$, the Gram matrix $G_{\tau,\y}$ is positive definite and symmetric. Thus, it has real and positive eigenvalues whose product coincides with $g_{\tau,\y}$. Since the eigenvalues depend continuosly on $G_{\tau,\y}$, this yields that $g_{\tau,\y}$ is bounded uniformly from below and above. Thus, there is $\varepsilon>0$ such that \cref{eq:gramianbound} and \cref{eq:norm_sqrt_g} hold,
which implies the $(\boldsymbol{b},p,\varepsilon)$-holomorphy and continuity of \cref{eq:gramiandetroot}.
\end{proof}

\subsubsection{The Pullback Operator}
For each $\y \in \mathbb{U}$ and $\varphi \in L^2(\Gamma_\y)$, we set
\begin{equation}
	\tau_\y \varphi
	\coloneqq
	\varphi\circ\boldsymbol{r}_\y
    \in L^2(\hat\Gamma).
\end{equation}

\begin{lemma}\label{lmm:pullback_operator}
Let \cref{assump:parametric_boundary} be satisfied.
Then, for each $\y \in \mathbb{U}$ the map
\begin{equation}
	L^2(\Gamma_\y)
	\ni
	\varphi
	\mapsto
	\hat{\varphi}
	\coloneqq
	\tau_\y \varphi
	\in
	L^2(\hat{\Gamma})
\end{equation}
defines an isomorphism, i.e.
$\tau_\y\in\mathscr{L}_{\normalfont{\text{iso}}}(L^2_\y(\Gamma),L^2(\hat{\Gamma}))$
for each $\y \in \mathbb{U}$.
\end{lemma}
\begin{proof}
As a consequence of \cref{assump:parametric_boundary} 
 for each $\y \in \mathbb{U}$ we obtain
\begin{equation}
	\int\limits_{\Gamma_\y} 
	\snorm{\varphi(\mathbf{x})}^2
	\text{d} 
	s_{\mathbf{x}}
	=
	\sum_{\tau \in \mathcal{G}} 
	\int\limits_{\widetilde{\tau}} 
	\snorm{
	\left(
		\tau_\y \varphi
		\circ
		\chi_\tau
	\right)
	(\widetilde{\mathbf{u}}) 
	}^2
	\sqrt{g_{\tau,\y}(\widetilde{\mathbf{u}})} 
	\text{d} 
	\widetilde{\mathbf{u}},
\end{equation}
and
\begin{equation}
	\int\limits_{\hat\Gamma} 
	\snorm{
	\left(
		\tau_\y \varphi
	\right)
	(\hat{\mathbf{x}}) 
	}^2
	\text{d} 
	s_{\hat{\mathbf{x}}}
	=
	\sum_{\tau \in \mathcal{G}} 
	\int\limits_{\widetilde{\tau}} 
	\snorm{
	\left(
		\tau_\y \varphi
		\circ
		\chi_\tau
	\right)
	(\widetilde{\mathbf{u}}) 
	}^2
	\sqrt{g_{\tau}(\widetilde{\mathbf{u}})} 
	\text{d} 
	\widetilde{\mathbf{u}},
\end{equation}
where $g_{\tau}$ is the Gram determinant
of $ \chi_\tau$. Using \cref{eq:gramianbound} on $\mathbb{U}\subset\mathcal{O}_{\boldsymbol{\rho}}$ yields the assertion.
\end{proof}

\subsubsection{Parametric Shape Holomorphy of the Normal Vector}
Certain integral operators, such as the double layer operator in
potential layer theory, depend on the outward pointing normal.
To this end, we establish parametric holomorphy of the outward
pointing normal vector. 
\begin{lemma}\label{lmm:normal_derivative}
Let \cref{assump:parametric_boundary} be fulfilled and $(\boldsymbol{b},p,\varepsilon)$ as in \cref{lemma:gramian_holomorphic}. Then, pullback of the outward pointing normal
\begin{equation}\label{eq:pullbacknormal}
	\mathbb{U}
	\ni
	\y
	\mapsto
	\hat{\mathbf{n}}_\y\isdef
    \tau_\y 
	\mathbf{n}_{\Gamma_\y}
	\in 
	L^{\infty}(\hat\Gamma)
\end{equation}	
is $(\boldsymbol{b},p,\varepsilon)$-holomorphic and continuous.
\end{lemma}

\begin{proof}
Let $\tau\in\mathcal{G}$, $\chi_{\tau,\y}=\br_\y\circ\chi_\tau$, and $\hat{\bx}=\chi_{\tau}(\widetilde{\mathbf{u}})$. Then, the non-normalized outward-pointing normal
\begin{equation}
\hat{\mathbf{m}}_\y(\hat{\bx})=\big(\partial_1\chi_{\tau,\y}\times\partial_2\chi_{\tau,\y}\big)(\hat{\mathbf{u}})
\end{equation}
is $(\boldsymbol{b},p,\varepsilon)$-holomorphic due to \cref{lmm:holomorphic_parametric_rep}.
Moreover, due to $\norm{\hat{\mathbf{m}}_{\y}(\hat{\mathbf{x}})}=\sqrt{g_{\tau,\y}(\widetilde{\mathbf{u}})}$ for all $\y\in\mathbb{U}$, \cref{lemma:gramian_holomorphic} implies that also $\norm{\hat{\mathbf{m}}_{\y}}$ is $(\boldsymbol{b},p,\varepsilon)$-holomorphic and continuous.
Thus, since
\begin{equation}
	\hat{\mathbf{n}}_\y(\hat{\bx})
	=
	\frac{\hat{\mathbf{m}}_\y(\hat{\bx})}{\norm{\hat{\mathbf{m}}_\y(\hat{\bx})}},
	\quad
	\hat{\bx}
	\in
	\hat\Gamma,
\end{equation}
\cref{{eq:gramianbound}} implies the assertion.
\end{proof}

\subsection{Riesz-Thorin Interpolation Theorem}
\label{sec:rt_theorem}
We briefly recall a result concerning the boundedness
of BIOs in a $L^2$-based setting, to be used in the parametric holomorphy
analysis in \cref{sec:BIO}.  To this end, let $\Gamma$ be a bounded
Lipschitz boundary in $\IR^3$ and
\begin{equation}
	(\OT_\Gamma \varphi)(\bx)
	\isdef
	\int\limits_{\Gamma}
	{\sf t}(\by,\bx-\by)\varphi(\by)\mathrm{ds}_\by,
	\quad
	\bx
	\in
	\Gamma,
\end{equation}
an integral operator with kernel
$\mathsf{t}: \Gamma \times \IR^3 \backslash \{{\bf 0}\} \rightarrow \IC$.
In the following, to show the mapping property 
$\OT_\Gamma\colon L^2(\Gamma)\to L^2(\Gamma)$, 
we rely on the \emph{Riesz-Thorin interpolation theorem}
\cite[Theorem 2.b.14]{LT77}, which states that
the operator norm of $\OT_{\Gamma}: L^2(\Gamma)
\rightarrow  L^2(\Gamma)$
is bounded according to
\begin{align}\label{eq:rieszthorin}
	\norm{
		\OT_{{\Gamma}}
	}_{\mathscr{L}\left(L^2(\Gamma), L^2(\Gamma)\right)}
	\leq
	\norm{
		\OT_{{\Gamma}}
	}_{\mathscr{L}\left(L^1(\Gamma), L^1(\Gamma)\right)}^{\half}
	\norm{
		\OT_{{\Gamma}}
	}_{\mathscr{L}\left(L^\infty(\Gamma),L^\infty(\Gamma)\right)}^{\half}.
\end{align}
To estimate the operator norms in the upper bound of \cref{eq:rieszthorin}
we use their explicit characterizations
\begin{align}	
	\norm{
		\OT_{{\Gamma}}
	}_{\mathscr{L}\left(L^1(\Gamma), L^1(\Gamma)\right)}
	&=
	\underset{\by \in {\Gamma}}{\operatorname{ess} \sup} 
	\int\limits_{\Gamma}|\mathsf{t}(\by, \bx-\by)| \mathrm{d} s_\bx,
	\\
	\norm{
		\OT_{{\Gamma}}
	}_{\mathscr{L}\left(L^\infty(\Gamma), L^\infty(\Gamma)\right)}
	&=
	\underset{\bx \in {\Gamma}}{\operatorname{ess \sup}}
	\int\limits_{\Gamma}|\mathsf{t}(\by, \bx-\by)| \mathrm{d} s_\by.
\end{align}

\section{Parametric Shape Holomorphy of Boundary Integral Operators}
\label{sec:BIO}
In this section, we present and prove the parametric holomorphy property
for the parameter-to-operator \cref{eq:param_to_operator} and
parameter-to-solution maps \cref{eq:param_to_solution}
in an $L^2$-based setting.
To this end, for each $\y \in \mathbb{U}$ and by using the pullback operator 
$\tau_\y: L^2(\Gamma_\y)\rightarrow L^2(\hat\Gamma)$,
we can transport the BIO $\OA_{\y}: L^2(\Gamma_\y)\rightarrow L^2(\Gamma_\y)$
introduced in \cref{eq:parametric_BIO} to the reference boundary $\hat\Gamma$ as follows
\begin{align}\label{eq:A_reference}
	\hat{\OA}_\y
	\coloneqq
	\tau_\y
	\,
	\OA_{\y}
	\,
	\tau^{-1}_\y.
\end{align}
As a consequence of \cref{lmm:pullback_operator}, 
for each $\y \in \mathbb{U}$, one has that
${\OA}_{\y}: L^2(\Gamma_\y) \rightarrow  L^2(\Gamma_\y)$
is a bounded linear operator if and only if
$\hat{\OA}_\y\colon L^2(\hat\Gamma)\to L^2(\hat\Gamma)$ is so as well.

The remainder of the section is structured as follows.
In \cref{sec:kernel_bounded}, assuming
that the kernel function ${\sf a}: \Gamma_\y \times \IR \rightarrow \IC$
is bounded, we prove parametric holomorphy the map $\y \mapsto \hat{\OA}_\y$,
in the sense of \cref{def:bpe_holomorphy}. 
Next, in \cref{sec:parametric_holomorphy}, we no longer assume 
${\sf a}(\cdot,\cdot)$ to be bounded, but only integrable, and extend 
the parametric holomorphy result to this scenario.  
Finally, in \cref{sec:parameric_d2s_map},
we establish parametric holomorphy of the parameter-to-solution map.
\subsection{Parametric Shape Holomorphy for BIOs with Bounded Kernel}
\label{sec:kernel_bounded}
We proceed to establish parametric shape 
holomorphy of a class of BIOs characterized by a bounded kernel function.



\begin{theorem}\label{thm:holomorphy_bounded_kernel}
Let \cref{assump:parametric_boundary} be fulfilled
with $\boldsymbol{b} \in \ell^p(\IN)$ and $p\in (0,1)$.
For each $\tau \in \mathcal{G}$ we define
\begin{align}
	\mathsf{k}_{\tau,\y}
	\left(
		\hat{\bx}
		,
		\widetilde{\mathbf{u}}
	\right)
	\coloneqq
	\mathsf{a}
	\left(
		\boldsymbol{r}_\y \circ \chi_\tau
		\left(
			\widetilde{\mathbf{u}}
		\right)
		,
		\boldsymbol{r}_\y(\hat{\mathbf{x}})
		-
		\boldsymbol{r}_\y \circ \chi_\tau
		\left(
			\widetilde{\mathbf{u}}
		\right)
	\right),
	\quad
	(\hat{\bx},\widetilde{\mathbf{u}})
	\in 
	\hat{\Gamma} \times \widetilde{\tau}.
\end{align}
In addition, for each $\tau \in \mathcal{G}$ we assume that the map
\begin{equation}\label{eq:ktauy}
	\mathbb{U}
	\ni
	\y
	\mapsto
	\mathsf{k}_{\tau,\y}
	\in
	L^{\infty}
	\left(
		\hat{\Gamma} \times \widetilde{\tau}
	\right)
\end{equation}
is $(\boldsymbol{b},p,\varepsilon)$-holomorphic
and continuous for some $\varepsilon>0$.
Then, the map
\begin{equation}
	\mathcal{A}:
	\mathbb{U}
	\rightarrow
	\mathscr{L}
	\left(
		L^2(\hat\Gamma)
		,
		L^2(\hat\Gamma)
	\right):
	\y
	\mapsto
	\hat{\mathsf{A}}_{\y}
\end{equation}
with $\hat{\OA}_\y$ as in \cref{eq:A_reference} is $(\boldsymbol{b},p,\varepsilon)$-holomorphic
and continuous.
\end{theorem}

\begin{proof}
Throughout, let $\boldsymbol\rho\coloneqq (\rho_j)_{j\geq1}$ 
be any  $(\boldsymbol{b},\varepsilon)$-admissible 
sequence of numbers of numbers strictly larger than one.
For the sake of readability, we divide the proof
into three steps.\\

{\bf Step \encircled{A}: Localization of $\hat{\OA}_\y$ to a patch.} 
Using the decomposition into patches $\mathcal{G}$
of the reference boundary $\hat\Gamma$
as introduced in \cref{eq:boundarydecomposition},
for any $\hat{\varphi} \in L^2(\hat{\Gamma})$ and 
$\y \in \mathbb{U}$ one has
\begin{align}\label{eq:decomposition_A_panels}
	\left(
		\hat{\OA}_\y
		\,
		\hat{\varphi} 
	\right)
	\left(
		\hat{\bx}
	\right)
	=
	\sum_{\tau \in \mathcal{G}} 
	\left(
		\hat{\OA}_{\tau,\y}
		\,
		\hat{\varphi} 
	\right)
	\left(
		\hat{\bx}
	\right),
	\quad
	\hat\bx
	\in
	\hat\Gamma,
\end{align}
where, for $\hat{\varphi} \in L^2(\hat\Gamma)$, we set
\begin{equation}\label{eq:local_A_y}
	\left(
		\hat{\OA}_{\tau,\y}
		\,
		\hat{\varphi} 
	\right)
	\left(
		\hat{\bx}
	\right)
	\coloneqq
	\int\limits_{\widetilde{\tau}} 
	\mathsf{b}_{\tau,\y}
	(\hat\bx,\widetilde{{\bf u}})
	\left(
		\hat{\varphi}
		\circ
		\chi_\tau
	\right)
	(\widetilde{\mathbf{u}})
	\,
	\text{d} 
	\widetilde{\mathbf{u}},
\end{equation}
with
\begin{equation}
\mathsf{b}_{\tau,\y}(\hat\bx,\widetilde{{\bf u}})\coloneqq\mathsf{k}_{\tau,\y}\left(
\hat{\bx},\widetilde{\mathbf{u}}\right) \sqrt{g_{\tau,\y}(\widetilde{\mathbf{u}})},
\qquad 
(\hat\bx,\hat{\bf u})\in\hat{\Gamma} \times \widetilde{\tau}.
\end{equation}
To obtain an extension of \eqref{eq:local_A_y} in
$\y\in\mathbb{U}$ to complex-valued inputs $\z \in \mathcal{O}_{\boldsymbol{\rho}}$,
for each $\tau \in \mathcal{G}$, let $\mathsf{k}_{\tau,\z}$ denote the holomorphic
extension of $\mathsf{k}_{\tau,\y}$ to $\z \in \mathcal{O}_{\boldsymbol{\rho}}$ and define
\begin{equation}
	\mathsf{b}_{\tau,\z}
	(\hat\bx,\widetilde{{\bf u}})
	\coloneqq
	\mathsf{k}_{\tau,\z}
	\left(
		\hat{\bx}
		,
		\widetilde{\mathbf{u}}
	\right)
	\sqrt{g_{\tau,\z}(\widetilde{\mathbf{u}})} 
	,
	\quad
	(\hat\bx,\widetilde{{\bf u}})
	\in
	\hat{\Gamma} \times \widetilde{\tau},
\end{equation}
and
\begin{equation}\label{eq:defAhattauz}
	\left(
		\hat{\OA}_{\tau,\z}
		\,
		\hat{\varphi} 
	\right)
	\left(
		\hat{\bx}
	\right)
	\coloneqq
	\int\limits_{\widetilde{\tau}} 
	\mathsf{b}_{\tau,\z}
	(\hat\bx,\widetilde{{\bf u}})
	\left(
		\hat{\varphi}
		\circ
		\chi_\tau
	\right)
	(\widetilde{\mathbf{u}})
	\,
	\text{d} 
	\widetilde{\mathbf{u}},
\end{equation}
for $\hat{\varphi} \in L^2(\hat\Gamma)$. This induces the map
\begin{equation}\label{eq:linftyAt}
	\mathcal{A}_\tau:
	\mathbb{U}
	\rightarrow
	\mathscr{L}
	\left(
		L^2(\hat\Gamma)
		,
		L^2(\hat\Gamma)
	\right):
	\y
	\mapsto
	\hat{\mathsf{A}}_{\tau,\y}.
\end{equation}

{\bf Step \encircled{{\sf B}}: $(\boldsymbol{b},p,\varepsilon)$-holomorphy of $\mathcal{A}_{\tau}$.}
Now, for each $\tau \in \mathcal{G}$, the $(\boldsymbol{b},p,\varepsilon_{\tau,g})$-holomorphy and continuity of \cref{eq:gramiandetroot} and the $(\boldsymbol{b},p,\varepsilon_{\tau,k})$-holomorpy and continuity of \cref{eq:ktauy} imply that the map 
\begin{equation}\label{eq:linftykernelderivative}
	\mathbb{U}
	\ni
	\y
	\mapsto
	\mathsf{b}_{\tau,\y}
	\in
	L^\infty
	\left(
		\hat{\Gamma} 
		\times 
		\widetilde{\tau}
	\right)
\end{equation}
is $(\boldsymbol{b},p,\varepsilon_\tau)$-holomorphic with $\varepsilon_\tau=\min\{\varepsilon_{\tau,g},\varepsilon_{\tau,k}\}>0$
and continuous.
Moreover, for all $\z\in\mathcal{E}_{\boldsymbol\rho}$ it holds
\begin{equation}\label{eq:boundAtzwithbtz}
	\norm{
		\hat{\OA}_{\tau,\z}
	}_{\mathscr{L}(L^2(\hat\Gamma), L^2(\hat\Gamma))}
	\leq
	C_{\tau}
	\norm{
		\mathsf{b}_{\tau,\z}
	}_{
	L^\infty
	\left(
		\hat{\Gamma} 
		\times 
		\widetilde{\tau}
	\right)
	},
\end{equation}
which, together with
the $(\boldsymbol{b},p,\varepsilon)$-holomorphy
of the map $\y\mapsto \mathsf{b}_{\tau,\y}$ from \cref{eq:linftykernelderivative},
allows to verify
\cref{def:bpe_holomorphy1} and \cref{def:bpe_holomorphy3} in the \cref{def:bpe_holomorphy} of 
$(\boldsymbol{b},p,\varepsilon)$-holomorphy  for the map \cref{eq:linftyAt}.

To verify \cref{def:bpe_holomorphy2} we note that the $(\boldsymbol{b},p,\varepsilon)$-holomorphy of \cref{eq:linftykernelderivative} implies that for all $\z\in\mathcal{O}_{\boldsymbol{\rho}}$ we have that 
$
\mathsf{b}_{\tau,\z}
$
is holomorphic, i.e., it holds
\begin{equation}
\lim_{|h|\to 0^+}
\frac{
\norm{
\mathsf{b}_{\tau,\z + h \boldsymbol{e}_j}
-
\mathsf{b}_{\tau,\z}
-
h\partial_{z_j}\mathsf{b}_{\tau,\z}
}_{L^{\infty}(\hat{\Gamma}\times\widetilde{\tau})}
}{
\snorm{h}
}
=0.
\end{equation}
From the linearity of the integral and in analogy to \cref{eq:boundAtzwithbtz} we obtain
\begin{align*}
&\norm{
\hat{\OA}_{\tau,\z + h \boldsymbol{e}_j}-\hat{\OA}_{\tau,\z}-h\partial_{z_j}\hat{\OA}_{\tau,\z}
}_{L^2(\hat{\Gamma})\to L^2(\hat{\Gamma})}\\
&\qquad\qquad\qquad\qquad\leq
C_\tau
\norm{
\mathsf{b}_{\tau,\z + h \boldsymbol{e}_j}
-
\mathsf{b}_{\tau,\z}
-
h\partial_{z_j}\mathsf{b}_{\tau,\z}
}_{L^{\infty}(\hat{\Gamma}\times\widetilde{\tau})},
\end{align*}
where $C_\tau>0$ does not depend on $h$, which implies
\begin{equation}
\lim_{|h|\to 0^+}
\frac{
\norm{
\hat{\OA}_{\tau,\z + h \boldsymbol{e}_j}-\hat{\OA}_{\tau,\z}-h\partial_{z_j}\hat{\OA}_{\tau,\z}
}_{L^2(\hat{\Gamma})\to L^2(\hat{\Gamma})}
}{
\snorm{h}
}
=0,
\end{equation}
i.e., that $\hat{\OA}_{\tau,\z}$ is holomorphic for each $\z\in\mathcal{O}_{\boldsymbol{\rho}}$.
Therefore, we have verified \cref{def:bpe_holomorphy2} in 
\cref{def:bpe_holomorphy} and \cref{eq:linftyAt}
is $(\boldsymbol{b},p,\varepsilon)$-holomorphic
for each $\tau \in \mathcal{G}$.
\\
{\bf Step \encircled{{\sf C}}: $(\boldsymbol{b},p,\varepsilon)$-holomorphy of $\mathcal{A}$.}
It remains to conclude that $\mathcal{A}$ is $(\boldsymbol{b},p,\varepsilon)$-holomorphic with $\varepsilon$ taken as the minimum $\varepsilon$ of all patches (whose number is finite),
thus rendering a strictly positive value.
It follows from \cref{eq:decomposition_A_panels} that
\begin{equation}
\mathcal{A}=\sum_{\tau\in\mathcal{G}}\mathcal{A}_{\tau}
\end{equation}
and since the addition of a finite number of holomorphic
maps preserves holomorphy, this yields the desired result.
\end{proof}

\subsection{Parametric Shape Holomorphy of BIOs with Singular Kernel}
\label{sec:parametric_holomorphy}
We now address the case in which the kernel function
${\sf a}: \Gamma_\y \times \IR^3\rightarrow \IC$
of \cref{eq:parametric_BIO} is allowed to have a singularity
in the second variable, however
we assume that is integrable. Let us define
\begin{align}
	\mathfrak{U} 
	\coloneqq 
	\{
		\bz \in \IC^3: 
		\Re\{\bz \cdot \bz\}>0
	\}.
\end{align}
%
\begin{theorem}\label{thm:holomorphy_sing_kernel}
Let \cref{assump:parametric_boundary} be fulfilled
with $\boldsymbol{b} \in \ell^p(\IN)$ and $p\in (0,1)$.
In addition, assume that
\begin{enumerate}
	\item\label{thm:holomorphy_sing_kernel1}
	There exists $\varepsilon>0$
	such that for each $\tau \in \mathcal{G}$,
	$\widetilde{\mathbf{u}} \in \widetilde{\tau}$
	and for all
	$\bz \in \mathfrak{U}$ 
	the map
	\begin{equation}\label{eq:bound_kernel}
		\mathbb{U}
		\ni
		\y
		\mapsto
		\mathsf{a}
		\left(
		\boldsymbol{r}_\y \circ \chi_\tau
		\left(
			\widetilde{\mathbf{u}}
		\right)
		,
		\bz
		\right)
		\in 
		\IC
	\end{equation}
	is $(\boldsymbol{b},p,\varepsilon)$-holomorphic and continuous.
	\item\label{thm:holomorphy_sing_kernel2}
	Let $\boldsymbol\rho\coloneqq (\rho_j)_{j\geq1}$ 
	be any $(\boldsymbol{b},\varepsilon)$-admissible 
	sequence of numbers strictly larger than one.
	For any $\z \in \mathcal{O}_{\boldsymbol{\rho}}$, $\widetilde{\mathbf{u}} \in \widetilde{\tau}$, and $\tau\in\mathcal{G}$, the map\footnote{We note the subtle difference in notation between the parametric variable $\z\in\mathbb{U}$ and the spatial variable $\bz\in\mathfrak{U}$.}
	\begin{equation}
		\mathfrak{U}
		\ni
		\bz
		\mapsto
		\mathsf{a}
		\left(
		\boldsymbol{r}_\z \circ \chi_\tau
		\left(
			\widetilde{\mathbf{u}}
		\right)
		,
		\bz
		\right)
		\in 
		\IC
	\end{equation}
	is holomorphic.
	\item\label{thm:holomorphy_sing_kernel3}
	Again, let $\boldsymbol\rho\coloneqq (\rho_j)_{j\geq1}$ 
	be any $(\boldsymbol{b},\varepsilon)$-admissible 
	sequence of numbers strictly larger than one.
	There exists $C_{\mathsf{a}}>0$ and $\nu \in (0,2)$
	such that for any $\tau \in \mathcal{G}$, $\widetilde{\mathbf{u}}\in \widetilde{\tau}$,
	and $\z \in \mathcal{O}_{\boldsymbol{\rho}}$
	one has for $\left(\hat{\bx},\hat{\by}\right)\in\left(\hat{\Gamma}\times\hat{\Gamma}\right)^\star$
	\begin{equation}\label{eq:bound_kernel_singular}
   		\snorm{
		\mathsf{a}
		\left(
		\boldsymbol{r}_\z \circ \chi_\tau
		\left(
			\hat{\by}
		\right)
		,
		\boldsymbol{r}_\z(\hat{\bx})
        		-
        		\boldsymbol{r}_\z(\hat{\by})
		\right)
		}
		\leq
		\frac{C_{\mathsf{a}}}{
		\norm{
		\boldsymbol{r}_\z(\hat{\bx})
      		 -
        		\boldsymbol{r}_\z(\hat{\by})
		}^{\nu}
		}.
	\end{equation}

\end{enumerate}

Then, the map
\begin{equation}\label{eq:bpe_hol_A_sing}
	\mathcal{A}:
	\mathbb{U}
	\rightarrow
	\mathscr{L}
	\left(
		L^2(\hat\Gamma)
		,
		L^2(\hat\Gamma)
	\right):
	\y
	\mapsto
	\hat{\mathsf{A}}_{\y}
\end{equation}
with $\hat{\OA}_\y$ as in \cref{eq:A_reference} is $(\boldsymbol{b},p,\varepsilon)$-holomorphic
and continuous.
\end{theorem}
For the proof of the theorem, we follow the general strategy of the proof of
\cref{thm:holomorphy_bounded_kernel} and make the necessary modifications
for singular kernel functions.


\begin{proof}[Proof of \cref{thm:holomorphy_sing_kernel}]
As per usual, throughout this proof
let $\boldsymbol\rho\coloneqq (\rho_j)_{j\geq1}$ 
be any sequence of numbers strictly larger than
one that is $(\boldsymbol{b},\varepsilon)$-admissible 
for some $\varepsilon>0$.\\

{\bf Step \encircled{\sf A}: Localization of $\hat{\OA}_\y$ to a patch.}
In analogy to the proof of \cref{thm:holomorphy_bounded_kernel} we set
\begin{align}\label{eq:singulark}
	\mathsf{k}_{\tau,\y}
	\left(
		\hat{\bx}
		,
		\widetilde{\mathbf{u}}
	\right)
	\coloneqq
	\mathsf{a}
	\left(
		\boldsymbol{r}_\y \circ \chi_\tau
		\left(
			\widetilde{\mathbf{u}}
		\right)
		,
		\boldsymbol{r}_\y(\hat{\mathbf{x}})
		-
		\boldsymbol{r}_\y \circ \chi_\tau
		\left(
			\widetilde{\mathbf{u}}
		\right)
	\right)
\end{align}
and
\begin{equation}
	\mathsf{b}_{\tau,\z}
	\left(
		\hat{\bx}
		,
		\widetilde{\mathbf{u}}
	\right)
	\coloneqq
	\mathsf{k}_{\tau,\z}
	\left(
		\hat{\bx}
		,
		\widetilde{\mathbf{u}}
	\right)
	\sqrt{g_{\tau,\z}(\widetilde{\mathbf{u}})}
\end{equation}
for each $(\hat{\bx},\widetilde{\mathbf{u}})\in\hat{\Gamma} \times \widetilde{\tau}$
with $\hat{\bx} \neq \chi_{\tau}(\widetilde{\mathbf{u}})$,
$\tau \in \mathcal{G}$, and $\z \in \mathcal{O}_{\boldsymbol{\rho}}$ and define 
\begin{equation}
	\left(
		\hat{\OA}_{\tau,\z}
		\,
		\hat{\varphi} 
	\right)
	\left(
		\hat{\bx}
	\right)
	\coloneqq
	\int\limits_{\widetilde{\tau}} 
	\mathsf{b}_{\tau,\z}
	(\hat\bx,\hat{{\bf u}})
	\left(
		\hat{\varphi}
		\circ
		\chi_\tau
	\right)
	(\widetilde{\mathbf{u}})
	\,
	\text{d} 
	\widetilde{\mathbf{u}},
\end{equation}
for $\hat{\varphi} \in L^2(\hat\Gamma)$. 
Using the Riesz-Thorin interpolation introduced in 
\cref{sec:rt_theorem} we obtain for each $\tau \in \mathcal{G}$
\begin{equation}
	\norm{
		\hat{\OA}_{\tau,\z}
	}_{\mathscr{L}\left(L^2(\hat{\Gamma}),L^2(\hat{\Gamma})\right)}
   	\leq
    	\norm{
		\hat{\OA}_{\tau,\z}
	}_{\mathscr{L}\left(L^1(\hat{\Gamma}), L^1(\hat{\Gamma})\right)}^{\half}
    	\norm{
		\hat{\OA}_{\tau,\z}
	}_{\mathscr{L}\left(L^\infty(\hat{\Gamma}), L^\infty(\hat{\Gamma})\right)}^{\half}.
\end{equation}
\cref{lmm:lipschitz_preservation} yields for each 
$\tau \in \mathcal{G}$ and
for any $\z \in \mathcal{O}_{\boldsymbol{\rho}}$
\begin{equation}\label{eq:geteta}
\begin{aligned}
	\norm{
		\boldsymbol{r}_{\z}(\hat\bx)
		-
		\boldsymbol{r}_\z \circ \chi_\tau \left( \widetilde{\mathbf{u}}\right)
	}^2
	&
	\geq	
	\Re\{
		\left(
			\boldsymbol{r}_{\z}(\hat\bx)
			-
			\boldsymbol{r}_\z \circ \chi_\tau \left( \widetilde{\mathbf{u}}\right)
		\right)
		\cdot
		\left(
			\boldsymbol{r}_{\z}(\hat\bx)
			-
			\boldsymbol{r}_{\z} \circ \chi_\tau \left( \widetilde{\mathbf{u}}\right)
		\right)
	\} \\
	&
	\geq
	\eta
	\norm{
		\hat\bx
		-
		\chi_\tau \left( \widetilde{\mathbf{u}}\right)
	}^2\\
    &
	>0.
\end{aligned}
\end{equation}
This, together with \cref{thm:holomorphy_sing_kernel3} from the assumptions, yields
\begin{equation}
\begin{aligned}
	\snorm{
		\mathsf{b}_{\tau,\z}
		(\hat\bx,\hat{\bf u}) 
	}
	&
	\leq
	\frac{
		C_{\mathsf{a}}
	}{
		\norm{
			\boldsymbol{r}_{\boldsymbol{z}}(\hat\bx)
			-
			\boldsymbol{r}_\z \circ \chi_\tau \left( \widetilde{\mathbf{u}}\right)
		}^\nu
	}
	\snorm{
		\sqrt{g_{\tau,\z}(\widetilde{\mathbf{u}})}
	}
	\\
	&
    	\leq
 	\frac{
	\widetilde{C}_{\mathsf{a}}
	}{
	\norm{
		\hat\bx
		-
		\chi_\tau \left( \widetilde{\mathbf{u}}\right)
	}^\nu
	}
	\snorm{
		\sqrt{g_{\tau,\z}(\widetilde{\mathbf{u}})}
	},
\end{aligned}
\end{equation}
with $\widetilde{C}_{\mathsf{a}} = C_{\mathsf{a}}\eta^{-\nu/2}$ and for any
$(\hat\bx,\hat{\bf u}) \in \hat{\Gamma} \times \widetilde{\tau}$
with $\hat\bx \neq \chi_\tau \left( \widetilde{\mathbf{u}}\right)$. 
Thus, for each $\z\in\mathcal{O}_{\boldsymbol\rho}$, it holds
\begin{equation}
\begin{aligned}
   	 \norm{
	 	\hat{\OA}_{\tau,\z}
	}_{\mathscr{L}\left(L^1(\hat{\Gamma}),L^1(\hat{\Gamma})\right)}
    	&
	=
  	\underset{\widetilde{\mathbf{u}} \in \widetilde{\tau}}{\operatorname{ess} \sup} 
	\int\limits_{\hat\Gamma}
	\snorm{
	\mathsf{b}_{\tau,\z}
	\left(
		\hat{\bx}
		,
		\widetilde{\mathbf{u}}
	\right)
	}	
	\text{d} \text{s}_{\hat{\mathbf{x}}}\\
	&\leq
	\widetilde{C}_{\mathsf{a}}
	\,
	\displaystyle
	\underset{\widetilde{\mathbf{u}} \in \widetilde{\tau}}{\operatorname{ess} \sup} 
	\snorm{
		\sqrt{g_{\tau,\z}(\widetilde{\mathbf{u}})}
	}\,
      	\underset{\widetilde{\mathbf{u}} \in \widetilde{\tau}}{\operatorname{ess} \sup} 
	\int\limits_{\hat\Gamma}
	\frac{
		1
	}{
	\norm{
		\hat\bx
		-
		\chi_\tau \left( \widetilde{\mathbf{u}}\right)
	}^\nu
	}
   	 \text{d} \text{s}_{\hat{\mathbf{x}}},
\end{aligned}
\end{equation}
where the latter integral is finite for any
$\widetilde{\mathbf{u}}\in\widetilde{\tau}$ since
$\nu \in (0,2)$. 

Furthermore, for each $\z\in\mathcal{O}_{\boldsymbol\rho}$, it holds
\begin{equation}
\begin{aligned}
   	\norm{
		\hat{\OA}_{\tau,\z}
	}_{\mathscr{L}\left(L^\infty(\hat{\Gamma}), L^\infty(\hat{\Gamma})\right)}
   	&
	=
   	\underset{\hat{\mathbf{x}} \in \hat\Gamma}{\operatorname{ess} \sup} 
	\int\limits_{\widetilde{\tau}}
	\snorm{
	\mathsf{b}_{\tau,\z}
	\left(
		\hat{\bx}
		,
		\widetilde{\mathbf{u}}
	\right)
	}	
	\text{d} 
	\widetilde{\mathbf{u}}\\
	&\leq
	\widetilde{C}_{\mathsf{a}}
	\underset{\hat{\bf u} \in \widetilde{\tau}}{\operatorname{ess} \sup}
	\snorm{
		\sqrt{g_{\tau,\z}(\widetilde{\mathbf{u}})}
	}\,
   	\underset{\hat{\mathbf{x}} \in \hat\Gamma}{\operatorname{ess} \sup} 
	\int\limits_{\widetilde{\tau}}
	\frac{
		1
	}{
	\norm{
		\hat\bx
		-
		\chi_\tau \left( \widetilde{\mathbf{u}}\right)
	}^\nu
	}
 	\text{d} 
	\widetilde{\mathbf{u}}\\
	&\leq
	\widetilde{C}_{\mathsf{a}}
	\underset{\hat{\bf u} \in \widetilde{\tau}}{\operatorname{ess} \sup}
	\snorm{
		\sqrt{g_{\tau,\z}(\widetilde{\mathbf{u}})}
	}\,
   	\underset{\hat{\mathbf{x}} \in \hat\Gamma}{\operatorname{ess} \sup} 
	\int\limits_{\hat\Gamma}
	\frac{
		1
	}{
	\norm{
		\hat\bx
		-
		\hat{\by}
	}^\nu
	}
    \text{d}\text{s}_{\hat{\by}},
\end{aligned}
\end{equation}
where the assumptions on $\nu$ imply again that latter integral is
uniformly bounded for all $\hat{\bx}\in\hat{\Gamma}$.
This shows that $\hat{\OA}_{\tau,\z}\colon L^2(\hat{\Gamma})\to L^2(\hat{\Gamma})$
is a bounded linear operator and allows us to define for each ${\tau} \in 
\mathcal{G}$
\begin{equation}\label{eq:singularAtau}
	\mathcal{A}_\tau:
	\mathcal{O}_{\boldsymbol{\rho}}
	\rightarrow
	\mathscr{L}
	\left(
		L^2(\hat\Gamma)
		,
		L^2(\hat\Gamma)
	\right):
	\z
	\mapsto
	\hat{\mathsf{A}}_{\tau,\z}.
\end{equation}

{\bf Step \encircled{\sf A'}: Regularization of localized operators $\hat{\OA}_{\tau,\z}$.}
Let $\phi :[0,\infty)\rightarrow \IR$ be a $\mathscr{C}^1([0,\infty))$ function
with
\begin{equation}
\phi(t)
\begin{cases}
=0&t\in[0,\frac{1}{2}],\\
\in[0,1]& t\in(\frac{1}{2},1),\\
=1&t\in[1,\infty).
\end{cases}
\end{equation}
For each $n\in \IN$, $\tau \in \mathcal{G}$, and $\z\in \mathcal{O}_{\boldsymbol{\rho}}$ define
\begin{equation}\label{eq:btauzn}
	\mathsf{b}^{(n)}_{\tau,\z}
	\left(
		\hat{\bx}
		,
		\widetilde{\mathbf{u}}
	\right)
	\coloneqq
	\mathsf{b}_{\tau,\z}
    \left(
		\hat{\bx}
		,
		\widetilde{\mathbf{u}}
	\right)
	\phi(n
		\norm{\hat{\mathbf{x}}-\chi_\tau
		\left(
			\widetilde{\mathbf{u}}
		\right)}^\nu )
\end{equation}
and
\begin{equation}
	\left(
		\hat{\OA}^{(n)}_{\tau,\z}
		\,
		\hat{\varphi} 
	\right)
	\left(
		\hat{\bx}
	\right)
	\coloneqq
	\int\limits_{\widetilde{\tau}} 
	\mathsf{b}^{(n)}_{\tau,\z}
	\left(
		\hat{\bx}
		,
		\widetilde{\mathbf{u}}
	\right)
	\left(
		\hat{\varphi}
		\circ
		\chi_\tau
	\right)
	(\widetilde{\mathbf{u}}) 
	\,
	\text{d} 
	\widetilde{\mathbf{u}},
	\quad
	\hat\bx\in\hat\Gamma,
\end{equation}
for all $\hat{\varphi} \in L^2(\hat\Gamma)$. Using
the Riesz-Thorin interpolation theorem \cref{eq:rieszthorin}
we obtain
\begin{equation}\label{eq:rt_convergence}
\begin{aligned}
	&
	\norm{
		\hat{\OA}_{\tau,\z}-\hat{\OA}_{\tau,\z}^{(n)}
	}_{\mathscr{L}(L^2(\hat{\Gamma}),L^2(\hat{\Gamma}))}
	\\
	&
	\hspace{1.5cm}
	\leq
	\norm{
		\hat{\OA}_{\tau,\z}-\hat{\OA}_{\tau,\z}^{(n)}
	}_{\mathscr{L}(L^1(\hat{\Gamma}),L^1(\hat{\Gamma}))}^{\half}
	\norm{
		\hat{\OA}_{\tau,\z}-\hat{\OA}_{\tau,\z}^{(n)}
	}_{\mathscr{L}(L^\infty(\hat{\Gamma}),L^\infty(\hat{\Gamma}))}^{\half}.
\end{aligned}
\end{equation}
To estimate this quantity, we note that for $\hat\bx\in\hat\Gamma$
it holds
\begin{equation}
\begin{aligned}
	&\left(
		\hat{\OA}_{\tau,\z}
		\,
		\hat{\varphi} 
	\right)
	\left(
		\hat{\bx}
	\right)
	-
	\left(
		\hat{\OA}^{(n)}_{\tau,\z}
		\,
		\hat{\varphi} 
	\right)
	\left(
		\hat{\bx}
	\right)\\
	&\qquad\qquad\qquad=
	\int\limits_{\widetilde{\tau}} 
	\mathsf{b}_{\tau,\z}
	\left(
		\hat{\bx}
		,
		\widetilde{\mathbf{u}}
	\right)
	\left(
		1
		-
		\phi(
			n
			\norm{\hat{\mathbf{x}}
			-
			\chi_\tau
			\left(
				\widetilde{\mathbf{u}}
			\right)}^\nu 
		)
	\right)
	\left(
		\hat{\varphi}
		\circ
		\chi_\tau
	\right)
	(\widetilde{\mathbf{u}}) 
	\,
	\text{d} 
	\widetilde{\mathbf{u}}.
\end{aligned}
\end{equation}
Therein, along the lines of step \encircled{\sf A}, the integrand can
be estimated as
\begin{equation}
\begin{aligned}
\snorm{
	\mathsf{b}_{\tau,\z}
	\left(
		\hat{\bx}
		,
		\widetilde{\mathbf{u}}
	\right)
	\left(
		1
		-
		\phi(
			n
			\norm{\hat{\mathbf{x}}
			-
			\chi_\tau
			\left(
				\widetilde{\mathbf{u}}
			\right)}^\nu 
		)
	\right)}
    	\leq
	\widetilde{C}_{\mathsf{a}}
    	\frac{
	\left(
		1
		-
		\phi(
			n
			\norm{\hat{\mathbf{x}}
			-
			\chi_\tau
			\left(
				\widetilde{\mathbf{u}}
			\right)}^\nu 
		)
    	\right)
	}{
	\norm{
		\hat\bx
		-
		\chi_\tau \left( \widetilde{\mathbf{u}}\right)
	}^\nu
	}
	\snorm{
		\sqrt{g_{\tau,\z}(\widetilde{\mathbf{u}})}
	},
\end{aligned}
\end{equation}
for $(\hat\bx,\hat{\bf u}) \in \hat{\Gamma} \times \widetilde{\tau}$
with $\hat\bx \neq \chi_\tau \left( \widetilde{\mathbf{u}}\right)$.
Thus, it holds
\begin{equation}\label{eq:aprimeintegral1}
\begin{aligned}
    	&
	\norm{
		\hat{\OA}_{\tau,\z}-\hat{\OA}_{\tau,\z}^{(n)}
	}_{\mathscr{L}\left(L^1(\hat{\Gamma}), L^1(\hat{\Gamma})\right)}\\
    	&\qquad\qquad
	=
	\underset{\widetilde{\mathbf{u}} \in \widetilde{\tau}}{\operatorname{ess} \sup} 
	\int\limits_{\hat\Gamma}
	\snorm{
	\mathsf{b}_{\tau,\z}
	\left(
		\hat{\bx}
		,
		\widetilde{\mathbf{u}}
	\right)
	}
	\left(
		1
		-
		\phi(
			n
			\norm{\hat{\mathbf{x}}
			-
			\chi_\tau
			\left(
				\widetilde{\mathbf{u}}
			\right)}^\nu 
		)
	\right)	
	\text{ds}_{
	\hat{\mathbf{x}}
	}
	\\
	&\qquad\qquad
	\leq
	\widetilde{C}_{\mathsf{a}}
	\underset{\widetilde{\mathbf{u}} \in \widetilde{\tau}}{\operatorname{ess} \sup} 
	\snorm{
		\sqrt{g_{\tau,\z}(\widetilde{\mathbf{u}})}
	}
	\,
	\underset{\hat{\by} \in \hat\Gamma}{\operatorname{ess} \sup} 
	\int\limits_{\hat\Gamma}
	\frac{
		1
		-
		\phi(
			n
			\norm{\hat{\mathbf{x}}
			-
			\hat{\mathbf{y}}
			}^\nu 
		)
	}{
	\norm{
		\hat{\mathbf{x}}
		-
		\hat{\mathbf{y}}
	}^\nu
	}	
	\text{ds}_{
	\hat{\mathbf{x}}
	}.
\end{aligned}
\end{equation}
On the other hand, it holds
\begin{equation}\label{eq:aprimeintegral2}
\begin{aligned}
    	&
	\norm{
		\hat{\OA}_{\tau,\z}-\hat{\OA}_{\tau,\z}^{(n)}
	}_{\mathscr{L}\left(L^{\infty}(\hat{\Gamma}), L^{\infty}(\hat{\Gamma})\right)}\\
	&\qquad\qquad\leq\underset{\hat{\bx} \in \hat\Gamma}{\operatorname{ess} \sup} \int\limits_{\widetilde{\tau}} 
	\snorm{
	\mathsf{b}_{\tau,\z}
	\left(
		\hat{\bx}
		,
		\widetilde{\mathbf{u}}
	\right)
	\left(
		1
		-
		\phi(
			n
			\norm{\hat{\mathbf{x}}
			-
			\chi_\tau
			\left(
				\widetilde{\mathbf{u}}
			\right)}^\nu 
		)
	\right)
	}
	\text{d} 
	\widetilde{\mathbf{u}}\\
	&\qquad\qquad
	\leq
	\widetilde{C}_{\mathsf{a}}
	\underset{\widetilde{\mathbf{u}} \in \widetilde{\tau}}{\operatorname{ess} \sup} 
	\snorm{
		\sqrt{g_{\tau,\z}(\widetilde{\mathbf{u}})}
	}
   	\,
	\underset{\hat{\bx} \in \hat\Gamma}{\operatorname{ess} \sup}
	\int\limits_{\widetilde{\tau}}
	\frac{
		1
		-
		\phi(
			n
			\norm{\hat{\mathbf{x}}
			-
			\chi_\tau
			\left(
				\widetilde{\mathbf{u}}
			\right)}^\nu 
		)
	}{
	\norm{
		\hat\bx
		-
		\chi_\tau \left( \widetilde{\mathbf{u}}\right)
	}^\nu
	}
	\text{d} 
	\widetilde{\mathbf{u}}
	\\
	&\qquad\qquad
	\leq
	\widetilde{C}_{\mathsf{a}}
	\underset{\widetilde{\mathbf{u}} \in \widetilde{\tau}}{\operatorname{ess} \sup} 
	\snorm{
		\sqrt{g_{\tau,\z}(\widetilde{\mathbf{u}})}
	}
   	\underset{\hat{\bx} \in \hat\Gamma}{\operatorname{ess} \sup}
	\int\limits_{\hat{\Gamma}}
	\frac{
		1
		-
		\phi(
			n
			\norm{\hat{\mathbf{x}}
			-
			\hat{\by}}^\nu 
		)
	}{
	\norm{
		\hat\bx
		-
		\hat{\by}
	}^\nu
	}
	\text{d} \text{s}_{\hat{\by}}.
\end{aligned}
\end{equation}
Upon noting that the integrals in \cref{eq:aprimeintegral1} and \cref{eq:aprimeintegral2}
are equal, it remains to estimate the essential supremum of this integral.
Let us calculate
\begin{equation}
\begin{aligned}
    \underset{\hat{\bx} \in \hat\Gamma}{\operatorname{ess} \sup}
	\int\limits_{\hat{\Gamma}}
	\frac{
		1
		-
		\phi(
			n
			\norm{\hat{\mathbf{x}}
			-
			\hat{\by}}^\nu 
		)
	}{
	\norm{
		\hat\bx
		-
		\hat{\by}
	}^\nu
	}
	\text{d} \text{s}_{\hat{\by}}
    &\leq
    \underset{\hat{\bx} \in \hat\Gamma}{\operatorname{ess} \sup}
	\int\limits_{\hat{\Gamma}\cap B\left(\hat{\bx},n^{-\frac{1}{\nu}}\right)}
	\frac{
		1
	}{
	\norm{
		\hat\bx
		-
		\hat{\by}
	}^\nu
	}
	\text{d} \text{s}_{\hat{\by}}\\
    &\leq
    C\int_{B\left(\hat{\bx},n^{-\frac{1}{\nu}}\right)}
    \frac{1}{\|\widetilde{\mathbf{u}}\|^\nu}
    \text{d}\widetilde{\mathbf{u}}\\
    &
    \leq C
\int\limits_{0}^{n^{-\frac{1}{\nu}}}
	\int\limits_{0}^{2\pi}
    r^{1-\nu}
	\text{d} 
	r
	\text{d} 
	\theta\\
    &
    \leq C
    \frac{2\pi}{2-\nu}n^{\frac{\nu-2}{\nu}},
\end{aligned}
\end{equation}
where the constant $C>0$ depends only on $\hat\Gamma$,
and not on $n$ or $\nu$.

Next, recalling \cref{eq:rt_convergence} we get
\begin{equation}\label{eq:unif_conv}
	\sup_{
		\z 
		\in 
		\mathcal{O}_{\boldsymbol{\rho}}
	}
	\norm{
		\hat{\OA}_{\tau,\z}-\hat{\OA}_{\tau,\z}^{(n)}
	}_{\mathscr{L}\left(L^2(\hat{\Gamma}), L^2(\hat{\Gamma})\right)}
	\leq
	C
	\widetilde{C}_{\mathsf{a}}
	\frac{2\pi}{2-\nu}
	\underbrace{
	\sup_{
		\z 
		\in 
		\mathcal{O}_{\boldsymbol{\rho}}
	}
	\underset{\widetilde{\mathbf{u}} \in \widetilde{\tau}}{\operatorname{ess} \sup} 
	\snorm{
		\sqrt{g_{\tau,\z}(\widetilde{\mathbf{u}})}
	}
	}_{(\spadesuit)}
	n^{\frac{\nu-2}{\nu}}.
\end{equation}
As a consequence of \cref{lemma:gramian_holomorphic},
the term $(\spadesuit)$ in \eqref{eq:unif_conv} is
bounded. 
Recalling that we have assumed $\nu \in (0,2)$ implies
\begin{equation}
	\lim_{n\rightarrow \infty}
	\sup_{
		\z 
		\in 
		\mathcal{O}_{\boldsymbol{\rho}}
	}
	\norm{
		\hat{\OA}_{\tau,\z}-\hat{\OA}_{\tau,\z}^{(n)}
	}_{\mathscr{L}\left(L^2(\hat{\Gamma}), L^2(\hat{\Gamma})\right)}
	=
	0.
\end{equation}
This implies that for each $\tau\in\mathcal{G}$
\begin{equation}\label{eq:Ataun}
	\mathcal{A}_\tau^{(n)}:
	\mathcal{O}_{\boldsymbol{\rho}}
	\rightarrow
	\mathscr{L}
	\left(
		L^2(\hat\Gamma)
		,
		L^2(\hat\Gamma)
	\right):
	\z
	\mapsto
	\hat{\mathsf{A}}_{\tau,\z}^{(n)}
\end{equation}
converges uniformly in $\mathcal{O}_{\boldsymbol{\rho}}$ 
to $\mathcal{A}_\tau$ from \cref{eq:singularAtau}.

{\bf Step \encircled{\sf B}: $(\boldsymbol{b},p,\varepsilon)$-holomorphy of localized, regularized operators $\hat{\OA}_{\tau,\y}^{(n)}$}.
We aim to verify the assumptions of \cref{thm:holomorphy_bounded_kernel}
to show that
\begin{equation}\label{eq:Ataunr}
	\mathcal{A}_\tau^{(n)}:
	\mathbb{U}
	\rightarrow
	\mathscr{L}
	\left(
		L^2(\hat\Gamma)
		,
		L^2(\hat\Gamma)
	\right):
	\y
	\mapsto
	\hat{\mathsf{A}}_{\tau,\y}^{(n)}
\end{equation}
is $(\boldsymbol{b},p,\varepsilon)$-holomorphic
with $\varepsilon>0$ \emph{independent} of $n \in \IN$.
Thus, by setting
\begin{equation}
	\mathsf{k}_{\tau,\y}^{(n)}
	\left(
		\hat{\bx}
		,
		\widetilde{\mathbf{u}}
	\right)
	\coloneqq
	\mathsf{k}_{\tau,\y}
	\left(
		\hat{\bx}
		,
		\widetilde{\mathbf{u}}
	\right)\phi(
		n
		\norm{\hat{\mathbf{x}}-\chi_\tau
		\left(
			\widetilde{\mathbf{u}}
		\right)}^\nu )
\end{equation}
with $\mathsf{k}_{\tau,\y}$ from \cref{eq:singulark} we need to show that
\begin{equation}\label{eq:ktaunr}
	\mathbb{U}
	\ni
	\y
	\mapsto
	\mathsf{k}_{\tau,\y}^{(n)}
	\in
	L^{\infty}
	\left(
		\hat{\Gamma} \times \widetilde{\tau}
	\right)
\end{equation}
is $(\boldsymbol{b},p,\varepsilon)$-holomorphic
with $\varepsilon>0$, again,  independent of $n$ and continuous.
To this end, due to \cref{eq:geteta}, for each fixed
$(\hat\bx,\hat{\bf u}) \in \hat{\Gamma} \times \widetilde{\tau}$ such
that $\bx \neq 	\chi_\tau\left(\widetilde{\mathbf{u}} \right)$
and any $\z \in \mathcal{O}_{\boldsymbol{\rho}}$ it holds
\begin{equation}
	\left(
		\boldsymbol{r}_{\z}(\hat\bx)
		-
		\boldsymbol{r}_\z \circ \chi_\tau \left( \widetilde{\mathbf{u}}\right)
	\right)
	\cdot
	\left(
		\boldsymbol{r}_{\z}(\hat\bx)
		-
		\boldsymbol{r}_{\z} \circ \chi_\tau \left( \widetilde{\mathbf{u}}\right)
	\right)
	\in 
	\mathfrak{U}.
\end{equation}
Thus, due to \cref{thm:holomorphy_sing_kernel3} in the assumptions
and \cref{eq:geteta},
it follows that
\begin{equation}
\begin{aligned}
	\sup_{\z\in \mathcal{O}_{\boldsymbol{\rho}}}
	\snorm{
		\mathsf{k}^{(n)}_{\tau,\z}
		(\hat\bx,\hat{\bf u}) 
	}
	&\leq
    	C_{\mathsf{a}}
    	\sup_{\z\in \mathcal{O}_{\boldsymbol{\rho}}}
	\frac{\phi(
		n
		\norm{\hat{\mathbf{x}}-\chi_\tau
		\left(
			\widetilde{\mathbf{u}}
		\right)}^\nu )
	}{
		\norm{
			\boldsymbol{r}_{\boldsymbol{z}}(\hat\bx)
			-
			\boldsymbol{r}_\y \circ \chi_\tau \left( \widetilde{\mathbf{u}}\right)
		}^\nu
	}\\
    &\leq
    \widetilde{C}_{\mathsf{a}}
    \frac{\phi(
    		n
    		\norm{\hat{\mathbf{x}}-\chi_\tau
		\left(
			\widetilde{\mathbf{u}}
		\right)}^\nu )}{\norm{\hat{\bx}-\chi_\tau(\widetilde{\mathbf{u}})}^\nu}\\
    &\leq
    \widetilde{C}_{\mathsf{a}}
	n
	\sup_{t\in [0,\infty)}
	\snorm{\phi'(t)}.
 \end{aligned}
\end{equation}
Thus,
\begin{equation}\label{eq:kernel_n_bound}
	\sup_{\z\in \mathcal{O}_{\boldsymbol{\rho}}}\norm{
		\mathsf{k}^{(n)}_{\tau,\z}
	}_{L^\infty( \hat{\Gamma} \times \widetilde{\tau})}
	\leq
	\widetilde{C}_{\mathsf{a}}
	n
	\sup_{t\in [0,\infty)}
	\snorm{\phi'(t)}.
\end{equation}
This allows to verify \cref{def:bpe_holomorphy1} and
\cref{def:bpe_holomorphy3} in the \cref{def:bpe_holomorphy} of 
$(\boldsymbol{b},p,\varepsilon)$-holomorphy 
and it only remains
to verify \cref{def:bpe_holomorphy2}.

Unfortunately, we cannot proceed as in step \encircled{\sf B} of the proof of
\cref{thm:holomorphy_bounded_kernel} as
$\mathsf{k}_{\tau,\z}\notin L^{\infty}\left(\hat{\Gamma}\times\widetilde{\tau}\right)$
due to the singularity of the kernel function and, 
thus, $\mathcal{O}_{\boldsymbol{\rho}} \ni \z\mapsto\mathsf{k}_{\tau,\z} \in L^{\infty}\left(\hat{\Gamma}\times\widetilde{\tau}\right)$
is not a holomorphic map.
Instead, we observe that
for each $(\hat\bx,\hat{\bf u}) \in \hat{\Gamma} \times \widetilde{\tau}$ the map
$\mathbb{U}\ni \y \mapsto \boldsymbol{r}_\y(\hat{\mathbf{x}}) -
\boldsymbol{r}_\y \circ \chi_\tau \left( \widetilde{\mathbf{u}}\right) \in \IC$
is $(\boldsymbol{b},p,\varepsilon)$-holomorphic 
as a consequence of \cref{lmm:holomorphic_parametric_rep},
with $\varepsilon>0$ independent of
$(\hat\bx,\hat{\bf u}) \in \hat{\Gamma} \times \widetilde{\tau}$.
Consequently, recalling \cref{thm:holomorphy_sing_kernel1} and
\cref{thm:holomorphy_sing_kernel2} in the assumptions,
we may conclude that for each fixed
$(\hat\bx,\hat{\bf u}) \in \hat{\Gamma} \times \widetilde{\tau}$ such
$\hat\bx \neq 	\chi_\tau\left(\widetilde{\mathbf{u}} \right)$
that the map
\begin{align}
	\mathbb{U}
	\ni
	\y
	\mapsto
	\mathsf{k}_{\tau,\y}(\hat\bx,\hat{\bf u})
	\in 
	\IC
\end{align}
is $(\boldsymbol{b},p,\varepsilon)$-holomorphic
for some $\varepsilon>0$ that does not depend on 
$(\hat\bx,\hat{\bf u}) \in \hat{\Gamma} \times \widetilde{\tau}$.
To show parametric holomorphy on $\hat{\Gamma} \times \widetilde{\tau}$,
let $\z \in \mathcal{O}_{\boldsymbol{\rho}}$, $j\in \IN$
and consider $h \in \IC$ small enough such that 
$\z + h \boldsymbol{e}_j \in \mathcal{O}_{\boldsymbol{\rho}}$. 
Using Taylor's expansion in the direction $z_j$ we get
\begin{equation}\label{eq:singktaylor}
\begin{aligned}
	\mathsf{k}^{(n)}_{\tau,\z + h \boldsymbol{e}_j}
	(\hat\bx,\hat{\bf u})
	&=
	\mathsf{k}^{(n)}_{\tau,\z}
	(\hat\bx,\hat{\bf u})
	+
	h
	\left(
		\partial_{z_j}
		\mathsf{k}_{\tau,\z}
	\right)\!
	(\hat\bx,\hat{\bf u})
	\phi(
		n
		\norm{\hat{\mathbf{x}}-\chi_\tau
	\left(\widetilde{\mathbf{u}}\right)}^\nu)\\
	&
	\qquad+
	h^2
	\int\limits_{0}^{1}
	(1-\eta)
	\left(
		\partial^2_{z_j}
		\mathsf{k}_{\tau,\z+\eta h \boldsymbol{e}_j}
	\right)\!
	(\hat\bx,\hat{\bf u})
	\phi(
		n
		\norm{\hat{\mathbf{x}}-\chi_\tau
	\left(\widetilde{\mathbf{u}}\right)}^\nu)
	\dd\eta.
\end{aligned}
\end{equation}

Let $\widetilde{\rho}_j<{\rho}_j$.
 For each $\z \in \mathcal{O}_{\widetilde{\boldsymbol{\rho}}}$
 with $\widetilde{\boldsymbol{\rho}} = (\rho_1,\rho_2,\dots,\widetilde\rho_{j},\dots)$
 there exists $\y\in \mathbb{U}$ such that
 $\snorm{z_j-y_j}<\widetilde\rho_j-1$. 
 Consider $h\in \IC$ satisfying
 $0<\snorm{h}<{\rho}_j -\widetilde{\rho}_j$. Set
 \begin{equation}
 	\zeta_h
 	=
 	{\rho}_j -\widetilde{\rho}_j
 	-
 	\snorm{h}
 \end{equation}
 We need to ensure
 $B(z_j+\eta h,\zeta_h) \subset \mathcal{O}_{\rho_j}$ for any $\eta \in [0,1]$.
 For $\widetilde{z}_j \in B(z_j+\eta h,\zeta_h) $ we have
 \begin{equation}\label{eq:eta_h}
 \begin{aligned}
 	\snorm{
 		\widetilde{z}_j
 		-
 		y_j
 	}
 	&
 	\leq
 	\snorm{
 		\widetilde{z}_j
 		-
 		z_j-\eta h
 	}
 	+
 	\snorm{h}
 	+
 	\snorm{
 		z_j
 		-
 		y_j
 	}
 	\\
 	&
 	\leq
 	\zeta_h
 	+
 	\snorm{h}
 	+
 	\widetilde\rho_j-1
 	\leq
 	\rho_j-1.
 \end{aligned}
 \end{equation}
Thus, for $\zeta_{h}$ as in \eqref{eq:eta_h} Cauchy's integral formula yields
\begin{equation}
    \begin{aligned}
	&\left(
		\partial^2_{z_j}
		\mathsf{k}_{\tau,\z+\eta h \boldsymbol{e}_j}
	\right)
	(\hat\bx,\hat{\bf u})
   	\phi(
		n
		\norm{\hat{\mathbf{x}}-\chi_\tau
	\left(\widetilde{\mathbf{u}}\right)}^\nu)\\
	&\qquad\qquad=
	\frac{1}{\pi \imath}
	\int\limits_{z'_j\in \partial B(z_j+h\eta,\zeta_{h})}
	\frac{
		\mathsf{k}_{\tau,\z'}
		(\hat\bx,\hat{\bf u})
	}{
		(z'_j-z_j-\eta h)^3
	}
    \phi(
    	n
	\norm{\hat{\mathbf{x}}-\chi_\tau
	\left(\widetilde{\mathbf{u}}\right)}^\nu)
	\dd z'_j,
    \end{aligned}
\end{equation}
with $\z' =(z_1,\dots,z_j',\dots)$, and
\begin{equation}
\begin{aligned}
	\snorm{
	\left(
		\partial^2_{z_j}
		\mathsf{k}_{\tau,\z+\eta h \boldsymbol{e}_j}
	\right)
	(\hat\bx,\hat{\bf u})
    	\phi(
		n
		\norm{\hat{\mathbf{x}}-\chi_\tau
	\left(\widetilde{\mathbf{u}}\right)}^\nu)
	}
	\leq
    	\frac{2}{\zeta_{h}^2}
	\sup_{\z \in \mathcal{O}_{\boldsymbol{\rho}}}
	\snorm{
		\mathsf{k}_{\tau,\z}^{(n)}
		(\hat\bx,\hat{\bf u})
	}.
\end{aligned}
\end{equation}

Hence, from \cref{eq:singktaylor},
\begin{equation}
	\norm{
	\frac{\mathsf{k}^{(n)}_{\tau,\z + h \boldsymbol{e}_j}
	-
	\mathsf{k}^{(n)}_{\tau,\z}}{h}
	-
	\mathsf{g}^{(n)}_{\tau,\z}
	}_{	L^\infty
	\left(
		\hat{\Gamma} 
		\times 
		\widetilde{\tau}
	\right)}
	\leq
    	\frac{\snorm{h}}{\zeta_{h}^2}
	\sup_{\z \in \mathcal{O}_{\boldsymbol{\rho}}}
	\norm{
		\mathsf{k}_{\tau,\z}^{(n)}
	}_{L^\infty(\hat{\Gamma}\times\widetilde{\tau})},
\end{equation}
where $\mathsf{g}^{(n)}_{\tau,\z} \in	 L^{\infty}
	\left(
		\hat{\Gamma} \times \widetilde{\tau}
	\right) $
is defined for each $n\in \IN$ as 
\begin{equation}
	\mathsf{g}^{(n)}_{\tau,\z}
	(\hat\bx,\hat{\bf u})
	\coloneqq
	\left(
		\partial_{z_j}
		\mathsf{k}_{\tau,\z}
	\right)\!
	(\hat\bx,\hat{\bf u})
	\phi(
		n
	\norm{\hat{\mathbf{x}}-\chi_\tau
	\left(\widetilde{\mathbf{u}}\right)}^\nu),
\end{equation}
and using \cref{eq:kernel_n_bound} we obtain
\begin{equation}
	\norm{
	\frac{\mathsf{k}^{(n)}_{\tau,\z + h \boldsymbol{e}_j}
	-
	\mathsf{k}^{(n)}_{\tau,\z}}{h}
	-
	\mathsf{g}^{(n)}_{\tau,\z}
	}_{	L^\infty
	\left(
		\hat{\Gamma} 
		\times 
		\widetilde{\tau}
	\right)}
	\leq
	C
	\widetilde{C}_{\mathsf{a}}
    	\frac{\snorm{h}}{\zeta^2_h}
	n
	\sup_{t\in [0,\infty)}
	\snorm{\phi'(t)}.
\end{equation}
Hence, for each $n \in \mathbb{N}$, the map
$
	\mathcal{O}_{\widetilde{\boldsymbol{\rho}}}
	\ni
	\z
	\mapsto
	\mathsf{k}^{(n)}_{\tau,\z}
	\in
	 L^{\infty}\left(\hat{\Gamma}\times\widetilde{\tau}\right)
$
is holomorphic at each $\z \in \mathcal{O}_{\widetilde{\boldsymbol{\rho}}}$ 
in the direction $z_j$.

To show that $\mathbb{U} \ni \y\mapsto\mathsf{k}^{(n)}_{\tau,\y} \in L^{\infty}\left(\hat{\Gamma}\times\widetilde{\tau}\right)$
is $(\boldsymbol{b},p,\varepsilon)$-holomorphic,
we observe that for any $\rho>1$ one has 
\begin{equation}
	\mathcal{O}_{{{\rho}}}
	= 
	\bigcup_{\widetilde{\rho}:
	\widetilde{\rho}<\rho} 
	\mathcal{O}_{\widetilde{{\rho}}}
\end{equation}
and one can extend the aforementioned property to
$\z \in \mathcal{O}_{{\boldsymbol{\rho}}}$.
Therefore, we have verified \cref{def:bpe_holomorphy2} in 
\cref{def:bpe_holomorphy} and \cref{eq:linftyAt}
is $(\boldsymbol{b},p,\varepsilon)$-holomorphic
for each $\tau \in \mathcal{G}$ with $\varepsilon>0$
independent of $n\in \mathbb{N}$.

{\bf Step \encircled{\sf C}: $(\boldsymbol{b},p,\varepsilon)$-holomorphy of $\mathcal{A}$.}
Now, fix $\tau\in\mathcal{G}$, and consider
\[
	\mathcal{A}_\tau^{(n)}
	\colon
	\mathcal{O}_{\boldsymbol{\rho}}
	\to
	\mathscr{L}
	\left(
		L^2(\hat{\Gamma}), L^2(\hat{\Gamma})
	\right)
\]
with $\mathcal{A}_\tau^{(n)}$ as defined in \cref{eq:Ataun}. 
Then, \encircled{\sf A} and \encircled{\sf A'} show that $\mathcal{A}_\tau^{(n)}$ 
converges uniformly to $\mathcal{A}_\tau$ on $\mathcal{O}_{\boldsymbol{\rho}}$ with $\mathcal{A}_\tau$ as in \cref{eq:singularAtau}. 
 Step \encircled{\sf B} shows that $\mathcal{A}_\tau^{(n)}$ is $(\boldsymbol{b},p,\varepsilon)$-holomorphic with $\varepsilon>0$ independent of $n$. 
 Thus, \cref{lmm:herve} yields that $\mathcal{A}_\tau$ is $(\boldsymbol{b},p,\varepsilon)$-holomorphic.
 
In the same way as in Step \encircled{\sf C} in the proof of
\cref{thm:holomorphy_bounded_kernel}, we now conclude that
\begin{equation}
\mathcal{A}=\sum_{\tau\in\mathcal{G}}\mathcal{A}_{\tau}
\end{equation}
with $\mathcal{A}$ as in \cref{eq:bpe_hol_A_sing} is
$(\boldsymbol{b},p,\varepsilon)$-holomorphic for
some  $\varepsilon>0$,
which is the assertion.
\end{proof}

\subsection{Parametric Shape Holomorphy of the Parameter--to--Solution Map}
\label{sec:parameric_d2s_map}
A key consequence of the presently obtained parametric holomorphy
result of the parameter-to-operator map \cref{eq:param_to_operator} 
is that of the parameter-to-solution map \cref{eq:param_to_solution}.

\begin{corollary}\label{cor:par-to-sol-holomorphy}
Let \cref{assump:parametric_boundary} be satisfied
with $\boldsymbol{b} \in \ell^p(\IN)$ and $p\in (0,1)$,
and let the assumptions of \cref{thm:holomorphy_sing_kernel}
be fulfilled as well. 
In addition, assume that
$\hat{\OA}_\y \in
\mathscr{L}_{\normalfont{\text{iso}}}(L^2(\hat{\Gamma}),L^2(\hat{\Gamma}))$
for each $\y \in \mathbb{U}$ and that the map
$\mathbb{U} \ni \y \mapsto \hat{f}_{\y} \in L^2(\hat\Gamma)$
is $(\boldsymbol{b},p,\varepsilon)$-holomorphic
and continuous for some $\varepsilon>0$.
Then there exists $\varepsilon'>0$ such that
the parameter-to-solution map 
\begin{equation}
	\mathcal{S}:
	\mathbb{U}
    	\to
    	L^2(\hat{\Gamma}):
	\y
	\mapsto 
	u_\y
	\isdef
	\hat{\mathsf{A}}_{\y}^{-1} \hat{f}_\y,
\end{equation}
is $(\boldsymbol{b},p,\varepsilon')$-holomorphic and continuous.
\end{corollary}
\begin{proof}
The assertion then follows as a direct consequence from \cref{prop:hol_maps_banach_spaces2} and \cref{prop:hol_maps_banach_spaces3} in \cref{prop:hol_maps_banach_spaces}.
\end{proof}

\section{Applications to Boundary Integral Equations on Parametric Boundaries}
\label{sec:application}
In this section, we explore the applicability of the main result of 
\cref{sec:BIO}, i.e. \cref{thm:holomorphy_sing_kernel}, to the
boundary integral operators obtained from the boundary reduction
of the sound-soft Helmholtz scattering problem.

\subsection{Model Problem: Sound-Soft Acoustic Scattering}
\label{sec:soudn_soft}
Let $\D\subset \mathbb{R}^3$ be a
bounded Lipschitz domain with boundary $\Gamma:=\partial \D$
and denote by $\D^\cc:= \mathbb{R}^3 \backslash\overline{\D}$ 
the corresponding exterior domain. 
Given $\kappa>0$ and $\boldsymbol{\hat{d}}_{\text{inc}} \in \mathbb{S}^2
\coloneqq \{\bx \in \IR^3: \norm{\bx}=1\}$, 
we define an incident plane wave $u^\text{inc}(\bx)\coloneqq\exp(\imath
\kappa \bx \cdot \boldsymbol{\hat{d}}_{\text{inc}})$ 
and aim to find the scattered wave $u^\text{scat}\in H^1_\loc(\D^\cc)$
such that the total field $u\coloneqq u^\text{inc}+ u^\text{scat}$ satisfies
\begin{subequations}
\begin{align}
	\Delta u + \kappa^2 u &=0, \quad \text{in } \D^\cc, \label{eq:sound_soft}
	\\
	u &= 0, \quad \text{on } \Gamma,
\end{align}
\end{subequations}
and the scattered field additionally satisfies the Sommerfeld radiation conditions
\begin{align}\label{eq:Sommerf}
\frac{\partial u^{\text{scat}}}{\partial r}(\bx) - \imath \kappa u^{\text{scat}}(\bx) = o \left( r^{-1}\right)
\end{align}
as $r\coloneqq \norm{\bx}\rightarrow \infty$, uniformly in $\hat{\bx}:= \bx/r$.
Thus, since $u^{\text{inc}}$ satisfies \cref{eq:sound_soft} on its own,
we seek $u^{\text{scat}}\in H_\loc^1(\D^\cc)$ such that
\begin{subequations}\label{eq:sound_soft_problem}
\begin{align}
\Delta u^\text{scat} + \kappa^2 u^\text{scat} &=0, \quad\text{in } \D^\cc, 
\label{eq:HHeqn}
\\
u^\text{scat} &=-u^\text{inc}, \quad\text{on } \Gamma, 
\label{eq:SoundSftBc}
\end{align}
\end{subequations}
and \cref{eq:Sommerf} holds.

\subsection{Boundary Integral Operators}\label{sec:bios_def}
We proceed to perform a boundary reduction of 
\cref{eq:Sommerf} and \cref{eq:sound_soft_problem}
by means of boundary potentials in order to obtain a
equivalent boundary integral formulation.
To this end, we first introduce the Green's function
of the Helmholtz operator with leading frequency $\kappa \in \IR$ in $\IR^3$
\begin{equation}
	\text{G}^{(\kappa)}(\bz)
	\coloneqq
	\frac{\exp(\imath\kappa\norm{\bz})}{4\pi\norm{\bz}},
	\quad
	\bz
	\in 
	\IR^3 \backslash \{{\bf 0}\},
\end{equation}
with $\imath^2 =-1$ being the imaginary complex unit.
Next, we define the acoustic single layer potential 
$\mathcal{S}_{\Gamma}^{(\kappa)}: H^{-\frac{1}{2}}(\Gamma) \rightarrow
H^1_\loc(\D^\cc)$
\begin{equation}
	\left(
		\mathcal{S}^{(\kappa)}_{\Gamma}
		\varphi
	\right)
	(\bx)
	\isdef
	\int\limits_\Gamma
  	\text{G}^{(\kappa)}(\bx-\by)
	\varphi(\by)
	\text{d}\text{s}_\by,
	\quad
	\bx
	\in 
	\D^\cc,
\end{equation}
and the double layer potential
$\mathcal{D}_{\Gamma}^{(\kappa)}: H^{\frac{1}{2}}(\Gamma) \rightarrow
H^1(\Delta,\D^\cc)$
\begin{equation}
	\left(
		\mathcal{D}_{\Gamma}^{(\kappa)}
		\phi
	\right)
	(\bx)
	\isdef
	\int\limits_\Gamma
  	\frac{\partial \text{G}^{(\kappa)}(\bx-\by)}
	{\partial \mathbf{n}_\Gamma(\by)}
	\phi(\by)
	\text{d}\text{s}_\by,
	\quad
	\bx
	\in 
	\D^\cc.
\end{equation}
By $\mathbf{n}_\Gamma(\by)$ we denote the outward pointing
normal on $\by\in\Gamma$.
Both operators are continuous and satisfy \cref{eq:HHeqn}
and the Sommerfeld radiation condition \cref{eq:Sommerf}, see,
e.g., \cite{SS10}. 

By taking exterior Dirichlet and Neumann traces
\begin{equation}
	\gamma_\Gamma^\cc\colon H^1_\loc(\D^\cc)\to H^{\half}(\Gamma)
	\quad
	\text{and}
	\quad
	\frac{\partial}{\partial\mathbf{n}_\Gamma}\colon H^1(\Delta,\D^\cc)\to H^{-\half}(\Gamma)
\end{equation}
we obtain the single layer operator
\begin{align}\label{eq:SLO}
\OV^{(\kappa)}_{\Gamma}\isdef\gamma_\Gamma^\cc\mathcal{S}^{(\kappa)}_{\Gamma}\colon H^{-\half}(\Gamma)\to H^{\half}(\Gamma),
\end{align}
the double layer operator
\begin{align}\label{eq:DLO}
\half \mathsf{Id}+\OK^{(\kappa)}_{\Gamma}\isdef\gamma_\Gamma^\cc\mathcal{D}^{(\kappa)}_{\Gamma}\colon H^{\half}(\Gamma)\to H^{\half}(\Gamma),
\end{align}
and the adjoint double layer operator
\begin{align}\label{eq:ALO}
\half \mathsf{Id}-\OK^{(\kappa)'}_{\Gamma}\isdef\frac{\partial}{\partial\mathbf{n}_\Gamma}\mathcal{S}^{(\kappa)}_{\Gamma}\colon H^{-\half}(\Gamma)\to H^{-\half}(\Gamma).
\end{align}
The double layer operator and the adjoint double layer operator
are indeed adjoint operators in the bilinear $L^2(\Gamma)$-duality product \cite{SS10}.

\subsection{Combined Boundary Integral Formulation}
\label{sec:A_combined}
By means of the introduced potentials and boundary integral
operators, the sound-soft scattering problem from \cref{sec:soudn_soft} 
can be cast into an equivalent boundary integral formulation.
Here, we explore two widely well-know approaches.

\subsubsection{Direct Boundary Integral Formulation}
According to Green's representation formula,
the total field $u: \D^\cc \rightarrow \IC$ admits
the expression
\begin{align}\label{eq:int_rep_form1}
	u(\bx)
	=
	u^\text{inc}(\bx)
	-
	\mathcal{S}^{(\kappa)}_{\Gamma} 
	\left(
		\frac{\partial u}{\partial\mathbf{n}_\Gamma}
	\right)(\bx),
	\quad
	\bx \in \D^\cc.
\end{align}
The unknown boundary data $\frac{\partial u}{\partial\mathbf{n}_\Gamma}$
can, for example, be obtained by applying Dirichlet and Neumann traces to
\cref{eq:int_rep_form1} and solving
\begin{subequations}
\begin{align}
	\OV^{(\kappa)}_{\Gamma} 
	\frac{\partial u}{\partial\mathbf{n}_\Gamma} 
	&= 
	\gamma^\cc_{\Gamma} u^\text{inc}, 
	\quad \text{on} \; \Gamma,\quad \text{and},  \label{eq:BIE1}\\
	\left( \half \mathsf{Id} + \OK^{(\kappa)'}_{\Gamma}\right) 
	\frac{\partial u}{\partial\mathbf{n}_\Gamma} 
	&= 
	\frac{\partial u^\text{inc}}{\partial\mathbf{n}_\Gamma},  
	\quad \text{on} \; \Gamma.  \label{eq:BIE2}
\end{align}
\end{subequations}
However, well--posedness of \cref{eq:BIE1} and \cref{eq:BIE2} 
collapses when $\kappa^2$ corresponds to a Dirichlet or Neumann
eigenvalue of the Laplacian in $\D$, respectively. 
A fix to this issue consists in taking a suitable linear combination 
of \cref{eq:BIE1} and \cref{eq:BIE2}. 
Provided $\eta \in \IR \backslash \{0\}$, 
the so--called {\it coupling parameter}, we define
\begin{align}\label{eq:directBIO}
	\OA^{(\kappa,\eta)'}_{\Gamma}
	\coloneqq
	\half\mathsf{Id} 
	+ 
	\OK^{(\kappa)'}_{\Gamma}
	- 
	\imath \eta \OV^{(\kappa)}_{\Gamma}.
\end{align}
Exploiting that $\frac{\partial u}{\partial\mathbf{n}_\Gamma} \in L^2(\Gamma)$, 
see, e.g., \cite{Nec1967}, a new boundary integral
formulation of \cref{eq:sound_soft} using the direct approach reads as follows.
\begin{problem}\label{prb:problem_Aprime_k_n}
Seek $\phi \coloneqq\frac{\partial u}{\partial\mathbf{n}_\Gamma} \in L^2(\Gamma)$ 
such that
\begin{align}\label{eq:A_combined}
	\OA^{(\kappa,\eta)'}_{\Gamma}  
	\phi
	=
	\frac{\partial u^{\normalfont\text{inc}} }{\partial\mathbf{n}_\Gamma} 
	-
	\imath \eta
	\gamma^\cc_\Gamma
	u^{\normalfont\text{inc}}
	\in 
	L^{2}(\Gamma).
\end{align}
\end{problem}
Thus, one may establish that for each $\kappa \in \IR_{+}$ it holds
\begin{equation}
	\OA^{(\kappa,\eta)'}_{\Gamma} 
	\in 
	\mathscr{L}_{\normalfont{\text{iso}}}
	(
		L^2(\Gamma)
		,
		L^2(\Gamma)
	),
\end{equation}
thereby implying that \cref{prb:problem_Aprime_k_n}
is well--posed as an operator equation set in $L^2(\Gamma)$.

\subsubsection{Indirect Boundary Integral Formulation}
\label{sec:indirect_formulation}
Another approach to obtain a boundary integral formulation
for the sound-soft scattering problem
consists in expressing $u$ by means
of a {\it combined single and double layer potential},
namely we use the following \emph{ansatz} to represent 
$u\in H^{1}_{\loc}(\D^\cc)$:
\begin{align}\label{eq:combined_SD}
	u(\bx)
	=
	u^\text{inc}(\bx)
	+
	\left(\mathcal{D}^{(\kappa)}_{\Gamma} \varphi \right)(\bx)
	-
	\imath \eta \, 
	\left( \mathcal{S}^{(\kappa)}_{\Gamma} \varphi \right)(\bx),
	\quad
	\bx \in \D^\cc,
\end{align}
where $\varphi \in L^2(\Gamma)$ is an unknown boundary density
and $\eta \in \IR \backslash \{0\}$ is the coupling parameter. 
Application of the exterior Dirichlet trace
to \cref{eq:combined_SD} leads us to the following equivalent 
boundary integral formulation for \cref{eq:sound_soft}:

\begin{problem}
\label{prbm:BIE_sound_soft_A}
We seek $\varphi\in L^2(\Gamma)$ such that
\begin{align}\label{eq:A_combined_2}
	\OA^{(\kappa,\eta)}_{\Gamma}
	\varphi
	= 
	-
	\gamma_\Gamma^\cc
	u^{\normalfont\text{inc}}
	\in 
	L^2(\Gamma),
\end{align}
where
\begin{align}\label{eq:indirectBIO}
	\OA^{(\kappa,\eta)}_{\Gamma}
	\coloneqq
	\half\mathsf{Id} 
	+ 
	\OK^{(\kappa)}_{\Gamma}
	- 
	\imath \eta \OV^{(\kappa)}_{\Gamma}.
\end{align}
\end{problem}
We readily remark that the operators \cref{eq:directBIO} and \cref{eq:indirectBIO}
are adjoint to each other in the bilinear $L^2(\Gamma)$-duality pairing.
Thus, using similar tools to the ones as for the direct approach
in \cref{prb:problem_Aprime_k_n},
one can show that for each $\kappa \in \IR_{+}$
and $\eta \in \IR \backslash \{0\}$ it holds
\begin{equation}
	\OA^{(\kappa,\eta)}_{\Gamma} 
	\in 
	\mathscr{L}_{\normalfont{\text{iso}}}(L^2(\Gamma),L^2(\Gamma)).
\end{equation}
\medskip
In the following, to show shape holomorphy of the boundary
combined boundary integral operators \cref{eq:directBIO} and
\cref{eq:indirectBIO} from \cref{prb:problem_Aprime_k_n} and
\cref{prbm:BIE_sound_soft_A}. To this end,
we first show shape holomorphy
for each of the involved operators \cref{eq:SLO,eq:DLO,eq:ALO}.

\subsection{Parametric Shape Holomorphy of the Single Layer Operator}
\label{sec:shape_hol_V}
For each $\y \in \mathbb{U}$ we set
\begin{align}\label{eq:single_layer_ref_y}
	\hat{\OV}^{(\kappa)}_\y
	\coloneqq
	\tau_\y
	\,
	\OV^{(\kappa)}_{\Gamma_\y}
	\,
	\tau^{-1}_\y
	\in
	\mathscr{L}
	\left(
		L^2(\hat{\Gamma})
		,
		L^2(\hat{\Gamma})
	\right).
\end{align}
The mapping property stated in \cref{eq:single_layer_ref_y}
follows straightforwardly from \cref{lmm:pullback_operator}
and the fact that 
$\OV^{(\kappa)}_{\Gamma_\y}: L^2({\Gamma}_\y) \rightarrow L^2({\Gamma}_\y)$
defines a bounded linear operator, see e.g. \cite[Theorem 3]{CGL09}.
The following results establishes the parametric holomorphy of the
single layer BIO.

\begin{theorem}\label{thm:shape_hol_V}
Let \cref{assump:parametric_boundary} be 
fulfilled with $\boldsymbol{b} \in \ell^p(\IN)$ and $p\in (0,1)$.
Then, the map
\begin{align}\label{eq:map_V_hol}
	\mathcal{V}^{(\kappa)}:
	\mathbb{U}
	\rightarrow
	\mathscr{L}
	\left(
		L^2(\hat{\Gamma}),
		L^2(\hat{\Gamma})	
	\right)
	:
	\y
	\mapsto
	\hat{\OV}^{(\kappa)}_\y
\end{align}
is $(\boldsymbol{b},p,\varepsilon)$-holomorphic and continuous.
\end{theorem}
\begin{proof}
We first construct a complex extension of $\hat{\OV}^{(\kappa)}_\y$
and then show that this extension satisfies all assumptions of \cref{thm:holomorphy_sing_kernel}.
The single layer operator
admits the representation \cref{eq:parametric_BIO} with
\begin{align}
	\mathsf{a}
	(\by, \bz)
	=
	\frac{
		\exp(\imath \kappa \norm{\bz})
	}{
		4\pi \norm{\bz}
	},
	\quad
	\bz
	\in \IR^3
	\backslash
	\{{\bf 0}\}.
\end{align}
Thus, setting $\norm{\bz}_\IC \coloneqq \sqrt{\bz \cdot \bz}$ with 
$\sqrt{\cdot}$ signifiying the principal branch of the square 
root yields a complex extension
\begin{align}\label{eq:kernel_a_single_layer}
	\mathsf{a}
	(\by, \bz)
	=
	\frac{\exp(\imath\kappa\norm{\bz}_\IC)}{4\pi \norm{\bz}_\IC}
\end{align}
for all $\z \in \mathfrak{U}$. Thus, it remains to verify each item of the assumptions of \cref{thm:holomorphy_sing_kernel}. As per usual, let $\boldsymbol\rho\coloneqq (\rho_j)_{j\geq1}$ 
be any sequence of numbers strictly larger than
one that is $(\boldsymbol{b},\varepsilon)$-admissible. 
\begin{description}
	\item[\cref{thm:holomorphy_sing_kernel1}:]
	Our first observation is that $\mathsf{a}(\cdot,\cdot)$
	as in \cref{eq:kernel_a_single_layer} is independent 
	of the first argument. Consequently, in this particular 
	case, for each fixed $\bz \in \mathfrak{U}$ the map
	$\mathbb{U} \ni \y \mapsto \mathsf{a} \left(
		\boldsymbol{r}_\y \circ \chi_\tau
		\left(
			\widetilde{\mathbf{u}}
		\right)
		,
		\bz
		\right)
		\in 
		\IC
	$ is trivially $(\boldsymbol{b},p,\varepsilon)$-holomorphic
	in each direction, as it is constant.
	\item[\cref{thm:holomorphy_sing_kernel2}:]
    We first note that the principal branch
	of the square root is holomorphic in $\IC \backslash (-\infty,0]$ and
    that the map $\bz \mapsto \bz \cdot \bz$ is a multivariate polynomial, and thus holomorphic.
    Thus, the map $\mathfrak{U} \ni \bz \mapsto \norm{\bz}_\IC \in \IC$
	is holomorphic. Moreover, the function 
	$t\mapsto \exp(\imath\kappa t)$ is entire and thus holomorphic and $t\mapsto \frac{1}{t}$ 
	is holomorphic as well, except at $t=0$.
	Consequently, for any $\z \in \mathcal{O}_{\boldsymbol{\rho}}$
	the map
	\begin{equation}
		\mathfrak{U}
		\ni
		\bz
		\mapsto
		\mathsf{a}
		\left(
		\boldsymbol{r}_\z \circ \chi_\tau
		\left(
			\widetilde{\mathbf{u}}
		\right)
		,
		\bz
		\right)
		\in 
		\IC
	\end{equation}
	is holomorphic.
    \item[\cref{thm:holomorphy_sing_kernel3}:] For each $\tau \in \mathcal{G}$, $\widetilde{\mathbf{u}} \in \widetilde{\tau}$, $\z \in \mathcal{O}_{\boldsymbol{\rho}}$, and $\hat{\bx}\in\hat{\Gamma}$ and abbreviating 
    $\bz=\boldsymbol{r}_\z(\hat{\bx})-\boldsymbol{r}_\z \circ \chi_\tau
    \left(
        \widetilde{\mathbf{u}}
    \right)$ one has
	\begin{equation}\label{eq:a_single_layer}
	\mathsf{a}
   	 \left(
    	\boldsymbol{r}_\z \circ \chi_\tau
    	\left(
     	   \widetilde{\mathbf{u}}
  	  \right)
  	  ,
   	 \bz
   	 \right)
	=
	\frac{
		\exp(
			\imath\kappa\Re\{\norm{\bz}_\IC\}
			-
			\kappa\Im\{\norm{\bz}_\IC\}
		)
	}{
		4\pi \norm{\bz}_\IC
	}.
	\end{equation}
	Observing that $\snorm{ \norm{\bz}_\IC} = \norm{\bz}$ yields
	\begin{equation}\label{eq:bound_a_V_1}
	\snorm{\mathsf{a}(\by,\bz)}
	\leq
	\frac{\exp(\kappa\snorm{\Im\{\norm{\bz}_\IC\}})}{4\pi \snorm{\norm{\bz}_\IC}}
	\leq
	\frac{\exp(\kappa\norm{\bz})}{4\pi \norm{\bz}}.
	\end{equation}
	We remark that at this point the bound in \cref{eq:bound_a_V_1}
	does not completely fulfil \cref{thm:holomorphy_sing_kernel3} in 
	the assumptions of Theorem \ref{thm:holomorphy_sing_kernel}. 
	It follows from the definition of $\mathcal{O}_{\boldsymbol{\rho}}$
	that for each $\z \in \mathcal{O}_{\boldsymbol{\rho}}$ there exists
	$\y \in \mathbb{U}$ such that $\snorm{z_j-y_j}\leq \rho_j-1$, for all
	$j\in \IN$.
	Next, we observe that for any $\hat{\bx},\hat{\by} \in \hat\Gamma$
	one has
	\begin{equation}\label{eq:triangle_ineq_single_layer}
	\begin{aligned}
		\norm{
		\boldsymbol{r}_\z(\hat{\bx})
		-
		\boldsymbol{r}_\z(\hat{\by})
    		}
		\leq
		&
		\norm{
			\boldsymbol{r}_\z(\hat{\bx})
			-
			\boldsymbol{r}_\y(\hat{\bx})
		}
		+
		\norm{
			\boldsymbol{r}_\y(\hat{\bx})
			-
			\boldsymbol{r}_\y(\hat{\by})
		} \\
		&
		+
		\norm{
			\boldsymbol{r}_\y(\hat{\by})
			-
			\boldsymbol{r}_\z(\hat{\by})
		}.
	\end{aligned}
	\end{equation}
	It follows from \cref{eq:affine_parametric_representation}
	that for any $\hat{\bx} \in \hat{\Gamma}$
	\begin{equation}\label{eq:triangle_ineq_single_layer_2}
	\begin{aligned}
		\norm{
			\boldsymbol{r}_\z(\hat{\bx})
			-
			\boldsymbol{r}_\y(\hat{\bx})
		}
		&
		\leq
		\sum_{j\geq 1}
		\snorm{z_j-y_j}
		\norm{
			\bvarphi_j
		}_{L^\infty(\hat\Gamma;\IR^3)}
		\\
		&
		\leq
		\sum_{j\geq 1}
		\left(
			\rho_j-1
		\right)
		\norm{
			\bvarphi_j
		}_{\mathscr{C}^{0,1}(\hat\Gamma; \IR^3)}
		\\
		&
		\leq
		\varepsilon,
	\end{aligned}
	\end{equation}
	where the last inequality comes from the fact that 
	the sequence $\boldsymbol\rho\coloneqq (\rho_j)_{j\geq1}$ 
	is $(\boldsymbol{b},\varepsilon)$-admissible. 
	In addition, for any $\hat{\bx}, \hat{\by}  \in \hat\Gamma$ 
	and each $\y\in\mathbb{U}$ one has
	\begin{equation}\label{eq:triangle_ineq_single_layer_3}
		\norm{
			\boldsymbol{r}_\y(\hat{\bx})
			-
			\boldsymbol{r}_\y(\hat{\by})
		}
		\leq
		\norm{\hat{\bx}-\hat{\by}}
		\norm{
			\boldsymbol{r}_\y
		}_{\mathscr{C}^{0,1}(\hat\Gamma; \IR^3)}.
	\end{equation}
	The map 
	$\mathbb{U} \ni \y 
	\mapsto 
	\boldsymbol{r}_\y 
	\in \mathscr{C}^{0,1}(\hat\Gamma; \IR^3)$ is continuous according to 
	\cref{lmm:holomorphic_parametric_rep}, thus rendering the map 
	$\mathbb{U} \ni \y \mapsto 	
	\norm{
		\boldsymbol{r}_\y
	}_{\mathscr{C}^{0,1}(\hat\Gamma; \IR^3)}$
	continuous as well, as any norm depends continuoulsy on
	its argument. As pointed out previously, $\mathbb{U}$ 
	defines a compact subset when equipped with the product topology.
	Therefore, the map $\y \mapsto \norm{
		\boldsymbol{r}_\y
	}_{\mathscr{C}^{0,1}(\hat\Gamma; \IR^3)}$
	attains a maximum.
	Consequently, by collecting
	\cref{eq:triangle_ineq_single_layer} and 
	\cref{eq:triangle_ineq_single_layer_2}, one has that for
	any $\hat{\bx}, \hat{\by}  \in \hat\Gamma$ 
	and each $\y \in \mathbb{U}$ it holds
	\begin{align}
		\norm{
		\boldsymbol{r}_\z(\hat{\bx})
		-
		\boldsymbol{r}_{\z}(\hat{\by})
    		}
		\leq
		2\varepsilon
		+
		\text{diam}(\hat\Gamma)
		\,
		\max_{\y \in \mathbb{U}}
		\norm{
		\boldsymbol{r}_\y
		}_{\mathscr{C}^{0,1}(\hat\Gamma; \IR^3)}.
	\end{align}
	Hence, with 
	\begin{equation}
		C_{\mathsf{a}}
		\coloneqq 
		\exp
		\left(
			\kappa
			\left(
			2\varepsilon
			+
			\text{diam}(\hat\Gamma)
			\,
			\max_{\y \in \mathbb{U}}
			\norm{
			\boldsymbol{r}_\y
			}_{\mathscr{C}^{0,1}(\hat\Gamma; \IR^3)}
			\right)
		\right)
	\end{equation}
	$\mathsf{a}(\cdot,\cdot)$ as in \cref{eq:a_single_layer}
	satisfies \cref{thm:holomorphy_sing_kernel3} in 
	the assumptions of Theorem \ref{thm:holomorphy_sing_kernel}. 
\end{description}
Therefore, it follows from \cref{thm:holomorphy_sing_kernel} 
that the map in \cref{eq:map_V_hol} is 
$(\boldsymbol{b},p,\varepsilon)$-holomorphic and continuous.
\end{proof}

\subsection{Parametric Shape Holomorphy of the Double Layer Operator}
\label{sec:shape_hol_K}
For each $\y \in \mathbb{U}$ we define
\begin{align}\label{eq:double_layer_ref_y}
	\hat{\OK}_\y^{(\kappa)}
	\coloneqq
	\tau_\y
	\;
	\OK_{\Gamma_\y}^{(\kappa)}
	\tau^{-1}_\y
	\in
	\mathscr{L}
	\left(
		L^2(\hat{\Gamma})
		,
		L^2(\hat{\Gamma})
	\right).
\end{align}
Again, as argued previously for the single layer operator,
the mapping property stated in \cref{eq:double_layer_ref_y}
follows straightforwardly from \cref{lmm:pullback_operator}
and the fact that for each $\y\in \mathbb{U}$ one has 
$\OK_{\Gamma_\y}^{(\kappa)}: L^2(\Gamma_\y) \rightarrow L^2(\Gamma_\y) $
defines a bounded linear operator, see e.g. \cite[Theorem 3.5]{CGL09} 

However, unlike in the analysis for the single layer operator, 
we can not directly argue using solely \cref{assump:parametric_boundary}
and \cref{thm:shape_hol_V} as the term 
$\frac{\partial G^{(\kappa)}}{\partial\mathbf{n}}$ does not yield an integrable kernel
when considering the double layer operator set on a Lipschitz boundary. 
The latter difficulty can be overcome by assuming
that all surfaces are piecewise $\mathscr{C}^{1,1}$
as in the following assumption.

\begin{assumption}\label{assump:regularity_patches}
There is a decomposition $\mathcal{G}$ of $\hat{\Gamma}$ such that
for each $\y \in \mathbb{U}$
and each ${\tau} \in \mathcal{G}$ one has that
$\br_\y\circ\chi_\tau\in \mathscr{C}^{1,1}(\widetilde{\tau};\IR^3)$.
\end{assumption}

Next, we proceed to prove the parametric holomorphy result
for the double layer BIO.

\begin{theorem}\label{thm:shape_hol_double_layer}
Let \cref{assump:parametric_boundary} and 
\cref{assump:regularity_patches} be fulfilled
with $\boldsymbol{b} \in \ell^p(\IN)$ and $p\in (0,1)$.
Then, there exists $\varepsilon>0$ such that the map
\begin{align}\label{eq:map_K_hol}
	\mathcal{K}^{(\kappa)}:
	\mathbb{U}
	\rightarrow
	\mathscr{L}
	\left(
		L^2(\hat{\Gamma}),
		L^2(\hat{\Gamma})	
	\right)
	:
	\y
	\mapsto
	\hat{\OK}_\y^{(\kappa)}
\end{align}
is $(\boldsymbol{b},p,\varepsilon)$-holomorphic and continuous.
\end{theorem}
\begin{proof}
As in the analysis of the parametric holomorphy result
for single layer operator, we start by constructing a complex
extension of $\hat{\OK}_{\y}$. To this end, we start by remarking that
the double layer operator has the representation
\cref{eq:parametric_BIO} with
\begin{align}
	\mathsf{a}(\by,\bz) 
	=
	\big({\bf n}(\by)
	\cdot
	\bz\big)
	\frac{
	\exp(\imath\kappa\norm{\bz})
	}{4\pi\norm{\bz}^3}
	\left(
		\imath \kappa
		\norm{\bz}
		-1
	\right),
	\quad
    	\by\in\hat{\Gamma},
	\quad
	\bz
	\in \IR^{3}
	\backslash
	\{{\bf 0}\}.
\end{align}
Thus, for each $\tau \in \mathcal{G}$, $\widetilde{\mathbf{u}} \in \widetilde{\tau}$, $\z \in \mathcal{O}_{\boldsymbol{\rho}}$, and $\hat{\bx}\in\hat{\Gamma}$ and abbreviating $\bz=\hat{\bx}-\boldsymbol{r}_\z \circ \chi_\tau
    \left(
        \widetilde{\mathbf{u}}
    \right)$ 
we obtain the complex extension
\begin{align}\label{eq:kernel_a_double_layer}
	\mathsf{a}
	(\boldsymbol{r}_\z(\hat{\mathbf{x}}), \bz)
	=
    \big(\hat{\mathbf{n}}_\z(\hat{\mathbf{x}})
	\cdot
	\bz\big)
	\frac{
	\exp(\imath \kappa\norm{\bz}_\IC)
	}{4\pi\norm{\bz}^3_\IC}
	\left(
		\imath \kappa
		\norm{\bz}_\IC
		-1
	\right),
	\quad
	\bz \in \mathfrak{U}.
\end{align}
As per usual, let $\boldsymbol\rho\coloneqq (\rho_j)_{j\geq1}$ 
be any sequence of numbers strictly larger than
one that is $(\boldsymbol{b},\varepsilon)$-admissible.
We verify each item of the assumptions of \cref{thm:holomorphy_sing_kernel}.

\begin{description}
	\item[\cref{thm:holomorphy_sing_kernel1}:]
	Follows directly from \cref{lmm:normal_derivative}.
	\item[\cref{thm:holomorphy_sing_kernel2}:]
	As in the proof of \cref{thm:shape_hol_V}, 
	the maps $\mathfrak{U} \ni \bz \mapsto \norm{\bz}_\IC \in \IC $
	and $\IC\setminus\{0\}\ni t \mapsto \frac{1}{t} \in \IC$
	are holomorphic.
	Consequently, for any $\z \in \mathcal{O}_{\boldsymbol{\rho}}$
	the map
	$	
		\mathfrak{U}
		\ni
		\bz
		\mapsto
		\mathsf{a}
		\left(
		\boldsymbol{r}_\z(\hat{\mathbf{x}})
		,
		\bz
		\right)
		\in 
		\IC
	$
	is holomorphic as well.
	\item[\cref{thm:holomorphy_sing_kernel3}:]
    We first need to address the non-integrability of $\frac{\partial G^{(\kappa)}}{\partial\mathbf{n}}$. To this end, it is well known that for each $\tau \in \mathcal{G}$, $\hat\bx, \hat\by \in \tau $, and $\y\in\mathbb{U}$ it holds    
    \begin{align}\label{eq:dlnormalest}
		\underline{c}(\tau)
		\snorm{
			\hat{\bf n}_\y (\hat\by)
			\cdot
			\left(
				\boldsymbol{r}_\y  (\hat\bx)
				-
				\boldsymbol{r}_\y  (\hat\by)
			\right)
		}
		\leq
		\norm{ \hat\bx -  \hat\by}^2
        		\leq
        		\overline{c}(\tau)
        		\norm{ \boldsymbol{r}_\y(\hat\bx)-\boldsymbol{r}_\y (\hat\by)}^2,
	\end{align}
    for some $0<\underline{c}\leq\overline{c}<\infty$ depending on the patch $\tau$,
    see, e.g., \cite[Lemma 2.2.14]{SS10}. 
   As there are only finitely many patches, the constants in \cref{eq:dlnormalest}
    can be made independent of $\tau\in \mathcal{G}$ by taking \
    $\underline{c} = \min_{\tau \in \mathcal{G}} \underline{c}(\tau)$ and $\overline{c} = \max_{\tau \in \mathcal{G}} \overline{c}(\tau)$.
    We note particularly that the proof together with our assumptions guarantee that the constant can be chosen independently of $\y \in \mathbb{U}$. By a perturbation argument we can conclude that there is $\varepsilon>0$ and two possibly different constants $0<\underline{c}\leq\overline{c}<\infty$ such that 
    \begin{align}
        \underline{c}
		\snorm{
			\hat{\bf n}_\z (\hat\by)
			\cdot
			\left(
				\boldsymbol{r}_\z  (\hat\bx)
				-
				\boldsymbol{r}_\z  (\hat\by)
			\right)
		}
		\leq
		\norm{ \hat\bx -  \hat\by}^2
        \leq
        \overline{c}
        \norm{ \boldsymbol{r}_\z(\hat\bx)-\boldsymbol{r}_\z (\hat\by)}^2,
	\end{align}
    for all $\z\in\mathcal{O}_{\boldsymbol{\rho}}$. Abbreviating $\bz=\boldsymbol{r}_\z(\hat{\bx})-\boldsymbol{r}_\z  (\hat\by)$ and proceeding as in the proof of \cref{thm:shape_hol_V} implies
	\begin{equation}
	|\mathsf{a}
    \left(
    \boldsymbol{r}_\z  (\hat\by)
    ,
    \bz
    \right)|
	\leq
    \frac{\overline{c}}{\underline{c}}
	\frac{\exp(\kappa\norm{\bz})}{4\pi \norm{\bz}}\left(
		\kappa
		\norm{\bz}
		+1
	\right),
	\end{equation}
    i.e., the singularity of the complex extension of the double layer operator behaves similarly to the single layer operator case. The conclusion follows in complete analogy to the proof of \cref{thm:shape_hol_V}.
\end{description}
Finally, recalling \cref{thm:holomorphy_sing_kernel}
one may conclude that the map introduced in \cref{eq:map_K_hol} is 
$(\boldsymbol{b},p,\varepsilon)$-holomorphic and continuous.
\end{proof}

\begin{remark}
\Cref{assump:regularity_patches} only enters when using \cite[Lemma 2.2.14]{SS10} to state
\cref{eq:dlnormalest}. A careful inspection of its proof shows that the patch-wise $\mathscr{C}^{1,1}$
smoothness requirement in \cref{assump:regularity_patches} may be relaxed to
$\mathscr{C}^{1,\alpha}$ with $\alpha\in(0,1)$. More precisely, it suffices to assume that
for each $\y \in \mathbb{U}$
and each ${\tau} \in \mathcal{G}$ one requires
$\br_\y\circ\chi_\tau\in \mathscr{C}^{1,\alpha}(\widetilde{\tau};\IR^3)$ for some $\alpha\in(0,1)$.
\end{remark}

\subsection{Parametric Shape Holomorphy of the Adjoint Double Layer Operator}
In order to establish the parametric holomorphy 
result for the adjoint double layer operator from
the double layer operator itself, we consider for each $\y \in \mathbb{U}$ the bilinear $L^2(\Gamma_\y)$-inner product
$(\cdot,\cdot)_{L^2(\Gamma_\y)}$. 
For all $\varphi, \phi \in L^2(\Gamma_\y)$ it holds that
\begin{align}\label{eq:doublelayerduality}
	\dotp{{\OK}^{(\kappa)}_{\Gamma_\y} \varphi }{\phi}_{L^2(\Gamma_\y)}
	=
	\dotp{ \varphi }{{\OK}^{(\kappa)'}_{\Gamma_\y}\phi}_{L^2(\Gamma_\y)}.
\end{align}
For each $\y \in \mathbb{U}$ we define
\begin{align}\label{eq:adjoint_double_layer_ref_y}
	\hat{\OK}^{(\kappa)'}_{\y}
	\coloneqq
	\tau_\y
	\;
	{\OK}^{(\kappa)'}_{\Gamma_\y}
	\;
	\tau^{-1}_\y
	\in
	\mathscr{L}
	\left(
		L^2(\hat{\Gamma})
		,
		L^2(\hat{\Gamma})
	\right),
\end{align}
where the mapping properties are immediate. In fact, defining
a modified, real-valued inner product on $L^2(\hat{\Gamma})$,
yet still depending on $\y \in \mathbb{U}$, by
\begin{equation}\label{eq:modifiedinner}
\begin{aligned}
	\dotp{ \varphi }{\phi}_{L^2(\Gamma_\y)}
	&
	=
	\int\limits_{\Gamma_\y}
	\varphi(\by) \phi (\by)
	\text{d}s_{\by} 
	\\
	&
	=
	\sum_{\tau \in \mathcal{G}} 
	\int\limits_{\widetilde{\tau}} 
	\left(
		\tau_\y \varphi
		\circ
		\chi_\tau
	\right)
	(\widetilde{\mathbf{u}}) 
	\left(
		\tau_\y \phi
		\circ
		\chi_\tau
	\right)
	(\widetilde{\mathbf{u}}) 
	\sqrt{g_{\tau,\y}(\widetilde{\mathbf{u}})} 
	\text{d} 
	\widetilde{\mathbf{u}} 
	\\
	&
	\eqqcolon
	\dotp{\hat \varphi}{\hat \phi}_{L^2(\hat \Gamma),\y},
\end{aligned}
\end{equation}
with $\hat{\varphi}=\tau_\y\varphi,\hat{\phi}=\tau_\y\phi\in L^2(\hat{\Gamma})$ and $\y \in \mathbb{U}$,
one readily verifies that $\hat{\OK}^{(\kappa)'}_{\y}$ is indeed the adjoint
operator of $\hat{\OK}^{(\kappa)}_{\y}$ with respect to
$\dotp{\cdot}{\cdot}_{L^2(\hat \Gamma),\y}$. I.e.,
for each $\y \in \mathbb{U}$ and any $\hat\varphi, \hat\phi \in L^2(\hat\Gamma)$
it holds
\begin{align}\label{eq:refdoublelayeradjoint}
	\dotp{\hat{\OK}^{(\kappa)}_{\y} \hat\varphi }{\hat\phi}_{L^2(\hat \Gamma),\y}
	=
	\dotp{\hat\varphi }{\hat{\OK}^{(\kappa)'}_{\y}\hat\phi}_{L^2(\hat \Gamma),\y}.
\end{align}
Using \cref{eq:doublelayerduality},
parametric holomorphy of the adjoint double layer operator can be derived from
the corresponding result for the double layer operator as thoroughly described 
in the following result.

\begin{theorem}\label{thm:shape_hol_K_adj}
Let \cref{assump:parametric_boundary} and \cref{assump:regularity_patches}
be fulfilled with $\boldsymbol{b} \in \ell^p(\IN)$ and $p\in (0,1)$.
Then, there exists $\varepsilon>0$ such that the map
\begin{align}\label{eq:map_K_adj__hol}
	\mathcal{K}^{(\kappa)'}:
	\mathbb{U}
	\rightarrow
	\mathscr{L}
	\left(
		L^2(\hat{\Gamma}),
		L^2(\hat{\Gamma})	
	\right)
	:
	\y
	\mapsto
	\hat{\OK}^{(\kappa)'}_\y
\end{align}
is $(\boldsymbol{b},p,\varepsilon)$-holomorphic and continuous.
\end{theorem}

\begin{proof}
The complex extension of the modified inner product \cref{eq:modifiedinner}
to $\z \in \mathcal{O}_{\boldsymbol{\rho}}$ reads
\begin{align}
	\dotp{\hat \varphi}{\hat \phi}_{L^2(\hat \Gamma),\z}
	=
	\sum_{\tau \in \mathcal{G}} 
	\int\limits_{\widetilde{\tau}} 
	\left(
		\hat \varphi
		\circ
		\chi_\tau
	\right)
	(\widetilde{\mathbf{u}}) 
	\left(
		\hat\phi
		\circ
		\chi_\tau
	\right)
	(\widetilde{\mathbf{u}}) 
	\sqrt{g_{\tau,\z}(\widetilde{\mathbf{u}})} 
	\text{d} 
	\widetilde{\mathbf{u}}
\end{align}
and takes arguments $\hat{\varphi},\hat{\phi}\in L^2(\hat{\Gamma})$.
Thus, in accordance with \cref{eq:refdoublelayeradjoint}, we can
extend $\hat{\OK}^{(\kappa)'}_{\y}$ into
$\mathcal{O}_{\boldsymbol{\rho}}$ as the (unique)
bounded linear operator satisfying 
$\hat{\OK}^{(\kappa)'}_{\z}:L^2(\hat\Gamma) \rightarrow L^2(\hat\Gamma)$
\begin{align}
	\dotp{\hat{\OK}^{(\kappa)}_{\z} \hat\varphi }{\hat\phi}_{L^2(\hat\Gamma),\z}
	=
	\dotp{\hat\varphi }{\hat{\OK}^{(\kappa)'}_{\z}\hat\phi}_{L^2(\hat\Gamma),\z},
\end{align}
for all $\hat\varphi, \hat\phi \in L^2(\hat\Gamma)$ and $\z\in\mathcal{O}_{\boldsymbol{\rho}}$.
The map $\mathbb{U} \ni \y \mapsto \dotp{\cdot }{\cdot}_{\y,L^2(\hat\Gamma)}$ into the
Banach space of complex-valued bilinear forms on $L^2(\hat{\Gamma})\times L^2(\hat{\Gamma})$
is $(\boldsymbol{b},p,\varepsilon)$-holomorphic and continuous as a consequence
of \cref{lemma:gramian_holomorphic}, and so is 
$\mathbb{U} \ni \y \mapsto \hat{\OK}^{(\kappa)}_\y\in \mathscr{L}\left(L^2(\hat\Gamma) , L^2(\hat\Gamma)\right)$
as stated in \cref{thm:shape_hol_double_layer}.
\end{proof}

\subsection{Parametric Shape holomorphy of Combined Field Integral Equations}
With the shape holomorphy results of the single, double, and adjoint
double layer operator available, we can now conclude that
\cref{eq:directBIO} and \cref{eq:indirectBIO}, i.e.,
\begin{equation}
	\mathsf{A}^{(\kappa,\eta)}_{\Gamma}: 
	L^2(\Gamma) \rightarrow  L^2(\Gamma)
	\quad
	\text{and}
	\quad
	\mathsf{A}^{(\kappa,\eta)'}_{\Gamma}:
	L^2(\Gamma) \rightarrow  L^2(\Gamma),
\end{equation}
are shape holomorphic as well.
\begin{corollary}\label{cor:combinedholomorphy}
Let \cref{assump:parametric_boundary} and \cref{assump:regularity_patches} be satisfied
with $\boldsymbol{b} \in \ell^p(\IN)$ and $p\in (0,1)$.
Then there exists $\varepsilon>0$ such that
\begin{subequations}\label{eq:combinedholomorphy}
\begin{align}
    \mathcal{A}^{(\kappa,\eta)}\colon
	\mathbb{U}
	\rightarrow
	\mathscr{L}
	\left(
		L^2(\hat{\Gamma}),
		L^2(\hat{\Gamma})	
	\right)
	&\colon
    \y
    \mapsto
	\hat{\mathsf{A}}_\y^{(\kappa,\eta)}
	\coloneqq
	\tau_\y
	\;
	\mathsf{A}_{\Gamma_\y}^{(\kappa,\eta)}
	\tau^{-1}_\y,\\
    \mathcal{A}^{(\kappa,\eta)'}\colon
	\mathbb{U}
	\rightarrow
	\mathscr{L}
	\left(
		L^2(\hat{\Gamma}),
		L^2(\hat{\Gamma})	
	\right)
	&\colon
    \y
    \mapsto
	\hat{\mathsf{A}}_\y^{(\kappa,\eta)'}
	\coloneqq
	\tau_\y
	\;
	\mathsf{A}_{\Gamma_\y}^{(\kappa,\eta)'}
	\tau^{-1}_\y,
\end{align}
\end{subequations}
are $(\boldsymbol{b},p,\varepsilon)$-holomorphic and continuous.
\end{corollary}
\begin{proof}
This result is a consequence of \cref{thm:shape_hol_V}, \cref{thm:shape_hol_double_layer}, and \cref{thm:shape_hol_K_adj}.
\end{proof}

\subsection{Parametric Shape Holomorphy of the Parameter-to-Solution Map}
We are interested in the depedence of the solution
to \cref{prb:problem_Aprime_k_n} and \cref{prbm:BIE_sound_soft_A}
upon the parametric boundary shape. More precisely, 
we are interested in establishing holomorphy of the \emph{parameter-to-solution} maps
\begin{equation}\label{eq:d2s_map_1}
	\mathbb{U}
	\ni
	\y
	\mapsto
	\hat{\phi}_\y
	\coloneqq
	\left(
		\hat{\OA}^{(\kappa,\eta)'}_{\y} 
	\right)^{-1}
	\hat{f}_\y
	\in L^2(\hat{\Gamma}),
\end{equation}
and
\begin{equation}\label{eq:d2s_map_2}
	\mathbb{U}
	\ni
	\y
	\mapsto
	\hat{\varphi}_\y
	\coloneqq
	\left(
		\hat{\OA}^{(\kappa,\eta)}_{\y} 
	\right)^{-1}
	\hat{g}_\y
	\in L^2(\hat{\Gamma}),
\end{equation}
where
\begin{align}\label{eq:pullbackrhs}
	\hat{f}_\y	
	\coloneqq
	\tau_\y {f}_{\Gamma_\y}
	\in 
	L^2(\hat{\Gamma})
	\quad
	\text{and}
	\quad
	\hat{g}_\y	
	\coloneqq
	\tau_\y {g}_{\Gamma_\y}
	\in 
	L^2(\hat{\Gamma}).
\end{align}
With the results stated in \cref{thm:shape_hol_V},
\cref{thm:shape_hol_double_layer}, and \cref{thm:shape_hol_K_adj},
together with \cref{prop:hol_maps_banach_spaces}, one can 
establish the following result.

\begin{corollary}
Let \cref{assump:parametric_boundary}  and \cref{assump:regularity_patches} be satisfied
with $\boldsymbol{b} \in \ell^p(\IN)$ and $p\in (0,1)$.
Then there exists $\varepsilon>0$ such that 
the parameter-to-solution maps introduced in
\cref{eq:d2s_map_1} and \cref{eq:d2s_map_2}
are $(\boldsymbol{b},p,\varepsilon)$-holomorphic and continuous.
\end{corollary}
\begin{proof}
The proof is in complete analogy to the proof of \cref{cor:par-to-sol-holomorphy}.
To this end, we remark that \cref{cor:combinedholomorphy} implies that there exists
an $\varepsilon>0$ such that the maps
$\mathcal{A}^{(\kappa,\eta)'}$ and $\mathcal{A}^{(\kappa,\eta)}$ 
are $(\boldsymbol{b},p,\varepsilon)$-holomorphic and continuous. Moreover, it holds
$\OA^{(\kappa,\eta)'}_{\Gamma},\OA^{(\kappa,\eta)}_{\Gamma}\in\mathscr{L}_{\normalfont{\text{iso}}}(\Gamma_\y,\Gamma_\y)$
for all $\y\in\mathbb{U}$.
Thus, \cref{lmm:pullback_operator} implies that
\begin{equation}
	\hat{\mathsf{A}}_\y^{(\kappa,\eta)},
	\;
	\hat{\mathsf{A}}_\y^{(\kappa,\eta)'}
	\in
	\mathscr{L}_{\normalfont{\text{iso}}}
	\left(
		L^2(\hat{\Gamma})
        ,
        L^2(\hat\Gamma)
	\right)
\end{equation}
for each $\y\in\mathbb{U}$.

It only remains to show the parametric holomorphy of $\y\mapsto\hat{f}_\y$
and $\y \mapsto \hat{g}_\y$.
To this end, in \cref{prb:problem_Aprime_k_n}, the right-hand side is given
through a plane incident wave $u^\text{i}(\bx):=\exp(\imath
\kappa \bx \cdot \boldsymbol{\hat{d}}_{\text{inc}})$, yielding
\begin{equation}
	\hat{f}_\y
	(\hat{\bx})
	=
	\kappa
	\hat{\mathbf{n}}_\y(\hat{\bx})
	\cdot
	\boldsymbol{\hat{d}}_{\text{inc}}
	\exp
	\left(
		\imath\kappa \boldsymbol{r}_\y(\hat{\bx})
		\cdot 
		\boldsymbol{\hat{d}}_{\text{inc}}
	\right)
	-
	\imath \eta
	\exp
	\left(
		\imath\kappa \boldsymbol{r}_\y(\hat{\bx})
		\cdot 
		\boldsymbol{\hat{d}}_{\text{inc}}
	\right).
\end{equation}
\cref{lmm:holomorphic_parametric_rep,lmm:normal_derivative} imply that
that $\mathbb{U} \ni \y \mapsto \hat{f}_\y\in L^2(\hat{\Gamma})$
is indeed $(\boldsymbol{b},p,\varepsilon)$-holomorphic and continuous
for some $\varepsilon>0$.
Similarly, for \cref{prbm:BIE_sound_soft_A}, we have
\begin{equation}
	\hat{g}_\y
	(\hat{\bx})
	=
	-
	\exp
	\left(
		\imath\kappa \boldsymbol{r}_\y(\hat{\bx})
		\cdot 
		\boldsymbol{\hat{d}}_{\text{inc}}
	\right),
	\quad
	\hat{\bx}
	\in
	\hat\Gamma,
\end{equation}
for which it follows with \cref{lmm:holomorphic_parametric_rep}
that $\mathbb{U} \ni \y \mapsto \hat{g}_\y\in L^2(\hat{\Gamma})$
is $(\boldsymbol{b},p,\varepsilon)$-holomorphic and continuous
for some $\varepsilon>0$.
\end{proof}

\subsection{Parametric Shape Holomorphy of the Far-Field Pattern}\label{ssec:shape_hol_far_field}
For later reference, following \cite[Chapter 5]{CK12}, we also consider the
{\it far-field patterns} of the single and double
layer potentials $\mathcal{S}^{(\kappa)}_{\Gamma}$ and $\mathcal{D}^{(\kappa)}_{\Gamma}$ introduced
in \cref{sec:bios_def}. For $\lambda\in L^2(\Gamma)$ and an 
\emph{observation direction} $\boldsymbol{d}_o\in \mathbb{S}^2$ they
are defined as
\begin{align}
	\text{FF}^{(\kappa)}_{\mathcal{S},\Gamma}(\lambda)
	&
	\coloneqq
	{\frac{1}{4\pi}}
	\int\limits_{\Gamma}
	\exp(-\imath \kappa \, \boldsymbol{d}_o \cdot \by)
	\lambda(\by)
	\text{d} \text{s}_{{\by}},\\
	\text{FF}^{(\kappa)}_{\mathcal{D},\Gamma} (\lambda)
	&
	\coloneqq
	{-\imath
	{\frac{\kappa}{4\pi}}}
	\int\limits_{\Gamma}
	\left(\boldsymbol{d}_o \cdot \mathbf{n}_{\Gamma}\right)
	\exp(-\imath \kappa \,\boldsymbol{d}_o \cdot \by)
	\lambda(\by)
	\text{d} \text{s}_{{\by}},
\end{align}
respectively. In the following, we show that these far-field patterns
are shape holomorphic as well. To this end, for each  $\y \in \mathbb{U}$,
we set 
\begin{align}\label{eq:farfieldpullbacks}
	\hat{\text{FF}}^{(\kappa)}_{\mathcal{S},\y}
	\coloneqq
	\tau_\y
	\,
	{\text{FF}}^{(\kappa)}_{\mathcal{S},\Gamma_\y}
	\,
	\tau^{-1}_\y
	\quad
	\text{and}
	\quad
	\hat{\text{FF}}^{(\kappa)}_{\mathcal{D},\y}
	\coloneqq
	\tau_\y
	\,
	{\text{FF}}^{(\kappa)}_{\mathcal{D},\Gamma_\y}
	\,
	\tau^{-1}_\y.
\end{align}
\begin{lemma}\label{lmm:holomorphic_farfield}
Let \cref{assump:parametric_boundary} and \cref{assump:regularity_patches} be satisfied
with $\boldsymbol{b} \in \ell^p(\IN)$ and $p\in (0,1)$.
Assume that there exists $\varepsilon'>0$
such that the map $\mathbb{U} \ni  \y \mapsto \hat{\lambda}_\y \in L^2(\hat{\Gamma})$
is $(\boldsymbol{b},p,\varepsilon')$-holomorphic and continuous.
Then there exists $\varepsilon>0$ such that
\begin{align}\label{eq:shape_hol_farfield}
	\mathbb{U}
	\ni
	\y
	\mapsto
	\normalfont\hat{\text{FF}}^{(\kappa)}_{\mathcal{S},\y}
	\left(
		 \hat{\lambda}_\y
	\right)
	\in \IC
	\quad
	\text{and}
	\quad
	\mathbb{U}
	\ni
	\y
	\mapsto
	\normalfont\hat{\text{FF}}^{(\kappa)}_{\mathcal{D},\y}
	\left(
		 \hat{\lambda}_\y
	\right)
    \in \IC
\end{align} 
are $(\boldsymbol{b},p,\varepsilon)$-holomorphic and continuous.
\end{lemma}

\begin{proof}
For each $\y \in \mathbb{U}$, the explicit representation of \cref{eq:farfieldpullbacks}
is
\begin{align}
	\hat{\text{FF}}^{(\kappa)}_{\mathcal{S},\y}(\hat{\lambda}_\y)
	&=
	{\frac{1}{4\pi}}
    	\sum_{\tau \in \mathcal{G}} 
	\int\limits_{\widetilde{\tau}} 
	\hat{k}^{(\kappa)}_{\mathcal{S},\y}(\widetilde{\mathbf{u}})
	\left(
		\hat{\lambda}_\y \circ\chi_{\tau}
	\right)
	(\widetilde{\mathbf{u}})
	\mathrm{d} \widetilde{\mathbf{u}},\\
  	  \hat{\text{FF}}^{(\kappa)}_{\mathcal{D},\y}(\hat{\lambda}_\y)
	&=
	{-\imath
	{\frac{\kappa}{4\pi}}}
	\sum_{\tau \in \mathcal{G}} 
	\int\limits_{\widetilde{\tau}} 
	\hat{k}^{(\kappa)}_{\mathcal{D},\y}(\widetilde{\mathbf{u}})
	\left(
		\hat{\lambda}_\y \circ\chi_{\tau}
	\right)
	(\widetilde{\mathbf{u}})
	\mathrm{d} \widetilde{\mathbf{u}},
\end{align}
where $\hat{\lambda}_\y \coloneqq \tau_\y \lambda \in L^2(\hat{\Gamma})$ and
\begin{align}
    \hat{k}^{(\kappa)}_{\mathcal{S},\y}(\widetilde{\mathbf{u}})
    &=
    \exp(-\imath \kappa \, \boldsymbol{d}_o \cdot \boldsymbol{r}_\y \circ \chi_{\tau}(\widetilde{\mathbf{u}}))
	\sqrt{g_{\tau,\y}(\widetilde{\mathbf{u}})}\\
    \hat{k}^{(\kappa)}_{\mathcal{D},\y}(\widetilde{\mathbf{u}})
    &=
    \left(\boldsymbol{d}_o \cdot \hat{\mathbf{n}}_{\y}\right)(\chi_{\tau}(\widetilde{\mathbf{u}}))
	\exp(-\imath \kappa \, \boldsymbol{d}_o \cdot \boldsymbol{r}_\y \circ \chi_{\tau}(\widetilde{\mathbf{u}}))
	\sqrt{g_{\tau,\y}(\widetilde{\mathbf{u}})}.
\end{align}
Therein, according to \cref{lmm:normal_derivative} and \cref{lemma:gramian_holomorphic}
the maps $\mathbb{U} \ni  \y \mapsto\hat{\mathbf{n}}_{\y} \in L^{\infty}(\hat{\Gamma},\IR^3)$ 
and $\mathbb{U} \ni  \y \mapsto \sqrt{g_{\tau,\y}} \in L^\infty(\widetilde{\tau})$ 
are $(\boldsymbol{b},p,\varepsilon)$-holomorphic and continuous, 
respectively. 
Consequently, for each $\tau \in \mathcal{G}$ the maps
\begin{equation}
	\mathbb{U}
	\ni
	\y
	\mapsto
	\hat{k}^{(\kappa)}_{\mathcal{S},\y}(\widetilde{\mathbf{u}})
	\in 
	L^2(\widetilde{\tau})
    \quad
    \text{and}
    \quad
	\mathbb{U}
	\ni
	\y
	\mapsto
	\hat{k}^{(\kappa)}_{\mathcal{D},\y}(\widetilde{\mathbf{u}})
	\in 
	L^2(\widetilde{\tau})
\end{equation}
are $(\boldsymbol{b},p,\varepsilon)$-holomorphic and continuous, thus
rendering the maps in \cref{eq:shape_hol_farfield} as well.
\end{proof}

\section{Implications in High-dimensional Approximation}\label{sec:HDA}
In this section, we explore the implications of the presently
established parametric holomorphy result on examples
in sparse polynomial approximation and
reduced order modelling, in Bayesian shape inversion, and
in expression rates of artificial neural networks.


\subsection{Sparse Polynomial Approximation and Reduced Order Modelling}\label{sec:sparsepolynomial}
Our aim in the following is to derive best $n$-term polynomial approximation
results to the parameter-to-solution maps
\cref{eq:d2s_map_1} and \cref{eq:d2s_map_2} of \cref{prb:problem_Aprime_k_n}
and \cref{prbm:BIE_sound_soft_A}
and discuss its consequences. That is, we aim to approximate
\begin{equation}\label{eq:parametric_map_app}
	\Phi
	:
	\mathbb{U}
	\rightarrow
	L^2(\hat\Gamma)
	:
	\y
	\mapsto
    \hat{\phi}_\y,
	\hat{\varphi}_\y,
\end{equation}
by polynomials of finite degree by truncating the infinite expansion
\begin{align}\label{eq:legendreexpansion}
	\Phi(\y)
	=
	\sum_{\boldsymbol{\nu} \in \mathcal{F}} 
	c_{\boldsymbol{\nu}} P_{\boldsymbol{\nu}}(\y)
\end{align}
where
\begin{align}\label{eq:infinitepolynomial}
	\mathcal{F}
	\coloneqq
	\left\{
		\boldsymbol{\nu} \in \mathbb{N}_0^{\mathbb{N}}:|\boldsymbol{\nu}|
		<
		\infty
	\right\}
\end{align}
and
$P_{\boldsymbol{\nu}}(\boldsymbol{y})\coloneqq \prod_{j \geq 1} P_{\nu_j}\left(y_j\right)$, with $P_n$ denoting univariate polynomials of degree $n$.
Aiming to recall a result from \cite{CCS15}, we consider Legendre polynomials $P_n=P_n^{(q)}$ on the interval $[-1, 1]$ with normalization constraints $\|P_n^{(q)}\|_{L^{q}([-1,1])}=1$, $q\in\{2,\infty\}$.
The coefficients in \cref{eq:infinitepolynomial} are denoted by $(c_{\boldsymbol{\nu}}^{(q)})_{\boldsymbol{\nu}\in\mathcal{F}}$ accordingly.

We say that a set $\Lambda \subset \mathcal{F}$ is downward closed if 
$\boldsymbol{\nu} \in \Lambda$ and $\boldsymbol{\mu} \leq \boldsymbol{\nu}$
implies $\boldsymbol{\mu} \in \Lambda$. 
To state the sought result, we define for
$c:=\left(c_{\boldsymbol{\nu}}\right)_{\boldsymbol{\nu} \in \mathcal{F}} \subset \mathbb{R}$
its downward closed envelope as
$\mathbf{c}_{\boldsymbol{\nu}}\coloneqq \sup _{\boldsymbol{\mu} \geq \boldsymbol{\nu}}\norm{c_{\boldsymbol{\mu}}}_{L^2(\hat{\Gamma})}$ for $\boldsymbol{\nu} \in \mathcal{F}$,
where $\boldsymbol{\mu} \geq \boldsymbol{\nu}$ signifies that the inequality
is satisfied component wise. It is the smallest monotone decreasing
sequence which bounds $\left(\norm{c_{\boldsymbol{\nu}}}_{L^2(\hat{\Gamma})}\right)_{\boldsymbol{\nu} \in \mathcal{F}}$
element wise. In addition, for $p>0$ 
we define $\ell_m^p(\mathcal{F})$ as the space of sequences with 
downward closed envelope in $\ell^p(\mathcal{F})$.

For $1 \leq p \leq \infty$, we denote 
by $L^p(\mathbb{U} ; V)$ the Lebesgue-Bochner space consisting of strongly $\varrho$-measurable functions $f: \mathbb{U} \rightarrow {V}$ satisfying
$\|f\|_{L^p(\mathbb{U} ; V)}<\infty$, with
\begin{equation}
	\|f\|_{L^p(\mathbb{U} ; V)}
	:= 
	\begin{cases}
		\left(\displaystyle\int_{\mathbb{U}}\|f(\boldsymbol{y})\|_{{V}}^p 
		\mathrm{~d} \varrho(\boldsymbol{y})\right)^{1 / p} & 1 \leq p<\infty \\ 
		\underset{\y \in \mathbb{U}}{\operatorname{ess} \sup} 
		\|f(\boldsymbol{y})\|_{V} & p=\infty,
	\end{cases}
\end{equation}
where $\mathrm{d} \varrho(\boldsymbol{y})$ signifies the
tensor product of the uniform measure in $[-1,1]$, which in turn defines
a measure in $\mathbb{U}$. In particular, we set $L^p(\mathbb{U}) = L^p(\mathbb{U} ; \IR)$,
for $1 \leq p \leq \infty$.

The next result, which is an application of the main result
presented in \cite{CCS15}, addresses the best $n$-term
polynomial approximation of the map in \cref{eq:parametric_map_app}.

\begin{theorem}[{\cite[Theorem 2.2, ff.]{CCS15}}]\label{thm:sparsepolynomial}
Let \cref{assump:parametric_boundary} and \cref{assump:regularity_patches}
be satisfied with $\boldsymbol{b} \in \ell^p(\IN)$ and $p\in (0,1)$.
Then, $\left(\norm{c_{\boldsymbol{\nu}}^{(q)}}_{L^2(\hat{\Gamma})} \right)_{\boldsymbol{\nu} \in \mathcal{F}}$,
$q\in\{2,\infty\}$, from \cref{eq:legendreexpansion}
belong to $\ell^p_m(\mathcal{F})$ and there exists $C>0$ such that 
for each $n \in \IN$ there are finite, downward closed
sets $\Lambda^2_n,\Lambda^\infty_n \in \mathcal{F}$
of cardinality $n$ such that
\begin{equation}
	\norm{
		\Phi
		-
		\sum_{\boldsymbol{\nu} \in \Lambda^\infty_n} 
		c_{\boldsymbol{\nu}}^{(\infty)}
		P_{\boldsymbol{\nu}}^{(\infty)}
	}_{L^{\infty}(\mathbb{U};L^2(\hat{\Gamma}))} 
	\leq 
	C(n+1)^{
	-
	\left(
	\frac{1}{p}-1
	\right)
	}, 
\end{equation}
and
\begin{equation}
	\norm{
		\Phi
		-
		\sum_{{\boldsymbol{\nu}} \in \Lambda^2_n} 
		c_{\boldsymbol{\nu}}^{(2)}
		P_{\boldsymbol{\nu}}^{(2)}
	}_{L^{2}(\mathbb{U};L^2(\hat{\Gamma}))} 
	\leq 
	C(n+1)^{
	-
	\left(
	\frac{1}{p}-\frac{1}{2}
	\right)
	}.
\end{equation}
\end{theorem}

We explore the implications of our parametric holomorphy
result in the context of reduced order modelling. There,
the goal is to find a low-dimensional affine space
approximating the \emph{solution manifold}
\begin{align}
	\mathcal{M}
	\coloneqq
	\left\{
		\Phi(\y):
		\;
		\y \in \mathbb{U}
	\right\}
\end{align}
to \cref{eq:parametric_map_app} up to a given accuracy.
A commonly used concept in nonlinear approximation to 
quantify the approximation error to such solution manifolds
is the so-called Kolmogorov's $n$-width.
For a compact subset $\mathcal{K}$ of a Banach space
$X$ it is defined for $n\in \IN$ as
\begin{align}
	d_n(\mathcal{K})_X
	\coloneqq
	\inf _{\substack{X_n\subset X\\\operatorname{dim}\left(X_n\right) \leq n}} 
	\sup _{v \in \mathcal{K}} 
	\inf _{w \in X_n}
	\norm{v-w}_X.
\end{align}
As a direct consequence of \cref{thm:sparsepolynomial} we obtain
the following bound on Kolmogorov's $n$-width, giving an upper
bound on the minimal number of dimensions to reach a certain approximation
accuracy to the solution manifold.
\begin{corollary}
Let \cref{assump:parametric_boundary} and \cref{assump:regularity_patches} be satisfied
with $\boldsymbol{b} \in \ell^p(\IN)$ and $p\in (0,1)$. Then
\begin{equation}
	d_n(\mathcal{M})_{L^2(\hat\Gamma)}
	\leq
	C(n+1)^{
	-
	\left(
	\frac{1}{p}-1
	\right)
	}.
\end{equation}
\end{corollary}

\subsection{Bayesian Shape Inversion}
\label{sec:bayesian_inverse}
We proceed to discuss the significance of the parametric
holomorphy results in the context of computational
Bayesian shape inversion.
\subsubsection{Bayesian Inverse Problems}
Provided a \emph{parameter--to--solution map} $G: \mathbb{U} \rightarrow X $
with $X$ being a separable Banach space over $\IC$ and a
\emph{prior} probability measure $\mu_0$ on $\mathbb{U}$, one seeks
to update the prior measure from observational data. In the
following we will assume that $\mu_0$ is the uniform probability
measure in the sense of \cite[Section 2]{SS12}. To obtain observations
from solutions in $X$, an \emph{observation operator}
$O: X \rightarrow \mathbb{R}^K$ is assumed to be given, allowing
to introduce the \emph{uncertainty--to--observation} map as
\begin{equation}
	\mathcal{J}
	\coloneqq 
	O \circ G: 
	\mathbb{U} \rightarrow \mathbb{R}^K.
\end{equation} 
Further, the observations $\Upsilon \in \IR^K$ to update the prior measure are
assumed to be distorted by additive Gaussian noise,
\begin{align}\label{eq:model_noise}
	\Upsilon = \mathcal{J}({\y^\star})+\eta,
\end{align} 
where $\eta \sim \mathcal{N}(0,\Sigma)$, $\Sigma \in \mathbb{R}^{K \times K}_\text{sym}$ is a symmetric, positive definite covariance matrix and $\y^\star$ is the \emph{ground truth}.
The updated \emph{posterior probability measure} $\mu^{\Upsilon}$ can then be
expressed with respect to the prior measure as follows.

\begin{theorem}[{\cite[Theorem 2.1]{SS12}}]\label{thm:bayes_thm}
Assume that $\mathcal{J}: \mathbb{U} \rightarrow \IR^K$ is bounded and continuous.
Then $\mu^\Upsilon(\dd \y)$, the distribution of $\y \in \mathbb{U}$
given the data $\Upsilon \in \mathbb{R}^K$, is absolutely continuous with respect to
$\mu_0(\dd \y)$, i.e. there exists a parametric density $\Theta(\y)$
such that for each $\y \in \mathbb{U}$ the \emph{Radon-Nikodym} derivative
is given by
\begin{align}
	\frac{\dd \mu^\Upsilon}{\dd \mu_0}(\y) 
	= 
	\frac{1}{Z} 
	\Theta(\y)
\end{align}
with the posterior density
\begin{align}\label{eq:post_density}
	\Theta(\y) 
	\coloneqq \exp \left(-\Phi_{\Sigma}(\Upsilon, \boldsymbol{y}) \right)
	\quad
	\text{and}
	\quad
	Z
	\coloneqq 
	\int\limits_{\mathbb{U}}\Theta(\y) \mu_0(\dd \boldsymbol{y})>0,
\end{align}
with
\begin{align}
	\Phi_{\Sigma}(\Upsilon, \boldsymbol{y}) 
	= 
	\frac{1}{2}
	\left(
		\Upsilon - \mathcal{J}({\boldsymbol{y}}) 
	\right)^\top 
	\Sigma^{-1}  
	\left(
		\Upsilon - \mathcal{J}({\boldsymbol{y}})
	\right).
\end{align}
\end{theorem}
Given a quantity of interest $\phi\colon\mathbb{U}\to Y$,
being $Y$ a separable Banach space over $\IC$, \cref{thm:bayes_thm}
allows us to compute statistical moments of a quantity of interest with respect to
the posterior measure. For example, one can compute 
the mean with respect to $\mu^\Upsilon(\dd \y)$ as follows
\begin{align}\label{eq:expected_qoi}
	\mathbb{E}^{\mu^\Upsilon}[\phi]
	=
    \int_{\mathbb{U}}\phi(\y)\mu^\Upsilon(\dd\y)
    =
	\frac{1}{Z}
    \int\limits_{\mathbb{U}}\phi(\y) \Theta(\y) \mu_0(\dd\boldsymbol{y}),
\end{align}
with $Z$ as in \cref{eq:post_density}.

\subsubsection{Bayesian Shape Inversion in Acoustic Scattering}
\label{eq:bayesian_shape_inv}
Equipped with the frequency-robust boundary integral formulations
for the sound-soft acoustic scattering problem introduced in
\cref{sec:application}, the uncertainty--to--solution and observation
operators of the Bayesian shape inversion problem are
\begin{equation}\label{eq:unc-to-sol}
	G\colon\mathbb{U}\to L^2(\hat{\Gamma}):
	\y
	\mapsto
	\begin{cases}
	\hat{\phi}_{\y}=\left(
		\hat{\OA}^{(\kappa,\eta)'}_{\y} 
	\right)^{-1}
	\hat{f}_\y & \text{for \cref{prb:problem_Aprime_k_n}},\\
	\hat{\varphi}_{\y}=\left(
		\hat{\OA}^{(\kappa,\eta)}_{\y} 
	\right)^{-1}
	\hat{g}_\y & \text{for \cref{prbm:BIE_sound_soft_A}},
\end{cases}
\end{equation}
with $\hat{\OA}^{(\kappa,\eta)}_{\y}$ and $\hat{\OA}^{(\kappa,\eta)'}_{\y}$ as in \cref{eq:combinedholomorphy} and $\hat{f}_\y$ and $\hat{g}_\y$ as in \cref{eq:pullbackrhs}
and
\begin{equation}\label{eq:observation}
O\colon\mathbb{U}\to\IR^2,
\quad
\y
\mapsto
\begin{pmatrix}
\Re(u_\y^\infty)\\
\Im(u_\y^\infty)
\end{pmatrix}
\end{equation}
with
\begin{equation}\label{eq:far_field_u}
u_\y^\infty
=
\begin{cases}
\hat{\text{FF}}^{(\kappa)}_{\mathcal{S},\Gamma_\y}
	\left(\hat{\phi}_{\y}\right) & \text{for \cref{prb:problem_Aprime_k_n}},\\
\hat{\text{FF}}^{(\kappa)}_{\mathcal{S},\Gamma_\y}
	\left(\hat{\varphi}_{\y}\right)
 -\iota\eta\hat{\text{FF}}^{(\kappa)}_{\mathcal{D},\Gamma_\y}
	\left(\hat{\varphi}_{\y}\right)
 & \text{for \cref{prbm:BIE_sound_soft_A}},
\end{cases}
\end{equation}
and $\hat{\text{FF}}^{(\kappa)}_{\mathcal{S},\Gamma_\y}$ and $\hat{\text{FF}}^{(\kappa)}_{\mathcal{D},\Gamma_\y}$ as in \cref{eq:farfieldpullbacks}. Multiple observation directions
or directions of incident waves could also be considered with the straightforward adaptions.
The quantity of interest for our Bayesian shape inversion problem is
\begin{equation}\label{eq:qoi}
	\phi\colon
	\mathbb{U}\to\mathscr{C}^{0,1}(\hat\Gamma; \IR^3)
	\colon
	\y\mapsto\boldsymbol{r}_\y,
\end{equation}
with $\boldsymbol{r}_\y$ is as in \cref{eq:affine_parametric_representation}.
Other quantities of interest could also be considered.

\subsubsection{Parametric Regularity of the Posterior Distribution}
In this section, we present the parametric regularity results for the posterior density as defined in \cref{eq:post_density}. Prior research conducted in \cite{GP18} and \cite{SS14} has also investigated this topic and arrived at similar conclusions. Specifically, in \cite[Theorem 4 and Corollary 5]{GP18}, the authors established parametric regularity estimates for the posterior distribution using a real-valued differentiation approach.

In contrast, our approach centers on the sparsity of the posterior distribution. We observe that it can be expressed as the composition of an entire function and a $(\boldsymbol{b},p,\varepsilon)$-holomorphic map, which in turn implies that the posterior density is also $(\boldsymbol{b},p,\varepsilon)$-holomorphic. This finding provides a useful way to derive the sparsity of the posterior density.



\begin{corollary}\label{thm:paramertric_reg_theta}
Let \cref{assump:parametric_boundary} and 
\cref{assump:regularity_patches} be fulfilled
with $\boldsymbol{b} \in \ell^p(\IN)$ and $p\in (0,1)$.
In addition, let $G:\mathbb{U} \rightarrow L^2(\hat{\Gamma})$,
$O:\mathbb{U}\rightarrow \IR^2$, and 
$\phi:\mathbb{U} \rightarrow \mathscr{C}^{0,1}(\hat{\Gamma},\mathbb{R}^3)$
be as in \cref{eq:unc-to-sol}, \cref{eq:observation}, 
and \cref{eq:qoi}, respectively.
Then the maps
\begin{align}
	\mathbb{U}
	\ni
	\y
	\mapsto
	\Theta(\y)
	\in
	\IR
	\quad
	\text{and}
	\quad
	\mathbb{U}
	\ni
	\y
	\mapsto
    \phi(\y)
	\Theta(\y)
	\in
	\mathscr{C}^{0,1}(\hat{\Gamma},\mathbb{R}^3)
\end{align}
with $\Theta: \mathbb{U} \rightarrow \IC$ as in 
\cref{eq:post_density} are
$(\boldsymbol{b},p,\varepsilon)$-holomorphic and continuous.
\end{corollary}
\begin{proof}
We show that $O$ is indeed $(\boldsymbol{b},p,\varepsilon)$-holomorphic, which yields $(\boldsymbol{b},p,\varepsilon)$-holomorphy of $\mathcal{J}=O\circ G$. The remainder of the proof is in complete analogy to \cite[Theorem 4.1]{SS14}. Using \cref{lmm:holomorphic_farfield}, the $(\boldsymbol{b},p,\varepsilon)$-holomorphy of $O$ is immediate, once we show that taking the real and the imaginary part of a complex function is holomorphic. Indeed, we provide a proof of this statement in \cref{lem:realandimaginarybpe}.
\end{proof}

This result justifies the sparse polynomial interpolation results
and HoQMC methods also for the posterior distribution of the acoustic scattering problem.
Moreover, along the lines of the following section, it also allows
to construct surrogates of the posterior distribution using artificial
neural networks.

\begin{remark}
Although the theoretical results justify the application of these results, there may be computational issues when the posterior concentrates too much. This for example the case when the noise level, i.e., the eigenvalues of $\Sigma$, are small, see \cite{SS16} for a discussion.
\end{remark}

\subsection{Approximation of the Far-Field using Artificial Neural Networks}
Let $L \in \IN$, $\ell_0,\dots, \ell_L \in \IN$
and let
$\sigma:\IR \rightarrow\IR$ be a
non-linear function, referred to in the following as
the \emph{activation function}.
Given $({\bf W}_k,{\bf b}_k)_{k=1}^{L}$,
${\bf W}_k\in\IR^{\ell_k\times\ell_{k-1}}$, ${\bf b}_k\in\IR^{\ell _k}$,
define the affine transformation
 ${\bf A}_k: \IR^{\ell_{k-1}}\rightarrow \IR^{\ell_k}: {\bf x} \mapsto {\bf W}_k {\bf x}+{\bf b}_k$
for $k\in\{1,\ldots,L\}$. We define an \emph{artificial neural network (ANN)}
with activation function $\sigma$ as a map
$\Psi_{\mathcal{NN}}: \IR^{\ell_0}\rightarrow \IR^{\ell_L}$ with
\begin{align}\label{eq:ann_def}
	\Psi_{\mathcal{NN}}({\bf x})
	\coloneqq
	\begin{cases}
	{\bf A}_1(x), & L=1, \\
	\left(
		{\bf A}_L
		\circ
		\sigma
		\circ
		{\bf A}_{L-1}
		\circ
		\sigma
		\cdots
		\circ
		\sigma
		\circ
		{\bf A}_1
	\right)(x),
	& L\geq2,
	\end{cases}
\end{align}
where the activation function $\sigma:\IR\rightarrow \IR$
is applied componentwise. We define the depth and the width
of an ANN as
\begin{equation}
\normalfont\text{width}(\Psi_{\mathcal{NN}})=\max\{\ell_1,\ldots, \ell_L\}
\quad
\text{and}
\quad
\normalfont\text{depth}(\Psi_{\mathcal{NN}})=L.
\end{equation}
In the present work, we consider as activation function
the hyperbolic tangent
\begin{align}
	\sigma(x)
	=
	\text{tanh}(x)
	=
	\frac{\exp(x)-\exp(-x)}{\exp(x)+\exp(-x)}.
\end{align}	
When this particular function is used, we refer to 
\cref{eq:ann_def} as a tanh ANN.
However we remark that other choices exist, such as 
the Rectified Linear Unit (ReLU), Rectified Power Unit (RePU),
and the \emph{swish} activation function.

We are interested in the tanh ANN approximation of the far-field
at a specific observation direction $\boldsymbol{d}_{\text{obs}}
\in \mathbb{S}^2$ using the mathematical tools of \cite{ABD22}.
To this end, we define
\begin{equation}\label{eq:map_ff}
	\Phi:
	\mathbb{U}
	\rightarrow
	\IR:
	\y
	\mapsto
	\snorm{u_\y^\infty}^2,
\end{equation}
where, for each $\y\in \mathbb{U}$, $u_\y^\infty$
is given as in \cref{eq:far_field_u}. In the following,
we present a streamlined and tailored-to-our-purposes
version of the results presented
in \cite{ABD22} to this particular application, which allows
for a straightforward generalization to more involved examples.
To this end,
we first recall the approximation of Legendre 
polynomials by means of tanh ANNs.

\begin{lemma}[{\cite[Theorem 7.4]{ABD22}}]\label{lmm:approx_pol}
Let $\Lambda \subset \mathcal{F}$ with $\mathcal{F}$ as in
\cref{eq:infinitepolynomial} be a finite multi-index set,
$m(\Lambda)=\max_{\boldsymbol{\nu}\in\Lambda}\|\boldsymbol{\nu}\|_1$,
and $\vartheta \subset \IN$, $\snorm{\vartheta}=n$ satisfying
\begin{equation}
	\bigcup_{\boldsymbol{\nu} \in \Lambda}
	\normalfont\text{supp}(\boldsymbol{\nu})
	\subseteq
	\vartheta.
\end{equation}
Then, for any $\delta \in (0,1)$ there exists a tanh
ANN $\Psi_{\Lambda,\delta}: \IR^n \rightarrow \IR^{\snorm{\Lambda}}$
with
$\Psi_{\Lambda,\delta} = \left(\Psi_{\boldsymbol{\nu},\delta}\right)_{\boldsymbol{\nu} \in \Lambda}$
 such that
\begin{equation}
	\norm{
		P_{\boldsymbol{\nu}}^{(2)}
		-
		\Psi_{\boldsymbol{\nu},\delta}
		\circ
		\mathcal{T}_\vartheta
	}_{L^{\infty}(\mathbb{U})}
	\leq
	\delta,
	\quad
	\boldsymbol{\nu}
	\in
	\Lambda,
\end{equation}	
with 
$
	\normalfont\text{width}
	\left(
		\Psi_{\boldsymbol{\nu},\delta}
	\right)
	=
	\mathcal{O}
	\left(
		\snorm{\Lambda}
		m(\Lambda)
	\right)
$
and
$
	\normalfont\text{depth}
	\left(
		\Psi_{\boldsymbol{\nu},\delta}
	\right)
	=
	\mathcal{O}
	\left(
		\log_2(m(\Lambda))
	\right)
$,
where
$
	\mathcal{T}_\vartheta:
	\IR^\IN
	\rightarrow
	\IR^{\snorm{\vartheta}}
	:
	\y
	=
	(y_j)_{j\in \IN}
	\mapsto
	(y_j)_{j\in \vartheta}.
$
\end{lemma}

In combination with the sparse polynomial approximation results
from \cref{sec:sparsepolynomial} this result allows to derive
expression rates of tanh ANNs.

\begin{lemma}
Let \cref{assump:parametric_boundary} and 
\cref{assump:regularity_patches} be fulfilled
with $\boldsymbol{b} \in \ell^p(\IN)$ and $p\in (0,1)$.
Let $\Phi: \mathbb{U}\rightarrow \IR$ be as in 
\cref{eq:map_ff}.
Assume that $\boldsymbol{b}$ has a minimal 
monotone majorant in $\ell^p(\IN)$.
There exists a sequence of tanh neural network
$(\Psi^{(n)})_{n\in \IN}$
and $C>0$ such that
\begin{equation}
	\norm{
		\Phi
		-
		\Psi^{(n)}
		\circ
		\mathcal{T}_\vartheta
	}_{L^2(\mathbb{U})}
	\leq
	C
	n^{-\left(\frac{1}{p} - \frac{1}{2}\right)}
\end{equation}
with 
$
	\normalfont\text{width}
	\left(
		\Psi^{(n)}
	\right)
	=
	\mathcal{O}(n^2)
$,
$
	\normalfont\text{depth}
	\left(
		\Psi^{(n)}
	\right)
	=
	\mathcal{O}
	\left(
		\log_2(n)
	\right)
$, 
and $\vartheta = \{1,\dots,n\}$.
\end{lemma}

\begin{proof}
In analogy to \cref{thm:paramertric_reg_theta},
it follows from \cref{assump:parametric_boundary}, 
\cref{assump:regularity_patches}, and
\cref{lem:realandimaginarybpe} that the
map $\Phi:\mathbb{U} \mapsto \IR$ from \cref{eq:map_ff} is
$(\boldsymbol{b},p,\varepsilon)$-holomorphic and continuous.
Let $\widetilde{\boldsymbol{b}}$ be the monotone majorant of 
$\boldsymbol{b}$. Then, as discussed in the proof of
\cite[Corollary 8.2]{ABD22}, $\Phi$ as in 
is also $(\widetilde{\boldsymbol{b}},p,\varepsilon)$-holomorphic and continuous.
In analogy to \cref{eq:legendreexpansion}, we set
$\Phi
	=
	\sum_{\boldsymbol{\nu} \in \mathcal{F}} 
	c_{\boldsymbol{\nu}}^{(2)}
	P_{\boldsymbol{\nu}}^{(2)}
$, with $\left(c_{\boldsymbol{\nu}}^{(2)}\right)_{\boldsymbol{\nu} \in \mathcal{F}} $
being the Legendre coefficients of \cref{eq:map_ff}. 
Set
\begin{equation*}
	\widetilde{\Lambda}_n
	\coloneqq
	\left \{
		\boldsymbol{\nu} 
		=
		(\nu_k)_{k\in \IN}
		\in 
		\mathcal{F}:
		\:
		\prod_{k: \nu_k \neq 0}
		\left(
			\nu_k
			+
			1
		\right)
		\leq 
		n,
		\;
		\nu_k = 0,
		\;
		k>n
\right\}.
\end{equation*}
This set contains all \emph{anchored} sets of a size at most $n$
as pointed out in \cite[Proposition 2.18]{ABW22}. By anchored
set we mean a downward closed set satisfying for each $j\in \IN$
\begin{equation}
	\boldsymbol{e}_j
	\in
	\widetilde{\Lambda}_n
	\Rightarrow
	\{
	\boldsymbol{e}_1,
	\dots,
	\boldsymbol{e}_j
	\}
	\in 
	\widetilde{\Lambda}_n.
\end{equation}
Thus, recalling that \cref{thm:sparsepolynomial} implies that
$\left({c_{\boldsymbol{\nu}}^{(2)}} \right)_{\boldsymbol{\nu} \in \mathcal{F}}$ belong to $\ell^p_m(\mathcal{F})$ and according to \cite[Theorem 3.33]{ABW22} it holds
\begin{equation}
	\sqrt{
    \sum_{\boldsymbol{\nu} \notin 	\widetilde{\Lambda}_n}  
	\snorm{
		c_{\boldsymbol{\nu}}^{(2)}
	}^2
    }
	\leq
	C
	(n+1)^{-\left(\frac{1}{p} - \frac{1}{2}\right)}.
\end{equation}
Moreover, as a consequence of \cref{thm:sparsepolynomial},
for each $n\in \IN$ there exists a downward closed set $\Lambda_n$
of cardinality $n$ such that
\begin{equation}
	\norm{
		\Phi
		-
		\sum_{\boldsymbol{\nu} \in \Lambda_n} 
		c_{\boldsymbol{\nu}}^{(2)}
	    P_{\boldsymbol{\nu}}^{(2)}
	}_{L^{2}(\mathbb{U})}
	=
    \sqrt{
	\sum_{\boldsymbol{\nu} \notin \Lambda_n}
	\snorm{
		c_{\boldsymbol{\nu}}^{(2)}
	}^2
    }
	\leq 
	C(n+1)^{
	-
	\left(
	\frac{1}{p}-\frac{1}{2}
	\right)
	}, 
\end{equation}
and $\left(c_{\boldsymbol{\nu}}^{(2)}\right)_{\boldsymbol{\nu} \in \mathcal{F}} 
\in \ell^p_m(\mathcal{F})$.
Being $\widetilde{\boldsymbol{b}}$ monotonically increasing implies
that the set $\Lambda_n$ is anchored as well.
Set $\Lambda = \Lambda_n \cap \widetilde{\Lambda}_n$.
Then,
\begin{equation}
	\norm{
	\Phi
	-
	\sum_{\boldsymbol{\nu} \in \Lambda} 
	c_{\boldsymbol{\nu}}^{(2)}
	\Phi_{\boldsymbol{\nu},\delta}
	\circ
	\mathcal{T}_\vartheta
	}_{L^{2}(\mathbb{U})} 
	\leq
	(\spadesuit)
	+
	(\clubsuit),
\end{equation}
where
\begin{equation}
\begin{aligned}
	(\spadesuit)^2
	=
	\norm{
		\Phi
		-
		\sum_{\boldsymbol{\nu} \in \Lambda} 
		c_{\boldsymbol{\nu}}^{(2)}
	    P_{\boldsymbol{\nu}}^{(2)}
	}^2_{L^{2}(\mathbb{U})} 
	&
	=
	\sum_{\boldsymbol{\nu} \notin \Lambda} 
	\snorm{
		c_{\boldsymbol{\nu}}^{(2)}
	}^2
	\\
	&
	<
	\sum_{\boldsymbol{\nu} \notin \Lambda_n}
	\snorm{
		c_{\boldsymbol{\nu}}^{(2)}
	}^2
	+
	\sum_{\boldsymbol{\nu} \notin 	\widetilde{\Lambda}_n}  
	\snorm{
		c_{\boldsymbol{\nu}}^{(2)}
	}^2
	\\
	&
	\leq
	C
	(n+1)^{
	-2
	\left(
	\frac{1}{p}-\frac{1}{2}
	\right)}
\end{aligned}
\end{equation}
and
\begin{equation}
\begin{aligned}
	(\clubsuit)
	&
	=
	\norm{
		\sum_{{\boldsymbol{\nu}} \in \Lambda} 
	c_{\boldsymbol{\nu}}^{(2)}
	P_{\boldsymbol{\nu}}^{(2)}
	-
	\sum_{\boldsymbol{\nu} \in \Lambda} 
	c_{\boldsymbol{\nu}}^{(2)}
	\Phi_{\boldsymbol{\nu},\delta}
	\circ
	\mathcal{T}_\vartheta
	}_{L^{2}(\mathbb{U})} 
	\\
	&
	\leq
	\sum_{\boldsymbol{\nu} \in \Lambda} 
	\snorm{
		c_{\boldsymbol{\nu}}^{(2)}
	}
	\norm{
		P_{\boldsymbol{\nu}}^{(2)}
		-
		\Phi_{\boldsymbol{\nu},\delta}
		\circ
		\mathcal{T}_\vartheta
	}_{L^{\infty}(\mathbb{U})}.
\end{aligned}
\end{equation}
To conclude, we observe that for any $\boldsymbol{\nu} \in \Lambda\subset\widetilde{\Lambda}_n$,  
one has $\text{supp}(\boldsymbol{\nu}) = \{1,\dots,n \}$. Hence,
observing that
$\left(\snorm{c_{\boldsymbol{\nu}}^{(2)}} \right)_{\boldsymbol{\nu} \in \mathcal{F}}$ belongs to $\ell^p_m(\mathcal{F})\subset\ell^1_m(\mathcal{F})$,
\cref{lmm:approx_pol} with 
$\delta =(n+1)^{-\left(\frac{1}{p}-\frac{1}{2}\right)}$ 
and $\vartheta = \{1,\dots,n\}$ yields the final result.
Also, since any $\boldsymbol{\nu} \in \Lambda$ implies 
$\boldsymbol{\nu} \in \widetilde{\Lambda}_n$, we have
\begin{align}
m(\Lambda)=\max_{\boldsymbol{\nu}\in\Lambda}\|\boldsymbol{\nu}\|_1\leq n,
\end{align}
thus yielding the width and depth estimates stated in this result.
\end{proof}

\section{Conclusion and Outlook}\label{sec:concl}
In this work, we consider a class of boundary integral 
operators set on a collection of parametrically defined boundaries.
In this setting, we establish the holomorphic dependence of said
operators--as an element of
a Banach spaces of bounded linear operators--upon the parametric 
variables, i.e., we show that the parameter-to-operator map is holomorphic.
This result allows us to establish parametric holomorphy 
results for the solution of boundary integral equations involving the aforementioned 
operators, provided the invertibility of the operator for each parametric input, i.e.,
we show that the parameter-to-solution map is holomorphic.
Though we only consider BIOs in a $L^2$-based setting, 
this results enables parametric holomorphy analysis in three dimensional domains
defined on affinely parametrized polygonal/polyhedral domains.

To show the applicability of our parametric holomorphy results,
we explore the implications on
the boundary integral formulation of the sound-soft acoustic
wave scattering. In particular, herein we discussed the implications
in sparse polynomial approximation of the parameter-to-solution map,
reduced order modelling, in establishing parametric regularity of
the posterior distribution in Bayesian shape inversion, and in
the construction of efficient surrogates with artificial neural networks 
for the approximation of the far-field pattern.
Even though not thoroughly discussed, using the parametric regularity
and examples established in this article allows to prove dimension-independent
convergence rates for multilevel algorithms in forward and inverse acoustic
domain uncertainty quantification such as in \cite{DHJM2022}, where the domain
deformations were implemented using NURBS.

Further work comprises exploring the implications of 
the parametric holomorphy result for BIOs. A canonical example
is the analysis and implementation of the Galerkin proper orthogonal decomposition-ANN
as in \cite{HU18} for the sound-soft acoustic scattering problem.

\section*{Acknowledgement}
The authors acknowledge the support by the Deutsche Forschungsgemeinschaft (DFG, German Research Foundation) under Germany's Excellence Strategy
-- GZ 2047/1, Projekt-ID 390685813.

\appendix
\section{Auxilliary results}
\begin{lemma}\label{lem:realandimaginarybpe}
Let $f: \mathbb{U} \rightarrow \IC$ be
$(\boldsymbol{b},p,\varepsilon)$-holomorphic and continuous
for some $\varepsilon>0$. Then $\Re(f)$ and $\Im(f)$, i.e, 
real and imaginary part of $f$, are
$(\boldsymbol{b},p,\varepsilon)$-holomorphic and continuous.
\end{lemma}
\begin{proof}
For each $\y \in \mathbb{U}$ we define
\begin{equation}
	f_{\Re}
	(\y)
	\coloneqq
	\Re
	(f(\y))
	=
	\frac{
		f(\y)
		+
		\overline{
			f(\y)
		}
	}{2}
\end{equation}
and claim that $\y \mapsto f_{\Re}(\y)$ is 
$(\boldsymbol{b},p,\varepsilon)$-holomorphic and continuous.
To this end, for $\z \in \mathcal{O}_{\boldsymbol{\rho}}$, we define the complex extension
\begin{equation}\label{eq:extension_to_complex}
	f^{\IC}_{\Re}
	(\z)
	\coloneqq
	\frac{
		f(\z)
		+
		\overline{
			f(\overline{\z})
		}
	}{2},
\end{equation}
and claim that for each $\z \in \mathcal{O}_{\boldsymbol{\rho}}$ the complex derivative of $\z \mapsto f^{\IC}_{\Re}(\z)$ in direction $z_j$ is
\begin{equation}\label{eq:realpartderivative}
	\partial_{z_j}
	f^{\IC}_{\Re}
	(\z)
	=
	\frac{
		\partial_{z_j}
		f(\z)
		+
		\overline{
			\partial_{z_j}
			f(\overline{\z})
		}
	}{2},
	\quad
	{\z} \in \mathcal{O}_{\boldsymbol{\rho}}.
\end{equation}
Therein, $\partial_{z_j} f(\overline{\z})$
corresponds to the derivative of $f(\z)$
in the direction $z_j$ and evaluated at 
$\overline{\z} \in \mathcal{O}_{\boldsymbol{\rho}}$
(observe that if ${\z} \in \mathcal{O}_{\boldsymbol{\rho}}$,
then $\overline{\z} \in \mathcal{O}_{\boldsymbol{\rho}}$).
Now, we compute
\begin{equation}
\begin{aligned}
	\snorm{
	\frac{
		f^{\IC}_{\Re}({\z} +h \boldsymbol{e}_j)
		-
		f^{\IC}_{\Re}({\z})
		-
		\partial_{z_j}
		f^{\IC}_{\Re}
		(\z)
	}{
		h
	}
	}
	\leq
	&
	\frac{1}{2}
	\snorm{
		\frac{
			f({\z} +h \boldsymbol{e}_j)
			-
			f(\z)
			-
			\partial_{z_j}
		f(\z)
		}{
		h
		}
	}
	\\
	&
	+
	\frac{1}{2}
	\snorm{
		\frac{
			\overline{f(\overline{{\z} +h \boldsymbol{e}_j})}
			-
			\overline{f(\overline{\z})}
			-
			\overline{
				\partial_{z_j}
				f(\overline{\z})
			}
		}{
		h
		}
	}.
\end{aligned}
\end{equation}
On the one hand, for any $\bz \in \mathcal{O}_{\boldsymbol{\rho}}$
\begin{equation}
	\lim_{\snorm{h} \rightarrow 0^+}
	\snorm{
		\frac{
			f({\z} +h \boldsymbol{e}_j)
			-
			f(\z)
			-
			\partial_{z_j}
		f(\z)
		}{
		h
		}
	}
	=0.
\end{equation}
On the other hand, for any $\z \in \mathcal{O}_{\boldsymbol{\rho}}$\
\begin{equation}
	\lim_{\snorm{h} \rightarrow 0^+}
	\snorm{
		\frac{
			\overline{f(\overline{{\z} +h \boldsymbol{e}_j})}
			-
			\overline{f(\overline{\z})}
			-
			\overline{
				\partial_{z_j}
				f(\overline{\z})
			}
		}{
		h
		}
	}
	=
	\lim_{\snorm{h} \rightarrow 0^+}
	\snorm{
		\frac{
			{f(\overline{{\z} +h \boldsymbol{e}_j})}
			-
			{f(\overline{\z})}
			-
			{
				\partial_{z_j}
				f(\overline{\z})
			}
		}{
		h
		}
	}
	=
	0.
\end{equation}
Thus $\Re \left(f(\y)\right)$ is
$(\boldsymbol{b},p,\varepsilon)$-holomorphic.
Similarly, one can argue that $\Im \left(f(\y)\right)$ is 
$(\boldsymbol{b},p,\varepsilon)$-holomorphic and continuous.
\end{proof}

\bibliographystyle{siam}
\bibliography{ref}

\end{document}